\documentclass[a4paper,reqno,11pt]{amsart}

\usepackage[paperwidth=7in, paperheight=10in, margin=.875in]{geometry}
 \usepackage[backref,colorlinks,linkcolor=red,anchorcolor=green,citecolor=blue]{hyperref}
\usepackage{amsfonts,amssymb}
\usepackage{amsmath}
\usepackage{graphicx}
\usepackage{cite}
\usepackage{enumerate}
\numberwithin{equation}{section}

\newtheorem{theorem}{Theorem}[section] 
\newtheorem{proposition}[theorem]{Proposition} 
\newtheorem{definition}[theorem]{Definition} 
\newtheorem{lemma}[theorem]{Lemma} 
\newtheorem{corollary}[theorem]{Corollary} 
\theoremstyle{remark} 
\newtheorem{remark}[theorem]{Remark} 
\newtheorem{example}[theorem]{Example}

\DeclareMathOperator{\dive}{div}
\newcommand{\ri}{\rho_i}
\newcommand{\rj}{\rho_j}
\newcommand{\rri}{\bar{\rho}_i}
\newcommand{\rrj}{\bar{\rho}_j}
\newcommand{\ui}{u_i}
\newcommand{\uj}{u_j}
\newcommand{\uui}{\bar{u}_i}
\newcommand{\uuj}{\bar{u}_j}

\newcommand{\e}{\epsilon}	
\newcommand{\lb}{\lambda}
\newcommand{\T}{\mathbb{T}^3}

\newcommand{\R}{\mathbb{R}}

\def\XXint#1#2#3{{\setbox0=\hbox{$#1{#2#3}{\int}$} 
		\vcenter{\hbox{$#2#3$}}\kern-.5\wd0}}

\allowdisplaybreaks

\title[High friction limits of Euler--Navier--Stokes--Korteweg  multicomponent equations]{High friction limits  of Euler--Navier--Stokes--Korteweg equations for multicomponent models}

\author[G. Cianfarani Carnevale]{Giada Cianfarani Carnevale}
\address[Giada Cianfarani Carnevale]{Dipartimento di Ingegneria e Scienze dell'Informazione e Matematica, Universit\`a degli Studi dell'Aquila, Italy}
\email{giada.cianfaranicarnevale@graduate.univaq.it}

\author[C. Lattanzio]{Corrado Lattanzio}
\address[Corrado Lattanzio]{Dipartimento di Ingegneria e Scienze dell'Informazione e Matematica, Universit\`a degli Studi dell'Aquila, Italy} 
\email{corrado@univaq.it}

\begin{document}

\begin{abstract}
In this paper we analyze the high friction regime for   the Navier–Stokes–Korteweg equations for multicomponent systems.
According to the shape of the mixing and friction terms, we shall perform two limits:  the high friction limit toward an equilibrium system  for the limit densities and the \emph{barycentric}  velocity, and, after an appropriate time scaling, the diffusive relaxation toward parabolic, gradient flow equations for the limit densities.
 The rigorous justification of these limits 
is done by means of relative entropy techniques in the framework of weak, finite energy solutions of the relaxation models, rewritten in the enlarged formulation  in terms of the  \emph{drift velocity},  toward smooth solutions of the corresponding equilibrium dynamics. Finally,
since our estimates are uniform for small  viscosity, the results are also valid for the Euler--Korteweg multicomponent models, and the corresponding estimates can be obtained by sending the viscosity to zero.
\end{abstract}

\keywords{High friction limit, diffusive relaxation, Euler--Navier--Stokes--Korteweg equations, Stefan--Maxwell
systems, relative entropy method}

\subjclass[2010]{35L65,35B25}

\maketitle

\section{Introduction}\label{sec:0}
In this paper we study the high friction limit for the Navier-Stokes-Korteweg multicomponent systems \cite{WESSKRI,OR,HJT}, that is:
\begin{equation}\label{eq:nsk}
\left \{	\begin{aligned} 
		&\partial_t\rho_i + \dive(\rho_i u_i ) = 0 \\ 
		&\partial_t(\rho_i u_i ) +\dive (\rho_i u_i \otimes u_i) - 2\nu\dive (\mu_L(\rho_i) D(u_i)) - \nu \nabla( \lambda_L(\rho_i)\dive u_i)+ \nabla \rho_i^{\gamma}\\
		&\ = \rho_i \nabla \left(k(\rho_i) \Delta \ri + \frac{1}{2}k'(\rho_i)|\nabla \rho_i|^2 \right)   -\frac{1}{\e} \sum_{j=1}^{n}b_{i,j}\rho_i \rho_j (u_i-u_j) - \frac{M_i}{\e} \ri \ui ,
	\end{aligned} \right.
\end{equation}
where $i=1,\cdots,n$, $t>0$, $x \in \T$, the $n$-dimensional torus, $\rho_i$ are the particles' density, $\ui$ their velocities (and, accordingly,  $m_i= \ri \ui$  their momenta), $\nu \geq 0$ is the viscosity coefficient, $M_i\geq 0$, and the singular coefficient ${1}/{\e}$ in front of the  terms 
\begin{equation*}
   -\sum_{j=1}^{n}b_{i,j}\ri \rj(\ui-\uj) - M_i\ri\ui
\end{equation*}
is responsible for the high friction regime $\e\to 0$. The term $k(\ri) \geq 0$ stands for the capillarity coefficient and usually has the form of the power law, while
 $D(u_i)$ is the symmetric part of the gradient $\nabla u_i$, and the Lam\'e coefficients $\mu_L(\rho_i)$ and $\lb_L(\rho_i)$ satisfy, for every particles' density,
\begin{equation}\label{lamee}
	\mu_L(\rho_i) \geq 0; \qquad \frac{2}{n} \mu_L(\rho_i) + \lb_L(\ri) \geq 0.
\end{equation}
We consider, for simplicity, $\gamma$-law pressure, that is  $p(\rho_i) = \ri^{\gamma}$ with $\gamma >1$ and the corresponding internal energy $h(\ri)$ is given by:
\begin{equation*}
	h(\ri)= \frac{\ri^{\gamma}}{\gamma- 1};
\end{equation*}
more general hypotheses for the monotone function $p(\rho)$ can be considered. 
Moreover, as usual for this kind of models, we rewrite the pressure and  the Korteweg terms by introducing  the stress tensors $T_i$ in the following way:
\begin{equation*}
	- \nabla \rho_i^{\gamma} +  \rho_i \nabla \left(k(\rho_i) \Delta \ri + \frac{1}{2}k'(\rho_i)|\nabla \rho_i|^2 \right) = \dive T_i.
\end{equation*}

The novelty of the multicomponent model presented here is the presence  of many different particles, and as a consequence, the presence of the  interaction term 
\begin{equation*}
	\sum_{j=1}^{n}b_{i,j}\ri \rj(\ui-\uj), 
\end{equation*}
which stands for the momentum production rate due to diffusive mixing. Here, the
 nonnegative matrix $b_{i,j}$ models the interaction of the $i$-th and $j$-th components, with a strength  weighted by $\e$.
In addition, as already proposed in \cite{OR},
we also consider a diagonal term $M_i\ri\ui$,  $M_i\geq 0$, accounting for a (classical) friction term.  
The term modeling the mixing shall agree 
with   the conservation of the total momentum, while the latter friction term comes as a body force,  in accordance with the single component case \cite{OR}. Hence, we shall require
\begin{equation}\label{ip1}
		\sum_{i,j=1}^{n}b_{i,j}\ri \rj(\ui-\uj) = 0,
	\end{equation}
	which is true provided the matrix $\{b_{i,j}\}_{i,j=1}^{n}$ is symmetric.
	Moreover, this condition and  assumptions  $b_{i,j} \geq 0$ and $M_i \geq 0$
 guarantee the dissipative structure of the model. Indeed, introducing the vectors
${\hat \rho}  = (\rho_1, \cdots, \rho_n)$ and $\hat m=(m_1, \cdots, m_n)$,   the total mechanical energy associated to $\eqref{eq:nsk}$ is given by
\begin{equation}\label{eq:energiatot1}
\eta_{tot}(\hat \rho,\hat m,\nabla \hat \rho) := \sum_{i=1}^{n} \eta(\ri,m_i,\nabla \ri) := \sum_{i=1}^n \left (\frac{1}{2} \ri |\ui|^2 + \frac{\ri^{\gamma}}{\gamma -1 } + \frac{1}{2}k(\ri)|\nabla \ri|^2 \right )
\end{equation}
and it verifies the following relation:
\begin{align*}
   & \frac{d}{dt} \sum_{i=1}^n\int_{\T} \left (\frac{1}{2} \ri |\ui|^2 + \frac{\ri^{\gamma}}{\gamma -1 } + \frac{1}{2}k(\ri)|\nabla \ri|^2\right ) \;dx + \frac{1}{2\e} \sum_{i,j=1}^n \int_{\T} \frac{1}{2} b_{i,j} \ri \rj |\ui-\uj|^2 dx \\
   &\ + \frac{1}{\e} \sum_{i=1}^n  \int_{\T} M_i \ri |\ui|^2 \;dx =0,
\end{align*}
being 
\begin{equation}\label{eq:newdiss}
    \sum_{i,j=1}^{n}b_{i,j} \ri \rj (\ui-\uj) \cdot \ui  =  \frac{1}{2} \sum_{i,j=1}^{n}b_{i,j} \ri \rj |\ui-\uj|^2 \geq 0.
\end{equation}
Concerning the aforementioned mixing term, we collect here below all required assumptions, referred to   as
\emph{Stefan-Maxwell Ansatz} \cite{OR}. We introduce the  $\R^n \times \R^n$ matrix $B$ defined as 
\begin{equation}\label{eq:ourdeftilde}
     B_{i,j}: = b_{i,j}\rho_i \rho_j
\end{equation}
and we assume 
\begin{equation}\label{assumption}
	b_{i,j}= b_{j,i}  \geq 0\ \hbox{for any}\ i,j=1, \cdots,n; \quad  b_{i,i} =- \sum_{j=1, j \neq i}^{}b_{i,j}  .
\end{equation}
We observe that,  thanks to  \eqref{assumption}, we readily obtain
\begin{equation*}
    \sum_{j=1}^{n} b_{i,j} \rho_i \rho_j (\ui-\uj) = \sum_{j=1}^{n} B_{i,j} (\ui-\uj) = -\sum_{j=1}^{n} B_{i,j} \uj,
\end{equation*}
because, for any $j=1,\dots,n$,
\begin{equation*}
    \sum_{j=1}^{n} B_{i,j} \ui = 0.  
\end{equation*}
For later convenience, we thus introduce the matrix $\tau := -B$ and  $u$ solution of  \eqref{ip1} is equivalent to $ u \in Ker(\tau)$. Moreover, since $(1, \cdots, 1)\in Ker(\tau)$, then dim$Ker(\tau)\geq 1$. The study of this kernel will be crucial in the forthcoming discussions and we shall consider appropriate conditions for it in the sequel. 

Following \cite{OR,HJT}, the present investigation is confined in the analysis of the 
behavior of weak, finite energy solutions of such systems in the high friction regime, and not on their existence; 
for the latter,  for single component cases,   see \cite{AM1,AM2,AS,AS2} and the reference therein.
Hence, in the present paper we are thus dealing with the rigorous justification of relaxation limits \cite{CP,bou,CLL,yo}, in particular using relative entropy approach \cite{thanos}. 
As we shall point out later, when the diagonal term $M_i\rho_i u_i$ is present in the model, we shall obtain a nontrivial equilibrium dynamic after a time scaling,  leading to a diffusive relaxation limit.
These kind of limits have been addressed in different frameworks and with using many tools. In particular,  we refer to \cite{DM04} and the reference therein for the results concerning weak solutions and compactness arguments; in this context, see also \cite{ACCLS21} for a recent study concerning the relaxation limit for weak, finite energy solutions to the Quantum Navier--Stokes--Poisson system toward weak solutions of Quantum drift--diffusion equation.  
 Still in the context of diffusive relaxations, many other (diffusive) limits have been addressed using relative entropy tools; among others, see \cite{CG,Bia19,FT19,Carrillo}.
Finally, referring in particular to  
 multicomponent models, we shall also mention here  the relevant examples of relaxation limits  for (bipolar) Euler-Poisson  equations describing electrons and positively charged ions in plasmas or semiconductors \cite{JU2,Lat00,AT22}.

Our analysis takes advantage of relative entropy techniques in the framework of finite energy weak solutions of the enlarged formulation of $\eqref{eq:nsk}$, namely we rewrite the latter in terms of the \emph{drift velocity} $v_i$ \cite{BL}:
\begin{equation*}
	v_i= \frac{\nabla \mu(\ri)}{\ri},
\end{equation*} 
where $\mu(\ri)$ satisfies $\mu'(\ri)= \sqrt{\ri k(\ri)}$. In this way it is possible to obtain the following augmented formulation of $\eqref{eq:nsk}$:
\begin{equation}\label{eq:nskv}
\left \{	\begin{aligned} 
		&\partial_t\rho_i + \dive(\rho_i u_i ) = 0 
		\\ 
		&\partial_t(\rho_i u_i ) +\dive (\rho_i u_i \otimes u_i) - 2\nu\dive (\mu_L(\rho_i) D(u_i)) - \nu \nabla( \lambda_L(\rho_i)\dive u_i)+ \nabla \rho_i^{\gamma} 
		\\
		&\ = \dive(\mu(\ri)\nabla v_i) + \frac{1}{2} \nabla (\lambda(\ri) \dive v_i)  -\frac{1}{\e} \sum_{j=1}^{n}b_{i,j}\rho_i \rho_j (u_i-u_j) - \frac{M_i}{\e} \ri \ui 
		\\
		&\partial_t(\rho_i v_i ) +\dive (\rho_i v_i \otimes u_i) + \dive(\mu(\ri)^t\nabla u_i) + \frac{1}{2} \nabla (\lambda(\ri) \dive u_i) = 0,
	\end{aligned} \right.
\end{equation}
where $\lb(\ri) = 2(\mu'(\ri)\ri - \mu(\ri))$. 
Using the relation between $\mu(\ri)$ and $\lambda(\ri)$ we define:
\begin{equation*}
	\begin{aligned}
		&\dive S_i =  \dive(\mu(\ri)\nabla v_i) + \frac{1}{2} \nabla (\lambda(\ri) \dive v_i),\\
		& \dive K_i = \dive(\mu(\ri)^t\nabla u_i) + \frac{1}{2} \nabla (\lambda(\ri) \dive u_i),
	\end{aligned}
\end{equation*}
and we readily obtain
\begin{equation}\label{propeSK}
	\begin{aligned}
		\int_{\T} \dive S_i\cdot \ui \;dx &= - \int_{\T} \mu(\ri) \nabla v_i : \nabla \ui \;dx - \frac{1}{2}\int_{\T} \lambda(\ri) \dive v_i \dive \ui \;dx \\ &= 
		 - \int_{\T} \mu(\ri) {}^t\nabla \ui: \nabla v_i \;dx - \frac{1}{2}\int_{\T} \lambda(\ri) \dive \ui \dive v_i \;dx \\
		 &= \int_{\T} \dive K_i\cdot v_i \;dx. 
	\end{aligned}
\end{equation}
The above identity comes from the symmetry of $\nabla v_i$, which implies ${}^t\nabla \ui: \nabla v_i = \nabla v_i: \nabla \ui$. 

Hence, in contrast with the one used in \cite{OR,HJT}, as already done in \cite{CL},
the strategy we shall use here is  to estimate the following relative entropy, expressed in terms of the momenta $m_i=\ri \ui$ and the \emph{drift momenta} $J_i= \ri v_i = \nabla \mu(\ri)$.
To this end, using again the  notation 
$ \hat \rho  = (\rho_1, \cdots, \rho_n)$, $\hat m=(m_1, \cdots, m_n)$, $\hat J(x,t)=(J_1, \cdots, J_n)$, and the relative entropy 
is then given by
\begin{equation*}
\begin{aligned}
&  \eta_{tot}(\hat \rho,\hat m,\hat J |  \bar{\hat \rho}, \bar{\hat m}, \bar{\hat{J}})(t) := \sum_{i=1}^{n} \eta(\ri,\ui,J_i | \rri, \uui, \bar{J}_i) := \sum_{i=1}^{n} \Big(\eta(\ri,\ui,J_i) - \eta(\rri,\uui,\bar{J}_i) \\
& \ - \frac{\partial \eta}{\partial \ri}(\rri,\uui,\bar{J}_i)(\ri -\rri) -  \frac{\partial \eta}{\partial m_i}(\rri,\uui,\bar{J}_i)(m_i - \bar{m}_i) - \frac{\partial \eta}{\partial J_i}(\rri,\uui,\bar{J}_i)(J_i - \bar{J}_i)\Big),
\end{aligned}
\end{equation*}
where we recall $\eta_{tot}(\hat \rho,\hat m,\hat J)$ is defined as
\begin{equation}\label{etatot}
\eta_{tot}(\hat \rho,\hat m,\hat J) = \sum_{i=1}^{n} \eta(\ri,m_i,J_i) = \sum_{i=1}^{n} \left (\frac{1}{2} \frac{|m_i|^2}{\ri}+ \frac{1}{2} \frac{|J_i|^2}{\ri} + h(\ri)\right ),
\end{equation}
 that is, \eqref{eq:energiatot1} after we have introduced the variables $J_i$.
 As a consequence, the quadratic expression we shall estimate here becomes the integral of $\eta_{tot}$ in $dx$, namely:
\begin{equation}\label{sigmatot}
\begin{aligned}
\Sigma_{tot}(\hat \rho,\hat m,\hat J |  \bar{\hat \rho}, \bar{\hat m}, \bar{\hat{J}})(t) & : = \int_{\T} \sum_{i=1}^{n}\eta(\ri,\ui,J_i | \rri, \uui, \bar{J}_i) \;dx \\
& = \int_{\T} \sum_{i=1}^{n}\left( \frac{1}{2}\rho_i \left|\frac{m_i}{\ri} - \frac{\bar{m}_i}{\rri} \right|^2 +  \frac{1}{2} \rho_i \left|\frac{J_i}{\ri} - \frac{\bar{J}_i}{\rri} \right|^2 + h(\ri |  \rri) \right)dx.
\end{aligned}
\end{equation}

The behavior of the system under investigation here in the high friction regime $\e\to 0$ strongly relies on the presence or not of the diagonal term  $M_i\rho_i u_i\geq 0$. Indeed, when this term is present,  the high friction regime, after an appropriate time scaling, is given by the following parabolic equation: 
\begin{equation}\label{eq:diff-limit}
\partial_t \ri = \dive_x \left( \ri \nabla_x \left(h'(\ri) +  k(\ri)\Delta\ri + \frac{1}{2}k'(\ri)|\nabla \ri|^2 \right) \right), \text{ for any } i=1,\cdots,n,
\end{equation} 
both for Euler--Korteweg and Navier--Stokes--Korteweg systems, as one can   check by performing the classical Hilbert expansion; see Section \ref{sec:1} for details. Concerning the Lam\'e coefficients, besides the natural condition \eqref{lamee} needed to guarantee the dissipative nature of the viscosity terms, in this case we shall assume  only appropriate uniform integrability conditions, 
without a precise connection with the capillarity coefficient $k(\rho)$, as  it is usually needed in the analysis of these models, as these terms will come as an higher order errors in the limit.
In other words, in the case  $M_i >0$ we are dealing with a diffusive relaxation, as the hyperbolic system  converges  toward parabolic equilibrium systems \eqref{eq:diff-limit} in the  diffusive scaling.
Starting from \cite{LT}, where the Authors discussed  the case of  the diffusive relaxation of the Euler system with friction toward the porous media equation  in the single component case, 
a general framework for the relative entropy calculation and the analysis of the diffusive limits have been presented  in \cite{GLT,LT2}; see also \cite{GT} for the non--monotone pressure cases. The models under investigation here, without viscosity, mixing, and friction terms, are included in the framework of abstract Euler flows generated by the first variation of an energy functional $\mathcal{E}(\rho)$ introduced in these papers: 
\begin{equation*}
\left\{\begin{aligned}
		& \partial_t \rho + \dive (\rho u) =0\\
		&  \rho \partial_tu  + \rho u \cdot \nabla u =  - \rho \nabla \frac{\delta \mathcal{E}}{\delta \rho}.
		\end{aligned}\right.
\end{equation*}
Specifically,  system $\eqref{eq:nsk}$ is obtained for the following particular choice for $\mathcal{E}(\rho)$:
\begin{equation}\label{eq:potintro}
\mathcal{E}(\rho) =  \int \sum_{i=1}^n \left (h(\ri) + \frac12 k(\ri) |\nabla \ri|^2 \right )dx.
\end{equation}
These results have been improved following the enlarged formulation of \cite{BL}, and in particular the diffusive relaxation limit for the single component case is treated  in \cite{CL}. 
As already mentioned above, in the present work  we shall then adopt the same strategy of \cite{CL} for the   multicomponent system, thus improving the results contained in \cite{OR}, as we include  viscosity and capillarity effects in our model.
It is worth observing that the stability result we obtain in the Navier--Stokes--Korteweg case is uniform for $0<\nu\ll 1$, and thus we  obtain in particular the corresponding estimate also for the Euler--Korteweg model, by taking the limit 
  $\nu \rightarrow 0^+$ in the former inequality.

The case $M_i=0$, treated in \cite{HJT}, is significantly different, as we do not perform any diffusive time scaling, and the equilibrium system in the case $\nu=0$ is not parabolic, namely, there is no emergence of diffusive behavior in the high friction regime. Indeed, following the argument of this paper, after the Hilbert expansion of \eqref{eq:nskv} in Section \ref{sec:1}, we recover the following hyperbolic system at equilibrium, where the we introduce the notation $\bar{\cdot}$ to indicate the equilibrium dynamics for the system:
\begin{equation}\label{eq:limitm0intr}
    \left \{ \begin{aligned}
		&\partial_t\rri + \dive(\rri \bar u )=0 \\
		& \partial_t(\bar \rho \bar u) + \dive(\bar\rho \bar u \otimes \bar{u}) + \sum_{i=1}^{n}\nabla p(\rri) - 2\nu \sum_{i=1}^n\dive (\mu_L(\rri) D(\bar{u})) - \nu \sum_{i=1}^n\nabla( \lambda_L(\rri)\dive (\bar{u})) \\
		&\ = \sum_{i=1}^{n}\dive(\mu(\rri) \nabla \bar{v}_i) + \sum_{i=1}^{n}\frac{1}{2} \nabla(\lambda(\rri) \dive \bar{v}_i)   \\
		& \partial_t(\rri \bar{v}_i) + \dive(\rri \bar{v}_i \otimes \bar{u}) + \dive (\mu(\rri) {}^t\nabla(\bar{u})) + \frac{1}{2}\nabla(\lambda(\rri) \dive(\bar{u})) = 0.
	\end{aligned}
	\right.
\end{equation}
We recall here that  $\mu_L(\ri), \lambda_L(\ri) \geq 0$ stands for  the Lam\'e coefficients, while  $\mu(\ri),\lambda(\ri)$ are the capillarity ones used in the augmented formulation. Finally,  $\bar{u}$ stands for   the barycentric  velocity defined by the following relations:
\begin{equation*}
    \bar{\rho} = \sum_{i=1}^n \bar\rho_i, \ \bar{\rho}\bar{u} = \sum_{i=1}^n \bar\rho_i\bar u_i.
\end{equation*}
At it is manifest, contrarily to the previous analysis, the equilibrium dynamics is described by the group velocity $\bar{u}$ for each particle density $\rri$. This is a consequence of the absence of the diagonal term $M_i \ri\ui$  in the momentum equation and therefore  the interaction mixing term 
\begin{equation*}
    \sum_{j=1}^n b_{i,j}\ri\rj|\ui-\uj|
\end{equation*}
alone plays the role of an alignment term for the velocities $\ui$. 

Another important difference with the previous case  concerns the presence of the viscosity term also in the equilibrium system  \eqref{eq:limitm0intr}, which, in this scaling, are not higher order errors (see again Section \ref{sec:1} for details), and we are thus forced to manage them differently in the relative entropy calculation. 
For this, as in \cite{BL}, we shall make the following particular choices for the Lam\'e coefficients: 
 $\mu_L(\ri)=\mu(\ri)$ and $\lambda_L(\ri)=\lambda(\ri)$, which 
 will be crucial to estimate the viscosity terms by means of the relative entropy and thus obtain the  desired stability for the limit under investigation. Besides the presence of these extra terms in our model, thanks to the enlarged formulation adopted here, we are also able to consider  more general capillarity coefficients, and thus slightly generalize the results proved in \cite{HJT} also under this point of view.
Indeed, as already pointed out for the previous case, our stability estimate is uniform for $0<\nu\ll 1$ and it is consistent with the one obtained in the  Euler--Korteweg case in \cite{HJT} in the limit $\nu\to 0^+$.

The remaining part of this paper is organized as follows. In Section \ref{sec:1} we present the Hilbert expansion of system \eqref{eq:nskv} in both cases $M_i=0$ and $M_i>0$, underlying in particular  the differences   due to the presence/absence of the diagonal term $M_i\rho_i u_i $. In Section \ref{sec:4} we investigate the case of \eqref{eq:nskv} when $M_i>0$. As already noted before, the equilibrium is given by the parabolic equation for each particle density $\ri$, namely the gradient flow \eqref{eq:diff-limit} 
 written in terms of the drift velocity $v_i$,
both for Euler--Korteweg and Navier--Stokes--Korteweg, consistently with \cite{CL} for the single component case. We prove a stability estimate between weak, entropic solutions of \eqref{eq:nsk} and strong solutions of the equilibrium system,
exploiting the multicomponent version of the relative entropy inequality. The last section  is devoted to the complementary case $M_i=0$: the structure is the same of the one of Section \ref{sec:4}, however this time no diffusive scaling is present and the equilibrium is given by the  system satisfied by the densities and the  barycentric velocity. 
\section{Hilbert expansion}\label{sec:1}
In this section we shall perform an Hilbert   expansion for solution of \eqref{eq:nskv}
as a first step in the analysis of our high friction limit. To this end, we have to study the solvability properties of the linear system depending on $M_i \geq 0$ given by:
\begin{equation}\label{nonhom}
	-\sum_{j=1}^{n} b_{i,j} \ri \rj (\ui-\uj) - M_i \ri \ui = d_i, \quad d_i \in \R^3, \qquad i=1, \cdots, n, 
\end{equation}
and the associated homogeneous system
\begin{equation}\label{hom}
	\sum_{j=1}^{n} b_{i,j} \ri \rj (\ui-\uj) + M_i \ri \ui = 0, \qquad i=1, \cdots, n.
\end{equation}
The Authors in \cite{HJT} make the following hypothesis for the homogeneous system with $M_i=0$:
\begin{equation}\label{homT}
	\sum_{j=1}^{n} B_{i,j} (\ui-\uj) = 0, \qquad i=1, \cdots, n,
\end{equation}
where  $B_{i,j}$ is defined in \eqref{eq:ourdeftilde}.

\begin{itemize}
    \item[A1)]Let $\{b_{i,j}\}_{i,j=1}^{n}$ be a symmetric matrix such that $b_{i,j} \geq 0$ for $i\neq j$. 
    For any $\ri >0 $, $i=1, \cdots n$, system $\eqref{homT}$ has the one dimensional null space span$\{\bf{1}\}$, where $\textbf{1} = (1,\cdots, 1) \in \R^n$. 
\end{itemize}
Let us emphasize that, if  $b_{i,j}>0$ for any $i,j$, the above  hypothesis is automatically satisfied, as it is manifest from \eqref{eq:newdiss}.

The following result is proved  in \cite[Lemma 1]{HJT} and it will be used in the Hilbert expansion to provide a semi--explicit solution to system $\eqref{nonhom}$ in the case $M_i=0$ for any $i=1,\dots,n$. It is worth observing that such system is independent on the specific component of the vectors $\ui\in\R^3$ and therefore its resolution and the subsequent computations of this section are done component-wise (when not otherwise specified).
\begin{lemma}[Lemma 1 in \cite{HJT}]\label{lem:1}
	Let $d_1, \cdots, d_n \in \R^3$ satisfy $\sum_{i=1}^{n} d_i =0$, $\ri >0$ for $i=1, \cdots, n$ and assume condition A1) above.
Then the system
	\begin{equation*}
		- \sum_{j=1}^{n} B_{i,j}(u_i -u_j) = d_i \quad i=1, \cdots n, \quad \text{subject to } \sum_{i=1}^{n} \ri u_i=0
	\end{equation*}
	has the unique solution
	\begin{equation*}
		\ri u_i = - \sum_{j,k=1}^{n} \left( \delta_{i,j} \ri - \frac{\ri \rj}{\rho} \right) (\tau)^{-1}_{j,k}d_k, \quad \rho_n u_n = - \sum_{j=1}^{n} \rj u_j,
	\end{equation*}
	where $i=1, \cdots, n$ $\rho=\sum_{i=1}^{n} \ri >0$ and $(\tau)^{-1}_{i,j} \in \R^{(n-1)\times(n-1)}$ is the inverse of a regular submatrix, obtained from reordering the matrix $\tau_{i,j} \in \R^{n \times n}$ of  rank $n-1$ with coefficients
	\begin{equation}\label{eq:deftau}
		\tau_{i,j} =  \delta_{i,j} \sum_{k=1}^{n} B_{i,k} - B_{i,j}, \quad i,j=1,\cdots,n.
	\end{equation}
\end{lemma}
It is worth observing that, in view of the assumption \eqref{assumption}, the definition \eqref{eq:deftau} of $\tau$   in Lemma \ref{lem:1} coincides with the matrix $\tau = -B$, for $B$ defined in \eqref{eq:ourdeftilde}.

Now, we pass to the study of the solvability of \eqref{nonhom}, when the coefficients $M_i$ do not all vanish. Therefore, let us  identify the null-space of system $\eqref{hom}$ and, to this end, we prove the following result concerning the matrix 
\begin{equation}\label{eq:deftauM}
    \tau^M := \tau + M;\quad M =\mathrm{diag}(M_1\rho_1,\dots,M_n\rho_n).
\end{equation}
\begin{proposition}\label{prop22}
Let $B$ be any  positive semi-definite matrix as  in \eqref{eq:ourdeftilde} and let $\tau = -B$. Moreover, let $M$ be any diagonal, positive semi-definite matrix $M$ defined as in \eqref{eq:deftauM}. Then,  
the following property holds:
\begin{equation}
    Ker(M) \cap Ker(\tau) = Ker(\tau^M).
\end{equation}
\end{proposition}
\begin{proof}
($\subseteq$) 

Let  us suppose $\textbf{u} \in  Ker(M) \cap Ker(\tau)$, then
\begin{equation*}
    \tau^M \cdot \textbf{u} =(\tau + M) \cdot \textbf{u} = \tau \cdot \textbf{u} + M \cdot \textbf{u} = \bf{0},
\end{equation*}
therefore $\textbf{u} \in Ker(\tau^M)$. 

($\supseteq$) Let $\textbf{u} \in  Ker(\tau^M)$ so that $\tau \cdot \textbf{u} + M \cdot \textbf{u} = \bf{0}$. Multiplying this equaility by $\textbf{u}$, using the definition of $\tau_{i,j}$ and the property in \eqref{assumption},  we can then write
\begin{equation*}
    \sum_{i,j=1}^{n}b_{i,j} \ri \rj (\ui-\uj) \cdot \ui + \sum_{i=1}^{n}M_i \ri |\ui|^2 = 0.
\end{equation*}
Since $b_{i,j}=b_{j,i}$, this is equivalent to require
\begin{equation*}
   \frac{1}{2}\sum_{i,j=1}^{n}b_{i,j} \ri \rj |\ui-\uj|^2 + \sum_{i=1}^{n}M_i \ri |\ui|^2 = 0, 
\end{equation*}
and, being in addition $b_{i,j}$ and $M_i$ non negative for any $i$ and $j$, we further deduce 
\begin{equation*}
     \sum_{i=1}^{n}M_i \ri |\ui|^2 =  0.
\end{equation*}
Hence, for any index for which  $M_i\rho_i > 0$  we obtain  $u_i = 0$, that is,  since $M$ is diagonal, $M \cdot \textbf{u} = \bf{0}$. This implies $\textbf{u} \in Ker(M)$ and consequently $\textbf{u} \in Ker(\tau)$, being $- \tau \cdot \textbf{u} = M \cdot \textbf{u} =0$,  and  proof is complete.
\end{proof}

This simple observation helps us to analyze the kernel of the matrix $\tau^M$ if $M\neq 0$. However, even if the dim$(Ker(M) \cap Ker(\tau)) = 1 = Ker(\tau^M)$ and thus $Rk(\tau^M) = n-1$, we can not follow the proof of Lemma \ref{lem:1} to find the unique solution of the non homogeneous system \eqref{nonhom}. Due to the presence of the drug force in our model, the total momentum is not conserved, and we do not have in general additional information about its solutions. On the other side, 
when $Ker(\tau)\cap Ker(M) = \{ \textbf{0} \}$ the matrix $\tau^M$ is invertible and system \eqref{nonhom} has an unique solution given by
\begin{equation*}
    u =- (\tau^M)^{-1} d,
\end{equation*}
as in the case of our framework A2) below. Further  general conditions which guarantee the invertibility of $\tau^M$ could be considered as well, but one needs  specific information about the matrices $\tau$ and $M$. For this,  we analyze here below some simple examples to show that
it is not possible, in general, to get information on the increase of the rank of $\tau + M$ by looking only to the rank of $\tau$ and the rank of $M$, but one needs to know information on the structure of their null spaces, and not only on their dimensions.

\begin{example}
Let us consider a matrix ${\tau} \in \R^{4\times4}$ defined as in \eqref{eq:deftau} such that $Rk({\tau})=2$ and $\ri = \rj =1$, for every $i,j=1, \cdots, 4$; for instance:
\begin{equation*}
{\tau} =  \begin{pmatrix} 
1 & 0 & -1 & 0 \\
0 & 1 & 0 &-1 \\
-1 & 0 & 1 & 0 \\
0 &-1 & 0 &1 \\
\end{pmatrix}.
\end{equation*}
Then 
\begin{equation*}
     Ker({\tau})= \text{span}\{(1,0,1,0), (0,1,0,1)\}.
\end{equation*}

Now, for $M$ defined as follows
\begin{equation*}
M = \begin{pmatrix} 
m_1 & 0 & 0 & 0 \\
0 & 0 & 0 &0 \\
0 & 0 & 0 & 0 \\
0 &0 & 0 & 0 \\
\end{pmatrix},
\end{equation*}
with $m_1>0$, clearly we have 
\begin{equation*}
    Ker(M)= \text{span}\{(0,1,0,0), (0,0,1,0),(0,0,0,1)\}.
\end{equation*}
In this case the resulting matrix $\tau^M$ is such that $Rk(\tau^M) = 3 > Rk({\tau})=2$. Indeed, the solution of the 
homogeneous system $(\tau^M) \cdot \textbf{u} = 0$ is 
\begin{equation*}
    \{(0,t,0,t); t \in \R \} = Ker({\tau}) \cap Ker(M) = 
Ker(\tau^M).
\end{equation*}
Uusing the result of Proposition \ref{prop22}, we readily obtain this relation, being  
\begin{equation*}
    \{(0,t,0,t); t \in \R \} = Ker({\tau}) \cap Ker(M) = 
Ker(\tau^M).
\end{equation*}

Let us now consider two $4\times4$ matrices $M_1$,$M_2$ with rank 2 in two possible configurations, that is
\begin{equation*}
M_1 = \begin{pmatrix} 
m_1 & 0 & 0 & 0 \\
0 & m_2 & 0 &0 \\
0 & 0 & 0 & 0 \\
0 &0 & 0 & 0 \\
\end{pmatrix},
\quad
M_2 = \begin{pmatrix} 
m_1 & 0 & 0 & 0 \\
0 & 0 & 0 &0 \\
0 & 0 & m_3 & 0 \\
0 &0 & 0 & 0 \\
\end{pmatrix},
\end{equation*}
such that $m_1,m_2,m_3 >0$. The  solutions of the corresponding systems $({\tau} + M_i)\textbf{u}=0$ are discussed here below. In the first case,
\begin{equation*}
    Ker(M_1)= \text{span}\{(0,0,1,0),(0,0,0,1)\}.
\end{equation*}

From a direct inspection of the system $({\tau} + M_1)\textbf{u}=0$ we obtain only the trivial solution, that is $\tau+M_1$ is invertible.


Again, we remark   that $Ker({\tau}) \cap Ker(M_1) = \{ \textbf{0} \}$, and $\tau^M$ is invertible from Proposition \ref{prop22}.

In the second case, 
\begin{equation*}
    Ker(M_2)= \text{span}\{(0,1,0,0), (0,0,0,1)\}
\end{equation*}
the null space of ${\tau} +M_2$ is given by
span$\{(0,1,0,1)\}$, i.e., ${\tau}+M_2$  has rank 3. As before, 
\begin{equation*}
    \{(0,t,0,t); t \in \R \} = Ker({\tau}) \cap Ker(M_2) = 
Ker(\tau^M).
\end{equation*}

\end{example}

As already mentioned above, we present a possible framework for which $\tau^M$ becomes invertible. In analogy to condition A1) above, we make the following hypothesis; however, see also Remark \ref{rem:solOR}.
\begin{itemize}
    \item[A2)] Let $\{b_{i,j}\}_{i,j=1}^{n}$ be such that 
    hypothesis A1) holds and assume 
    there exists at least one index $\bar{i} \in \{1, \cdots, n\}$ such that $M_{\bar{i}}\rho_{\bar{i}}>0$.
\end{itemize}

\begin{corollary}\label{cor:A2}
If A2) holds, then  $\tau^M$ is invertible.
\end{corollary}
\begin{proof}
If $\textbf{u} \in Ker(\tau^M)$ then, in view of Proposition \ref{prop22}, $\textbf{u} \in Ker(\tau)\cap Ker(M)$ and, 
in particular, following the proof of this result,  we conclude
\begin{equation*}
  \sum_{i=1}^{n}M_i \ri |\ui|^2 =  0.
\end{equation*}
From A2) we know that there exists at least one index $\bar{i}$ such that $M_{\bar{i}}\rho_{\bar{i}} >0$ and, correspondingly, from the above equality we conclude    $u_{\bar{i}} = 0$. 
 Finally, since  $Ker(\tau)$ reduces to span$\{\mathbf{1}\}$ in view of condition A1), we conclude  $\textbf{u}  = (0, \cdots, 0)$ is the unique solution of \eqref{hom} and the proof is complete.
\end{proof}
\begin{remark}
In the case $M_i > 0$ for every $i=1,\cdots,n$ we do not need to impose any condition for the null space of system $\eqref{homT}$ as in A1). Indeed, following the previous analysis we get the same energy equality:
\begin{equation}\label{energy}
   \frac{1}{2}\sum_{i,j=1}^{n}b_{i,j} \ri \rj |\ui-\uj|^2 + \sum_{i=1}^{n}M_i \ri |\ui|^2 = 0 
\end{equation}
and, as before, 
\begin{equation*}
     \sum_{i=1}^{n}M_i \ri |\ui|^2 =  0.
\end{equation*}
Then, since $M_i>0$ for every $i$, then the only solution is the trivial one, namely $u_i=0$ for $i=1,\cdots,n$
\end{remark}
Summarizing, with $\tau_{i,j}$ given as before, system  \eqref{hom} rewrites as follows:
\begin{equation}\label{inv}
    \sum_{j=1}^n \tau_{i,j}u_j +
    M_i \ri \ui = 0, \qquad i=1, \cdots, n,
\end{equation}
and therefore, under the assumption A2), the matrix
$\tau^M$
is invertible and we can find the solution $\ui$ for every $ i \in \{1,\cdots,n\}$ from \eqref{nonhom}: $u =- (\tau^M_{i,j})^{-1} d$. This is the result corresponding to Lemma \ref{lem:1} when a diagonal friction term is present. In the sequel, we shall then refer to the aforementioned frameworks: A1) if $M=0$ and A2) if   $M\neq 0$.
The unique solution of the corresponding non homogeneous linear system will then be used  in the forthcoming Hilbert expansion,
 leading to the hyperbolic equilibrium dynamics satisfied by the barycentric velocity in the framework A1), and  leading to the (easier) case of a parabolic equilibrium dynamics, in the framework A2).

We are now ready to  perform the Hilbert expansion of system $\eqref{eq:nskv}$ in the high friction regime, namely for small $\e>0$. For this, let us introduce the following quantities:
\begin{equation}\label{he}
	\begin{aligned}
		& \ri^\e= \ri^0 + \e \ri^1 + O(\e^2); \\
		& \ui^\e = \ui^0 + \e \ui^1 + O(\e^2); \\
		& v_i^\e = v_i^0 + \e v_i^1 + O(\e^2).
	\end{aligned}
\end{equation}
We study separately the cases $M>0$ and $M=0$. 
\subsection{$\mathbf{M >0}$; framework A2)}
Let us start by inserting the expansions \eqref{he} into the equations $\eqref{eq:nskv}$.  Collecting the terms of the same order we get: 
\begin{itemize}
    \item  $O(1/\e)$:
    \begin{equation*}
\sum_{j=1}^{n} b_{i,j} \ri^0 \rj^0 (u_i^0 -u_j^0) + M_i \ri^0u_i^0  = \sum_{j=1}^{n} \tau^{M,0}_{i,j} \cdot u_j^0 = 0;\quad \tau^{M,0}_{i,j} =  -b_{i,j} \ri^0 \rj^0 + M_i \rho_i^0;
\end{equation*}

\item $O(1)$: 
 \begin{align*}
&\partial_t\rho_i^0 + \dive(\rho_i^0 u_i^0 ) = 0; \\
&\partial_t(\rho_i^0 u_i^0 ) +\dive (\rho_i^0 u_i^0 \otimes u_i^0) - 2\nu\dive (\mu_L(\rho_i^0) D(u_i^0)) - \nu \nabla( \lambda_L(\rho_i^0)\dive u_i^0)+ \nabla p(\rho_i^0) \\
&= \dive(\mu(\ri^0) \nabla v_i^0) + \frac{1}{2} \nabla(\lambda(\ri^0) \dive v_i^0 - \sum_{j=1}^{n} b_{i,j} \ri^1 \rj^0 (u_i^0 -u_j^0) \\ & - \sum_{j=1}^{n} b_{i,j} \ri^0 \rj^1 (u_i^0 -u_j^0) - M_i \ri^1u_i^0 -  \sum_{j=1}^{n} \tau^{M,0}_{i,j} u_j^1; \\
&\partial_t(\rho_i^0 v_i^0 ) +\dive (\rho_i^0 v_i^0 \otimes u_i^0) + \dive(\mu(\ri^0)^t\nabla u_i^0) + \frac{1}{2} \nabla (\lambda(\ri^0) \dive u_i^0) = 0.
\end{align*}
\end{itemize}
Under the framework A2), 
we know that $\tau^M$ is invertible  and therefore the only solution of the homogeneous system obtained at order $1/\e$ is $u_i^0=0$ for every $i=1, \cdots,n$. 
Hence, 
the momentum equation in $O(1)$ reduces to
\begin{equation}\label{pm1}
\sum_{j=1}^{n} \tau^{M,0}_{i,j} u_j^1 = - \nabla p(\ri^0) + \dive(\mu(\ri^0) \nabla v_i^0) + \frac{1}{2} \nabla(\lambda(\ri^0) \dive v_i^0   = : d_i^0,
\end{equation}
and, since $\tau^{M,0}$ is invertible, we can find $\ui^1$ as function of $\rho_i^0$:
\begin{equation}\label{gfstructure}
    u_i^1 = \sum_{k=1}^{n}(\tau^{M,0})_{i,k}^{-1} d^0_k.
\end{equation}
Moreover, the continuity equation and the ``drift" equation reduce to the trivial relations $\partial_t \ri^0=0$ and $\partial_t( \ri^0 v_i^0) = 0$, being $\ri^0v_i^0= \nabla \mu(\ri^0)$.

At this point it remains to express each component in $\R^3$ of the $n$-velocities $u_i$. Following the argument in \cite{OR} we compute in $\T \times(0,T)$ the Kronecker product:
\begin{equation}\label{eq:Kron}
\bar{\tau}^{M,0} = \tau^{M,0} \otimes I_3 \in \R^{3n \times 3n}.
\end{equation}
The matrix $\bar{\tau}^{M,0}$ is invertible since it is defined as the Kronecker
product of two invertible matrices, and therefore we can explicit the first order non-trivial momentum in the expansion, namely $\ri^0 \ui^1$ for every i $\in \{1,\cdots,n\}$.

Now, using \eqref{gfstructure} in the $O(\e)$ part of the continuity equation
\begin{equation}\label{eq:rho1}
    \partial_t\ri^1 + \dive(\ri^0 u_i^1) = 0,
\end{equation}
we readily obtain
\begin{equation}\label{eq:rho1bis}
    \partial_t\ri^1 + \dive \left( \ri^0 \; (\bar{\tau}^{M,0})^{-1} \cdot \left( \nabla p(\ri^0) - \dive(\mu(\ri^0) \nabla v_i^0) - \frac{1}{2} \nabla(\lambda(\ri^0) \dive v_i^0)\right) \right)=0.
\end{equation}

Therefore,  the continuity equation for $\ri^\e$ and $u_i^\e$ in \eqref{he} 
up to order $\e^2$ is given by summing  the trivial relation $\partial_t\ri^0=0$ to $\e$ times \eqref{eq:rho1}:  
\begin{equation*}
	\partial_t (\ri^0 +\e\ri^1)  + \e \dive( \ri^0 u_i^1)= 0, 
\end{equation*}
that is, using instead \eqref{eq:rho1bis}:
\begin{equation*}
	\partial_t \ri^\e + \e \dive\left(\ri^{\e} \;(\bar{\tau}^{M,\e})^{-1} \cdot \left( \nabla p(\ri^\e) - \dive(\mu(\ri^\e) \nabla v_i^\e) - \frac{1}{2} \nabla(\lambda(\ri^\e) \dive v_i^\e)\right) \right)=0,
\end{equation*}
where 
\begin{equation*}\label{eq:Kro}
\bar{\tau}^{M,\e} = \tau^{M,\e} \otimes I_3; \quad 
 \tau^{M,\e}_{i,j} =  -b_{i,j} \ri^\e \rj^\e + M_i \rho_i^\e.
\end{equation*}

Hence, 
in order to recover a non-trivial behaviour in the limit  $\e \rightarrow 0$, as customary in diffusive limits, one has to rescale the time variable $\partial_t \rightarrow \e \partial_t$ in the original model, as done in \cite{OR}. In this paper, 
the Authors consider the following system:
\begin{equation}\label{or}
	\begin{aligned}
		& \partial_t \ri + \frac{1}{\e} \dive( \ri \ui) = 0 \\
		& \partial_t(\ri \ui) + \frac{1}{\e} \dive(\ri u_i \otimes \ui) + \frac{1}{\e} \nabla p(\ri) = - \frac{1}{\e^2} \sum_{j=1}^{n} b_{i,j} \ri \rj (\ui- \uj) - \frac{1}{\e^2}M_i \ri \ui 
	\end{aligned}
\end{equation}
and show that as $\e \rightarrow 0$ the hyperbolic system converges to the parabolic system:
\begin{equation*}
\partial_t \rri + \dive\left( \tilde{C}(\bar{\rho})^{-1}  \nabla p(\rri) \right)=0 \quad  \text{ for } i=1, \cdots,n,
\end{equation*}
where $\tilde{C}(\bar{\rho}) = \bar{C}(\bar{\rho}) \otimes I_3$ as defined in Remark \ref{rem:solOR} below.
In the present paper we extend the results of \cite{OR} by adding capillarity and viscosity effects in the system,  thus studying the diffusive relaxation limit  of  the  model  \eqref{eq:sin2},  obtained after the aforementioned scaling of time in the augmented formulation \eqref{eq:nskv}; see Section \ref{sec:4} for details. 

Specifically, the system we shall investigate becomes
\begin{equation}\label{eq:nskv-diff}
	\begin{aligned} 
		&\partial_t\rho_i + \frac{1}{\e} \dive(\rho_i u_i ) = 0 
		\\ 
		&\partial_t(\rho_i u_i ) + \frac{1}{\e} \dive (\rho_i u_i \otimes u_i) - \frac{2}{\e}\nu\dive (\mu_L(\rho_i) D(u_i)) - \frac{\nu}{\e} \nabla( \lambda_L(\rho_i)\dive u_i)+ \frac{1}{\e} \nabla \rho_i^{\gamma} 
		\\
		&\ = \frac{1}{\e} \dive(\mu(\ri)\nabla v_i) + \frac{1}{2\e} \nabla (\lambda(\ri) \dive v_i)  -\frac{1}{\e^2} \sum_{j=1}^{n}b_{i,j}\rho_i \rho_j (u_i-u_j) - \frac{M_i}{\e^2} \ri \ui 
		\\
		&\partial_t(\rho_i v_i ) + \frac{1}{\e}\dive (\rho_i v_i \otimes u_i) + \frac{1}{\e} \dive(\mu(\ri)^t\nabla u_i) + \frac{1}{2\e} \nabla (\lambda(\ri) \dive u_i) = 0,
	\end{aligned} 
\end{equation}
Therefore, as noticed above, using the Hilbert expansion, from the $O(1/\e^2)$ term  
    \begin{equation*}
\sum_{j=1}^{n} b_{i,j} \ri^0 \rj^0 (u_i^0 -u_j^0) + M_i \ri^0u_i^0  = \sum_{j=1}^{n} \tau^{M,0}_{i,j} \cdot u_j^0 = 0,
\end{equation*}
we readily obtain $u^0_i = 0$ for any $i=1,\dots,n$. Hence,  the continuity equation starts from order $O(1)$, that is:
\begin{equation*}
    \partial_t\rho_i^0 + \dive(\rho_i^0 u_i^1) = 0,
\end{equation*}
and the momentum equation at order $O(1/\e)$ gives \eqref{pm1} and, as a consequence, \eqref{gfstructure}. 
Finally, for $\bar{\tau}^{M,0}$ defined in \eqref{eq:Kron}, we 
obtain the following parabolic system at equilibrium:
\begin{equation*}
    \partial_t\ri^1 + \dive \left( \ri^0 \; (\bar{\tau}^{M,0})^{-1} \cdot \left( \nabla p(\ri^0) - \dive(\mu(\ri^0) \nabla v_i^0) - \frac{1}{2} \nabla(\lambda(\ri^0) \dive v_i^0)\right) \right)=0.
\end{equation*}

\begin{remark}\label{rem:solOR}  
Let us remark that the Authors in \cite{OR} make the following assumption on the friction coefficients, motivated by experimental evidence: $M_i >> b_{i,j}$, ensuring that the matrix $\bar{C} =\text{diag}(M_i) - \text{diag}(\ri) B$, where $B= \{b_{i,j}\}_{i,j=1}^{n}$ is in fact diagonally dominant and positive definite, thus invertible using the Kronecker product as above. Thanks to this condition, each component in the limiting dynamics evolves according to its velocity, while, without the diagonal term $M$, one needs to introduce a 
\emph{barycentric velocity} as done in next section.
We stress that the invertibility condition, referred to the matrix  $\bar{C}$ associated to the non homogeneous system $\eqref{pm1}$ has the purpose to solve for $\ri^0 \ui^1$ while, from the discussion above, we deduced the invertibility of the matrix $-\tau_{i,j} + \text{diag}(M_i \ri)$, which allows us  to find directly $\ui^1$ in function of $\ri^0$; see \eqref{gfstructure}. 
\end{remark}

\subsection{$\mathbf{M =0}$; framework A1)}
In the case under investigation, namely for $M_i=0$ and $\nu \geq 0$,  we follow the approach of \cite{HJT}. We introduce the \emph{barycentric velocity u}, the \emph{relative velocities $w_i = \ui-u$},  and their  expansions, namely:
\begin{equation}\label{wi e u}
    \begin{aligned}
    & w_i^\e = w_i^0 + \e w_i^1 + O(\e^2);  \\
    & u^\e = u^0 + \e u^1 + O(\e^2).
\end{aligned}
\end{equation}
Moreover, we have to consider also the total density and the moments:
\begin{equation}\label{rel1}
	\rho= \sum_{i=1}^{n} \ri, \quad \rho u= \sum_{i=1}^{n} \ri \ui, \quad \rho v= \sum_{i=1}^{n} \ri v_i.
\end{equation}
Using these extra quantities, the original system $\eqref{eq:nskv}$ becomes: 
\begin{equation}\label{cp1}
\begin{aligned}
	&\partial_t \ri + \dive(\ri w_i + \ri u) = 0 
	\\
	&\partial_t (\ri w_i + \ri u) + \dive(\ri( w_i + u) \otimes ( w_i + u)) + \nabla \ri^{\gamma} - 2\nu\dive (\mu_L(\rho_i) D(w_i + u)) \\
	& - \nu \nabla( \lambda_L(\rho_i)\dive (w_i +u)) = 
	 \dive(\mu(\ri) \nabla v_i) + \frac{1}{2} \nabla(\lambda(\ri) \dive v_i) - \frac{1}{\e} \sum_{i,j=1}^{n}b_{i,j} \ri \rj (w_i - w_j)  
	 \\
	& \partial_t (\ri v_i) + \dive(\ri v_i \otimes ( w_i + u)) + \dive(\mu(\ri){}^t\nabla ( w_i + u)) + \frac{1}{2}\nabla(\lambda(\ri)\dive( w_i + u)) = 0,
\end{aligned}
\end{equation}
subject to the condition:
\begin{equation}\label{cond}
\sum_{i=1}^{n} \ri w_i = \sum_{i=1}^{n} \ri (\ui- u) = \sum_{i=1}^{n} \ri \ui - \rho u = 0.
\end{equation}
Using \eqref{wi e u} in \eqref{cond} the latter becomes:
\begin{equation*}
0 = \sum_{i=1}^{n} \ri w_i = \sum_{i=1}^{n} \ri^0 w_i^0 + \e \sum_{i=1}^{n} (\ri^0 w_i^1 + \ri^1 w_i^0) + O(\e^2),
\end{equation*}
that is, looking at order $O(1)$ and at order $O(\e)$:
\begin{equation}\label{cond2}
\sum_{i=1}^{n} \ri^0 w_i^0 = 0; \qquad \sum_{i=1}^{n} (\ri^0 w_i^1 + \ri^1 w_i^0) = 0.
\end{equation}
Finally, we use $\eqref{he}_1$ to deduce the following expansion of the total mass   $\rho$:
\begin{equation*}
\rho = \rho^0 + \e \rho^1 + O(\e^2),  
\end{equation*}
where 
\begin{equation*}
 \rho^0 := \sum_{i=1}^{n} \ri^0; \quad \rho^1 := \sum_{i=1}^{n} \ri^1.
\end{equation*}
Plugging the above expansions  in $\eqref{cp1}$ and collecting the terms of the same order in $\e$ we obtain:
\begin{itemize}
\item $O(1/\e)$:
\begin{equation}\label{O(1/e)}
	\sum_{j=1}^{n} b_{i,j} \ri \rj (w_i^0 -w_j^0) = 0;
\end{equation}
\item $O(1)$:
\begin{equation}\label{O(1)}
\begin{aligned}
&	 \partial_t \ri^{0} + \dive(\ri^0 w_i^0 + \ri^0 u^0) = 0; \\
\\
&	 \partial_t (\ri^0 w_i^0 + \ri^0 u^0) + \dive(\ri^0 (w_i^0 +u^0) \otimes (w_i^0 +u^0)) + \nabla p(\ri^0)
\\
&\ - 2\nu\dive (\mu_L(\rho_i) D(w_i^0 + u^0)) 
	 -\nu \nabla( \lambda_L(\rho_i^0)\dive (w_i^0 +u^0))
\\
&\ = \dive(\mu(\ri^0) \nabla v_i^0) + \frac{1}{2} \nabla(\lambda(\ri^0) \dive v_i^0) - \sum_{j=1}^{n} b_{i,j} \ri^0 \rj^0 (w_i^1-w_j^1)  
\\
&\ - \sum_{j=1}^{n} b_{i,j} (\ri^1 \rj^0 + \ri^0 \rj^1)(w_i^0-w_j^0);\\
\\
	 & \partial_t(\ri^0 v_i^0) + \dive(\ri^0 v_i^0 \otimes (w_i^0 + u^0)) + \dive (\mu(\ri^0) {}^t\nabla(w_i^0 + u^0)) 
	 \\
	 &\ + \frac{1}{2}\nabla(\lambda(\ri^0) \dive(w_i^0 + u^0)) = 0.
\end{aligned}
\end{equation}
\end{itemize}
We note that, being all densities strictly positive, condition $\eqref{cond2}_1$ implies $w_i^0 = 0$  for $i=1,\cdots,n$, 
simplifying $\eqref{O(1)}$ considerably. Then we sum from $i=1,\cdots,n$ the momentum equation  $\eqref{O(1)}_2$ obtaining in this way the equation satisfied by the barycentric velocity. Specifically:
\begin{equation}\label{cp3}
\begin{aligned}
 &\partial_t \left(\sum_{i=1}^{n} \ri^0 u^0 \right) + \dive\left(\sum_{i=1}^{n}\ri^0 u^0 \otimes u^0\right) + \sum_{i=1}^{n}\nabla p(\ri^0) - 2\nu \sum_{i=1}^n\dive (\mu_L(\rho^0_i) D(u^0)) \\
 & - \nu \sum_{i=1}^n\nabla( \lambda_L(\rho^0_i)\dive (u^0)) = \sum_{i=1}^{n}\dive(\mu(\ri^0) \nabla v_i^0) + \sum_{i=1}^{n}\frac{1}{2} \nabla(\lambda(\ri^0) \dive v_i^0) = 0, 
\end{aligned}
\end{equation}
where we used the equality
\begin{equation*}
   \sum_{i,j=1}^{n} b_{i,j} \ri^0 \rj^0 (w_i^1-w_j^1) = 0
\end{equation*}
due to symmentry.
Hence, 
denoting $\ri^0 = \bar{\ri}$,   $\sum_{i=1}^{n} \ri^0 = \bar{\rho}$,  $v_i^0 = \bar{v_i}$ and $u^0= \bar{u}$, the leading term in the Hilbert expansion of   \eqref{cp1} is given by the following system:
\begin{equation*}
	\begin{aligned}
		&\partial_t\rri + \dive(\rri \bar u )=0; \\ \\
		& \partial_t(\bar \rho \bar u) + \dive(\bar\rho \bar u \otimes \bar{u}) + \sum_{i=1}^{n}\nabla p(\rri) - 2\nu \sum_{i=1}^n\dive (\mu_L(\rri) D(\bar{u})) - \nu \sum_{i=1}^n\nabla( \lambda_L(\rri)\dive (\bar{u})) \\
		&\ =
		\sum_{i=1}^{n}\dive(\mu(\rri) \nabla \bar{v}_i) + \sum_{i=1}^{n}\frac{1}{2} \nabla(\lambda(\rri) \dive \bar{v}_i); \\   \\
		& \partial_t(\rri \bar{v}_i) + \dive(\rri \bar{v}_i \otimes \bar{u}) + \dive (\mu(\rri) {}^t\nabla(\bar{u})) + \frac{1}{2}\nabla(\lambda(\rri) \dive(\bar{u})) = 0.
	\end{aligned}
\end{equation*}
The latter can be rewrite in the following way:
\begin{equation}\label{sys1}
	\left\{\begin{aligned}
		&\partial_t\rri + \dive(\rri \bar u )=0 \\
		& \partial_t(\rri \bar u) + \dive(\rri \bar u \otimes \bar{u}) - 2 \nu  \dive(\mu_L(\rri) D(\bar u)) - \nu \nabla(\lambda_L(\rri)\dive \bar{u})+ \nabla p(\rri)  \\
		&\ =\dive(\mu(\rri) \nabla \bar{v}_i) + \frac{1}{2} \nabla(\lambda(\rri) \dive \bar{v}_i) + \bar{R}_i \\
		& \partial_t(\rri \bar{v}_i) + \dive(\rri \bar{v}_i \otimes \bar{u}) + \dive (\mu(\rri) {}^t\nabla(\bar{u})) + \frac{1}{2}\nabla(\lambda(\rri) \dive(\bar{u})) = 0.
	\end{aligned}\right.
\end{equation}
The term $\bar{R_i}$ is defined as:
\begin{equation}\label{defri}
\begin{aligned}
\bar{R_i}  & :=  \frac{\ri}{\bar{\rho}}\sum_{j=1}^n\left[ \dive{\bar{S_j}} - \nabla  p(\rrj)\right]  - \dive{\bar{S_i}}   + \nabla p(\rri)
\\
& \ + \frac{\ri}{\bar{\rho}}\sum_{j=1}^n \left[2\nu \dive{\mu_L(\rrj)D(\bar{u})}+ \nu \nabla (\lambda_L(\rrj)\dive{\bar{u}})\right] - 2\nu \dive{\mu_L(\rri)D(\bar{u})}
\\
&\ - \nu \nabla (\lambda_L(\rri)\dive{\bar{u}})\\
&=  \frac{\ri}{\bar{\rho}} \dive \bar{T} - \dive \bar{T_i} + \nu \dive \bar{D}\frac{\rri}{\bar{\rho}} - \nu \dive \bar{D}_i,
\end{aligned}
\end{equation}
where:
\begin{equation*}
    \begin{aligned}
    & \dive{\bar{T_i}}= \dive{\bar{S_i}} - \nabla p(\rri); \qquad \dive{\bar{T}} = \sum_{j=1}^n \dive{\bar{T_j}}; \\
    &\dive{\bar{D_i}}=2\dive{(\mu_L(\rri)D(\bar{u}))} + \nabla(\lambda_L(\rri) \dive{\bar{u}}); \qquad \dive{\bar{D}} = \sum_{j=1}^n \dive{\bar{D_j}}.
    \end{aligned}
\end{equation*}
In  Section \ref{sec:3} we compute the relative entropy between weak solutions of \eqref{cp1} and strong solutions of \eqref{sys1}, thus justifying  in this framework the formal analysis above. 
\begin{remark}\label{rem:specialvisc}
It is worth observing that in the proposed scaling, contrarily to what happens under the diffusive scaling used in the previous case ${M}>0$, the viscosity terms are not higher order, and they persist in the  limit \eqref{sys1}, as it also manifest looking at equation \eqref{O(1)}. For this reason, as proposed in \cite{BL}, in Section \ref{sec:3} we shall make the particular choice of Lam\'e coefficients $\mu_L(\ri)=\mu(\ri)$ and $\lambda_L(\ri)=\lambda(\ri)$, which allows us to control the  viscosity  terms by means of the relative entropy. On the other hand, being $\bar{u_i}, \dive{\bar{u_i}}, D(\bar{u_i})$ of order $\e$  in the diffusive scaling,  these terms can be treated as errors, and therefore this particular choice for  Lam\'e coefficients is not required; concerning this,   
see also \cite{CL}.
\end{remark}
\begin{remark}
Let us underline the relations between  equation \eqref{cp3} and the solvability for $w_i^1$ of $\eqref{O(1)}_2$  when $w_i^0=0$. 
With the notation 
\begin{align*}
    d_i^0 &:= \partial_t(\ri^0u^0) + \dive(\ri^0 u^0 \otimes u^0) + \dive S[\ri^0,v_i^0] - 2\nu \dive (\mu_L(\rho_i^0) D(u^0))
    \\
    &\ - \nu \nabla( \lambda_L(\rho_i^0)\dive (u^0)),
\end{align*}
equation $\eqref{O(1)}_2$ rewrites as
\begin{equation}\label{cond3}
	 - \sum_{j=1}^{n} b_{i,j} \ri^0 \rj^0(w_i^1-w_j^1)= d_i^0, 
\end{equation}
hence \eqref{cp3} is equivalent to the condition $\sum_{i=1}^{n} d_i^0 = 0$, which ensures there exists a unique solution $(w_i^1,..., w_n^1)$ to $\eqref{cond3}$ in view of  Lemma \ref{lem:1}.
\end{remark}
\section{High friction limit with diffusive scaling for $M_i > 0$}\label{sec:4}
 The system we are going to study in this framework generalizes the one considered in \cite{OR},  where no capillarity and viscosity effects  are taking into account and, 
for reader's convenience, let us rewrite it here below:
\begin{equation}\label{eq:sin}
	\begin{aligned} 
		&\partial_t\rho_i + \dive(\rho_i u_i ) = 0 \\ 
		&\partial_t(\rho_i u_i ) +\dive (\rho_i u_i \otimes u_i) - 2\nu\dive (\mu_L(\rho_i) D(u_i)) - \nu \nabla( \lambda_L(\rho_i)\dive u_i) + \nabla \rho_i^{\gamma} \\
		&\ = \dive S_i -\frac{1}{\e} \sum_{j=1}^{n}b_{i,j}\rho_i \rho_j (u_i-u_j) - \frac{M_i}{\e} \ri \ui\\
		&\partial_t(\ri v_i) + \dive (\rho_i u_i \otimes v_i) +  \dive K_i = 0,
	\end{aligned} 
\end{equation}
where
\begin{equation*}
	\begin{aligned}
		&\dive S_i =  \dive(\mu(\ri)\nabla v_i) + \frac{1}{2} \nabla (\lambda(\ri) \dive v_i) = \dive T_i + \nabla \ri^{\gamma},\\
		& \dive K_i = \dive(\mu(\ri)^t\nabla u_i) + \frac{1}{2} \nabla (\lambda(\ri) \dive u_i).
	\end{aligned}
\end{equation*}
As already noticed in Section \ref{sec:1}, 
 in order to obtain a non trivial limiting dynamics as $\e \rightarrow 0$, we rescale the time variable $\partial_t \rightarrow \e \partial_t$  to end up to a diffusive limit: 
\begin{equation}\label{eq:sin2}
	\begin{aligned} 
		&\partial_t\rho_i + \frac{1}{\e}\dive(\rho_i u_i ) = 0 \\ 
		&\partial_t(\rho_i u_i ) + \frac{1}{\e}\dive (\rho_i u_i \otimes u_i) - \frac{2}{\e}\nu\dive (\mu_L(\rho_i) D(u_i)) - \frac{\nu}{\e} \nabla( \lambda_L(\rho_i)\dive u_i) + \frac{1}{\e}\nabla \rho_i^{\gamma}  \\
		&\ = \frac{1}{\e}\dive S_i  - \frac{1}{\e^2} \sum_{j=1}^{n}b_{i,j}\rho_i \rho_j (u_i-u_j) - \frac{M_i}{\e^2} \ri \ui \\
		& \partial_t(\ri v_i) + \frac{1}{\e}\dive (\rho_i u_i \otimes v_i) + \frac{1}{\e} \dive K_i = 0.
	\end{aligned} 
\end{equation}
Then, the leading terms of the Hilbert expansion are
the  following equilibrium relations for each component of the mixture: 
\begin{align*}
& u_i^0 = 0, \\
&\ui^1= (\bar\tau^{M,0})^{-1}(-\nabla p(\ri^0) + \dive S_i(\ri^0,v_i^0)), \text{ \emph{the Darcy's law} in our context}, \\
&v_i^0= \frac{\nabla \mu(\ri^0) }{ \ri^0},
\end{align*} 
and each $\ri^0$ solves the gradient flow equation
\begin{equation}\label{gf}
	\partial_t \ri^0 + \dive(\ri^0(\bar\tau^{M,0})^{-1}(-\nabla p(\ri^0) + \dive S_i(\ri^0,v_i^0)))=0.
\end{equation}

In order to compare weak solution of  $\eqref{eq:sin2}$ and strong solution of the parabolic equilibrium $\eqref{gf}$ we rewrite the latter as Euler-Korteweg system with high friction and an error term of order $\e$, using the same strategy of \cite{CL,LT,LT2}. Indeed $\eqref{gf}$ is equivalent to:
\begin{equation}\label{gfeq}
	\begin{aligned}
	&\partial_t\rri + \frac{1}{\e}\dive(\rri \uui ) = 0 \\ 
	&\partial_t(\rri \uui ) + \frac{1}{\e}\dive (\rri \uui \otimes \uui) + \frac{1}{\e}\nabla \rri^{\gamma} = \frac{1}{\e}\dive(\rri\nabla \bar{v}_i)  - \frac{1}{\e^2} \sum_{j=1}^{n}b_{i,j} \rri \rrj (\uui-\uuj) - \frac{M_i}{\e^2} \rri \uui + \bar{e}_i\\
	& \partial_t( \rri \bar{v}_i) + \frac{1}{\e}\dive (\rri \uui \otimes \bar{v}_i) + \frac{1}{\e} \dive(\rri {}^t \nabla \uui) = 0,
	\end{aligned}
\end{equation} 
where we have denoted 
\begin{align*}
 \ri^0 &= \rri; 
 \\
  \uui &=  \e (\bar\tau^{M})^{-1}(-\nabla p(\rri) + \dive S_i(\rri,\bar{v}_i));
 \\
 \bar{v}_i &= \frac{\nabla \mu(\rri) }{ \rri} ;  
 \\
 \bar\tau^{M} & = \tau^M(\rri) \otimes I_3
 \\
\bar{e}_i  & =  \partial_t(\rri \uui ) + \frac{1}{\e}\dive (\rri \uui \otimes \uui) \\
 & =  \e \partial_t(\rri (\bar\tau^{M})^{-1}(-\nabla p(\rri) + \dive S_i(\rri,\bar{v}_i)))  \\
& \ + \e \dive(\rri(\bar\tau^{M})^{-1}(-\nabla p(\rri) + \dive S_i(\rri,\bar{v}_i)) \otimes (\bar\tau^{M})^{-1}(-\nabla p(\rri) + \dive S_i(\rri,\bar{v}_i))) \\
& =  O(\e).
\end{align*}
The aim of this part is to validate the large friction limit using relative entropy techniques. For this, we start by recalling the definition of the total mechanical energy of our system $\eta_{tot}(\hat \rho^{\e},\hat m^{\e},\hat J^{\e})$:
\begin{equation*}
\eta_{tot}(\hat \rho^{\e},\hat m^{\e},\hat J^{\e}) = \sum_{i=1}^{n} \eta(\ri^\e,m_i^\e,J_i^\e) = \sum_{i=1}^{n} \left (\frac{1}{2} \frac{|m_i^\e|^2}{\ri^\e}+ \frac{1}{2} \frac{|J_i^\e|^2}{\ri^\e} + h(\ri^\e)\right ),
\end{equation*}
and its space--integrated version
\begin{equation*}
\Sigma_{tot}(\hat \rho^\e, \hat m^\e,\hat J^\e)(t)  = \int_{\T} \sum_{i=1}^{n}\eta(\ri^\e,m_i^\e,J_i^\e)(t) \;dx.
\end{equation*}
Hence, the relative entropy between solutions of \eqref{eq:sin2} and \eqref{gfeq} reads
\begin{equation}\label{ret-diff}
\begin{aligned}
& \eta_{tot}(\hat \rho^\e,\hat m^{\e},\hat J^{\e} |  \bar{\hat \rho}, \bar{\hat m}, \bar{\hat{J}})(t) := \sum_{i=1}^{n} \eta(\ri^{\e},\ui^{\e},J_i^{\e} | \rri, \uui, \bar{J}_i) = \\
& \sum_{i=1}^{n} \eta(\ri^\e,m_i^\e,J_i^\e) - \sum_{i=1}^{n} \eta(\rri, \bar{m}_i,\bar{J}_i) - \sum_{i=1}^{n} \eta_{\ri}(\bar{\rho}, \bar{m},\bar{J}) (\ri^\e-\rri) \\
& - \sum_{i=1}^{n} \eta_{m_i}(\bar{\rho}, \bar{m},\bar{J})(m_i^\e-\bar{m}_i) -  \sum_{i=1}^{n} \eta_{J_i}(\bar{\rho}, \bar{m},\bar{J})(J_i^\e - \bar{J}_i),
\end{aligned}
\end{equation}
and 
\begin{equation*}
\Sigma_{tot}(\hat \rho^\e ,\hat m^ \e,\hat J^\e |  \bar{\hat \rho}, \bar{\hat m}, \bar{\hat{J}})(t)  = \int_{\T} \eta_{tot}(\hat \rho^\e,\hat m^{\e},\hat J^{\e} |  \bar{\hat \rho}, \bar{\hat m}, \bar{\hat{J}})(t) \;dx.
\end{equation*}

Let us now clarify  the notion of periodic
\emph{weak} and \emph{dissipative weak} solutions of the relaxation system $\eqref{eq:sin2}$ we shall consider in the sequel.
\begin{definition}[weak and dissipative weak solutions]\label{def:ws-diff}
A triple $(\hat \rho^{\e}, \hat m^{\e}, \hat J^{\e})$ (where $\hat \rho^\e= (\rho_1^\e, \cdots, \rho_n^\e)$, $\hat m^\e= (m_1^\e, \cdots, m_n^\e)$, $\hat J^\e= (J_1^\e, \cdots, J_n^\e)$) such that
\begin{align*}
& 0 \leq \ri^{\e} \in  C^0([0,\infty);L^{1}(\T)), \\
&(\ri^{\e} \ui^{\e},\ri^{\e} v_i^{\e})  \in  C^0([0,\infty);L^{1}(\T)^3\times L^{1}(\T)^3)\\
\end{align*}
is called a periodic \emph{weak  solution} of $\eqref{eq:sin2}$ if for any $i=1,\cdots,n$:
\begin{align*}
	&\sqrt{\ri^{\e}} \ui^{\e} \in L^{\infty}((0,T);L^2(\T)^3) \\
	&\sqrt{\ri^{\e}} v_i^{\e} \in L^{\infty}((0,T);L^2(\T)^3) \\
	&\ri^{\e} \in  C^0([0,\infty);L^{\gamma}(\T))\\
	&\mu_L(\ri^{\e}) D (\ui^\e) \in L^1((0,T);L^1(\T)^{3\times3}) \\
	&\lambda_L(\ri^\e) \dive \ui^\e \in L^1((0,T);L^1(\T)) \\
	& (\ri^{\e})^2 u_i^{\e} \in L^{\infty}((0,T);L^1(\T)^3) 
\end{align*}
and $(\ri^{\e},\ui^{\e},v_i^{\e})$ satisfy for all $\psi_i \in C_0^{1}([0,\infty); C^{1}(\R^3))$ and $\phi_i, \varphi_i \in C_0^1([0,\infty); C^1(\T)^3)$:
\begin{equation*}
- \int_{0}^{\infty} \int_{\T} \sum_{i=1}^{n}(\ri^{\e} \partial_t\psi_i + \frac{1}{\e}\ri^{\e}\ui^{\e} \cdot \nabla \psi_i) \; dxdt = \int_{\T} \sum_{i=1}^{n} \ri^{\e}(x,0) \psi_i(x,0)\;dx,
\end{equation*}
\begin{align*}
&- \int_{0}^{\infty} \int_{\T} \sum_{i=1}^{n}\ri^{\e} \ui^{\e} \partial_t \phi_i + \frac{1}{\e} \ri^{\e}\ui^{\e} \otimes \ui^{\e} : \nabla \psi_i + \frac{1}{\e}{\rho^{\e}_i }^{\gamma} \dive \psi_i                 + \frac{1}{\e}\mu(\ri^{\e})v_i \cdot \nabla \dive \phi_i 
\\
&\ + \frac{1}{\e}\nabla \mu(\ri^{\e}) \cdot (\nabla \phi_i v_i) dxdt 
 -\int_{0}^{\infty} \int_{\T}  \sum_{i=1}^{n} \frac{1}{\e}\left( \frac{1}{2} \nabla \lambda(\ri^{\e}) \cdot v_i^{\e} \dive \phi_i + \frac{1}{2} \lambda(\ri^{\e})v_i^{\e} \cdot \nabla \dive \phi_i \right)dxdt 
 \\
 &\ + \frac{2\nu}{\e} \int_0^{\infty} \int_{\T} \mu_L(\ri^\e)D(\ui^\e) : D(\phi_i)dxdt 
 + \frac{\nu}{\e} \int_0^{\infty} \int_{\T} \lambda_L(\ri^\e)\dive\ui^\e \dive \phi_idxdt
\\
&\ + \frac{1}{\e^2}\int_{0}^{\infty} \int_{\T} \sum_{i,j=1}^{n} b_{i,j}\ri^{\e}\rj^{\e}(\ui^{\e}-\uj^{\e})\cdot \phi_i dxdt 
+ \frac{1}{ \e^2}\int_{0}^{\infty} \int_{\T} \sum_{i=1}^{n} M_{i} \ri^\e u_i^\e \cdot \phi \;dxdt
\\
& = \int_{\T} \sum_{i=1}^{n} (\ri^{\e}\ui^{\e})(x,0) \cdot \phi_i(x,0)\;dx,	
\end{align*}
\begin{align*}
	& - \int_{0}^{\infty} \int_{\T} \sum_{i=1}^{n}(\ri^{\e} v_i^{\e} \partial_t \varphi_i + \frac{1}{\e}\ri^{\e}\ui^{\e} \otimes v_i^{\e} : \nabla \varphi_i + \frac{1}{\e}\mu(\ri^{\e})v_i \cdot \nabla \dive\varphi_i + \frac{1}{\e}\nabla \mu(\ri^{\e}) \cdot (\nabla \varphi_i v_i) ) \; dxdt
	\\
	&\ - \int_{0}^{\infty} \int_{\T} \sum_{i=1}^{n} \frac{1}{\e}\left( \frac{1}{2} \nabla \lambda(\ri^{\e}) \cdot v_i^{\e} \dive \varphi_i + \frac{1}{2} \lambda(\ri^{\e})v_i^{\e} \cdot \nabla \dive\varphi_i \right) \; dxdt 
	\\
	& = \int_{\T} \sum_{i=1}^{n} (\ri^{\e} v_i^{\e})(x,0) \cdot \varphi_i(x,0)\;dx.	
\end{align*}

If in addition $\eta_{tot}(\hat \rho^{\e}, \hat m^{\e}, \hat J^{\e}) \in C([0,\infty);L^1(\T))$  and the integrated energy inequality 
\begin{equation}\label{ene1-diff}
\begin{aligned}
&  - \int_{0}^{\infty} \Sigma_{tot}( \hat \rho^{\e}(t), \hat m^{\e}(t), \hat J^{\e}(t)) \theta'(t)\;dt
\\
&\ 
+\frac{\nu}{ \e} \int_{0}^{\infty} \int_{\T} \sum_{i,j=1}^n\left ( 2\mu_L(\ri^\e) |D(\ui^\e)|^2  +
\lambda_L(\ri^\e) |\dive \ui^\e|^2) \ \right )\theta(t)
\;dxdt
\\
&  \  +  \frac{1}{2 \e^2}\int_{0}^{\infty} \int_{\T} \sum_{i,j=1}^{n} b_{i,j} \ri^{\e} \rj^{\e}|\ui^{\e}-\uj^{\e}|^2\theta(t)\;dxdt 
+ \frac{1}{ \e^2}\int_{0}^{\infty} \int_{\T} \sum_{i=1}^{n} M_{i} \ri^\e|u_i^{\e}|^2\theta(t)\;dxdt 
\\ 
 &  \leq 
 \Sigma_{tot}(\hat \rho^{\e}(0),\hat m^{\e}(0),\hat J^{\e}(0))\theta(0)
\end{aligned}	
\end{equation}
holds for any $\theta \in W^{1,\infty}([0,\infty))$ compactly supported in $[0,\infty)$, then we call $(\hat \rho^{\e}, \hat m^{\e}, \hat J^{\e})$ a periodic \emph{dissipative weak solution}.

We say that a periodic weak  (dissipative) solution $(\hat \rho^{\e}, \hat m^{\e}, \hat J^{\e})$ for $\eqref{eq:sin2}$ has finite total mass and finite total energy if for any $T>0$ there exists a constant $K>0$ independent of $\e$ such that:
\begin{equation}\label{A1-diff}
\begin{aligned}
& \sup_{t \in (0,T)} \int_{\T} \sum_{i=1}^{n} \ri^{\e}\;dx \leq K; \\
& \sup_{t \in (0,T)} \int_{\T} \sum_{i=1}^{n} \left( \frac{1}{2}\frac{|m_i^{\e}|^2}{\ri^{\e}} + \frac{1}{2} \frac{|J_i^{\e}|^2}{\ri^{\e}} + \frac{{\ri^{\e}}^{\gamma}}{\gamma - 1} \right)\;dx \leq K.
\end{aligned}
\end{equation}
\end{definition}

The relative entropy inequality is stated in the following proposition; to shorten notations, we shall omit the $\e$ dependence in the remaining pat of the section. 
\begin{proposition}\label{prop:2}
Assume condition A2) holds and let $(\hat \rho^{\e}, \hat m^{\e}, \hat J^{\e})$ be a dissipative weak periodic solution in the sense of Definition \ref{def:ws-diff}  with finite total mass and energy $\eqref{A1-diff}$. Let $\bar{w}_i=(\bar{\rho}_i, \bar{u}_i,\bar{v}_i)$ be a smooth solution to $\eqref{gfeq}$ such that:
\begin{equation*}
	\bar{w}_i, \partial_t\bar{w}_i, \nabla \bar{w}_i, D^2\bar{w}_i, D^3\bar{\rho}_i \in L^{\infty}([0,T];L^{\infty}(\T))
\end{equation*}
for any  $i=1,\dots,n$. 
Then, the following inequality holds:
\begin{align}
		&\Sigma_{tot}(\hat \rho^\e, \hat m^\e, \hat J^\e | \bar{\hat \rho}, \bar{\hat m}, \bar{ \hat J}) \Big|_{s=0}^t + \frac{1}{2 \e^2}\int_{0}^{t} \int_{\T} \sum_{i,j=1}^{n}b_{i,j} \ri \rj |(\ui-\uj)-(\uui-\uuj)|^2dxds 
		\nonumber\\
		& \ + \frac{2\nu}{ \e} \int_{0}^{t} \int_{\T} \sum_{i=1}^{n}\mu_L(\ri) |D(\ui - \bar{u}_i)|^2 dxds + \frac{\nu}{ \e} \int_{0}^{t} \int_{\T} \sum_{i=1}^{n}\lambda_L(\ri)|\dive (\ui -\bar{u}_i)|^2 dxds  
		\nonumber\\
		& \ + \frac{1}{\e^2}\int_{0}^{t} \int_{\T} \sum_{i=1}^{n} M_i \ri |\ui-\uui|^2dxds 
		\nonumber\\
		& \leq 
		- \frac{1}{\e}\int_{0}^{t} \int_{\T} \sum_{i=1}^{n}p(\ri | \rri) \dive \uui dxds 
		\nonumber\\
		& \phantom{\leq} - \frac{1}{\e}\int_{0}^{t} \int_{\T} \sum_{i=1}^{n} \left(  \ri \nabla \uui : (\ui -\uui) \otimes (\ui -\uui) \right) dxds
		\nonumber\\
		&\phantom{\leq} - \frac{1}{\e}\int_{0}^{t} \int_{\T} \sum_{i=1}^{n} \ri(\mu''(\ri)\nabla \ri - \mu''(\rri)\nabla \rri) \cdot ((v-\bar{v}_i)\dive \uui - (\ui-\uui)\dive \bar{v}_i)dxds 
		\nonumber\\
		&\phantom{\leq}  - \frac{1}{\e}\int_{0}^{t} \int_{\T} \sum_{i=1}^{n} \ri(\mu'(\ri) - \mu'(\rri)) ((v-\bar{v}_i)\cdot \nabla \dive \uui - (\ui-\uui) \cdot\nabla \dive \bar{v}_i)dxds
		\nonumber\\
		& \phantom{\leq}  - \frac{1}{\e}\int_{0}^{t} \int_{\T} \sum_{i=1}^{n} \left(  \ri \nabla \uui : (v_i -\bar{v}_i) \otimes (v_i -\bar{v}_i) \right)dxds
		\nonumber\\
		&\phantom{\leq} -\frac{1}{\e^2}\int_{0}^{t} \int_{\T} \sum_{i,j=1}^{n}b_{i,j}\ri(\rj -\rrj)(\uui-\ui)\cdot(\uui-\uuj)dxds 
		\nonumber\\
		&\phantom{\leq}  +\int_{0}^{t} \int_{\T} \sum_{i=1}^{n}  \frac{\ri}{\rri} \bar{e}_i\cdot(\bar{u}_i-\ui) dxds
		\nonumber\\
		&\phantom{\leq}  - \frac{2\nu}{ \e} \int_{0}^{t} \int_{\T} \sum_{i=1}^{n}\mu_L(\ri) D(\bar{u}_i): D(\ui -\bar{u}_i) dxds
		\nonumber\\
		&\phantom{\leq}  - \frac{\nu}{ \e} \int_{0}^{t} \int_{\T} \sum_{i=1}^{n}\lambda_L(\ri) \dive \bar{u}_i(\dive \ui - \dive \bar{u}_i)dxds.
		\label{eq:relent2}
		\end{align}
\end{proposition}
\begin{proof}
\begin{flushleft}
	\emph{Step 1: Energy Inequalities.}
\end{flushleft}

As customary in these contexts, we use the following test function in $\eqref{ene1-diff}$:
\begin{equation}\label{theta-diff}
\theta(s) = 
\begin{cases}
 1, & \text{for}\ 0 \leq s< t, \\
\displaystyle{\frac{t-s}{\delta} + 1}, & \text{for}\  t \leq s < t+ \delta, \\
0, & \text{for}\ s > t+ \delta,
\end{cases}
\end{equation}
and, passing to the limit as $\delta \rightarrow 0$ we get:
\begin{equation}\label{eq:re1-diff}
	\begin{aligned}
	&\Sigma_{tot}(\hat \rho^\e ,\hat m^\e,\hat J^\e )(t) + 
	\frac{\nu}{\e}\int_{0}^{t} \int_{\T} \sum_{i,j=1}^{n}\left (2\mu_L(\ri) |D(\ui)|^2
	+  
	 \lambda_L(\ri) |\dive \ui|^2\right )dxds
	\\
	&\ \ + \frac{1}{2 \e^2} \int_{0}^{t} \int_{\T} \sum_{i,j=1}^{n} b_{i,j}\ri \rj |\ui-\uj|^2 \;dxds 
	 +  \frac{1}{\e^2} \int_{0}^{t} \int_{\T} \sum_{i=1}^{n} M_i \ri |\ui|^2 \;dxds 
	 \\
	 &\ \leq  	\Sigma_{tot}(\hat \rho(0), \hat m(0),\hat J(0)).
	\end{aligned}
\end{equation}
We multiply  $\eqref{gfeq}_{2}$ by $\uui$ and  $\eqref{gfeq}_{3}$ by $\bar{v}_i$, integrating in time and space we get the energy inequality of system $\eqref{gfeq}$: 
\begin{equation}\label{eq:re2-diff}
	\begin{aligned}
			& \Sigma_{tot}(\bar{\hat \rho}(t), \bar{\hat m}(t),\bar{\hat J}(t)) +  \frac{1}{2 \e^2} \int_{0}^{t} \int_{\T} \sum_{i,j=1}^{n} b_{i,j}\rri \rrj |\uui-\uuj|^2 \;dxds 
			\\
			&\ 
		 +  \frac{1}{\e^2} \int_{0}^{t} \int_{\T} \sum_{i=1}^{n} M_i \rri |\uui|^2 \;dxds 
		 \\
		 & 
		 \leq  	\Sigma_{tot}(\bar{\hat \rho}(0),\bar{\hat m}(0), \bar{\hat J}(0)) + \int_{0}^{t} \int_{\T} \sum_{i=1}^{n} \bar{e}_i \cdot \uui \;dxds.
	\end{aligned}
\end{equation}
\begin{flushleft}
	\emph{Step II: Equations for the difference.}
\end{flushleft}

Now we evaluate the linear part of the relative entropy using suitable test functions in the weak formulation according to  Definition \ref{def:ws-diff}. If we take the differences $(\ri- \rri, \ui - \bar{\ui}, v_i - \bar{v}_i)$ they have to satisfy:
\begin{align*}
 & - \int_{0}^{t} \int_{\T} \sum_{i=1}^{n}(\ri- \rri) \partial_s\psi_{i} + \frac{1}{\e}(\ri\ui - \rri\bar{\ui})\cdot \nabla\psi_i dxds
 \\
&\ = \int_{\T} \sum_{i=1}^{n}(\ri(x,0)- \rri(x,0)\psi_i(x,0)) dx,
\end{align*}
	\begin{align*}
	& - \int_{0}^{t} \int_{\T} \sum_{i=1}^{n}(\ri \ui- \rri \bar{\ui}) \cdot \partial_s\phi_{i} + \frac{1}{\e}(\ri\ui \otimes \ui - \rri\bar{\ui} \otimes \bar{\ui}): \nabla\phi_i + \frac{1}{\e}(\ri^{\gamma} - \rri^{\gamma}) \dive \phi_i  \; dxds \\
	&- \frac{1}{\e}\int_{0}^{t} \int_{\T} \sum_{i=1}^{n} \left[(\mu(\ri)v_i - \mu(\rri)\bar{v}_i)\nabla \dive \phi_i  + (\nabla \mu(\ri)v_i - \nabla \mu(\rri)\bar{v}_i): \nabla \phi_i \right]\;dxds \\
	& - \frac{1}{2\e} \int_{0}^{t} \int_{\T} \sum_{i=1}^{n} \left[(\nabla \lambda(\ri)v_i - \nabla \lambda(\rri)\bar
	v_i)\dive \phi_i + (\lambda(\ri)v_i - \lambda(\rri)\bar{v}_i) \nabla\dive\phi_i\right] \\
	& + \frac{\nu}{\e}\int_{0}^{t} \int_{\T} \sum_{i,j=1}^{n}2\mu_L(\ri) D(\ui):D(\phi_i) + \lambda_L(\ri)  \dive \ui \dive \phi_i \; dxds
	\\
	& =
	-\frac{1}{\e^2}\int_{0}^{t} \int_{\T} \sum_{i,j=1}^{n} b_{i,j}(\ri \rj (\ui-\uj) - \rri \rrj (\uui-\uuj)) \phi_i \;dxds 
	\\
	&\phantom{=}- \frac{1}{\e^2}\int_{0}^{t} \int_{\T} \sum_{i=1}^{n} M_i (\ri \ui - \rri \bar{\ui} )\cdot \phi_i \; dxds -\int_{0}^{t} \int_{\T} \sum_{i=1}^{n} \bar{e_i} \phi \; dxds 
	\\
	& \phantom{=}+ \int_{0}^{t} \int_{\T} \sum_{i=1}^{n} ((\ri \ui)(x,0)- (\rri \uui)(x,0))\phi_i(x,0)\;dx,
	\end{align*}
and
	\begin{align*}
	&- \int_{0}^{t} \int_{\T} \sum_{i=1}^{n}(\ri v_i- \rri \bar{v}_i) \cdot \partial_s\varphi_{i} + \frac{1}{\e}(\ri\ui \otimes v_i - \rri\bar{\ui} \otimes \bar{v}_i): \nabla\varphi_i \;dxds\\
	& + \frac{1}{\e}\int_{0}^{t} \int_{\T} \sum_{i=1}^{n} \left[ (\mu(\ri) \ui - \mu(\rri)\bar{\ui})\nabla \dive \varphi_i  + (\nabla \mu(\ri)u_i - \nabla \mu(\rri)\bar{\ui}): \nabla \varphi_i\right] \;dxds \\
	& + \frac{1}{2\e} \int_{0}^{t} \int_{\T} \sum_{i=1}^{n} \left[(\nabla \lambda(\ri)\ui - \nabla \lambda(\rri)\bar \ui)\dive \varphi_i + (\lambda(\ri)\ui - \lambda(\rri)\bar{\ui}) \nabla\dive\varphi_i \right]\; dxds \\
	& = 
	 \int_{0}^{t} \int_{\T} \sum_{i=1}^{n} ((\ri v_i)(x,0)- (\rri \bar{v}_i)(x,0))\varphi_i(x,0)\;dx,
	\end{align*}
where $\psi,\phi, \varphi$ are Lipschitz test functions ($\phi,\varphi$ vector-valued) compactly supported in $[0,\infty)$ in time and periodic in space. In the above formulation we choose:
\begin{align*}
& \psi_i = \theta(s) \left( h'(\rri) - \frac{1}{2} \frac{|\bar{m_i}|^2}{\rri} - \frac{1}{2} \frac{|\bar{J_i}|^2}{\rri} \right),\\
& \phi_i = \theta(s) \left(\frac{\bar{m}_i}{\rri} \right) \quad \varphi_i = \theta(s) \left(\frac{\bar{J}_i}{\rri} \right),
\end{align*}
where $\theta$ is defined in $\eqref{theta-diff}$. The previous relations become as $\delta \rightarrow 0$:
	\begin{align}
	&\int_{\T} \sum_{i=1}^{n} \left( h'(\rri) - \frac{1}{2} \frac{|\bar{m_i}|^2}{\rri} - \frac{1}{2} \frac{|\bar{J_i}|^2}{\rri} \right)(\ri- \rri)\Big|_{s=0}^t 
	\nonumber\\
	&- \int_{0}^{t} \int_{\T} \sum_{i=1}^{n} \partial_s\left( h'(\rri) - \frac{1}{2} \frac{|\bar{m_i}|^2}{\rri} - \frac{1}{2} \frac{|\bar{J_i}|^2}{\rri} \right) (\ri- \rri) dxds \nonumber\\
	&- \frac{1}{\e}\int_{0}^{t} \int_{\T} \sum_{i=1}^{n} \nabla_x \left( h'(\rri) - \frac{1}{2} \frac{\bar{m_i}^2}{\rri} - \frac{1}{2} \frac{\bar{J_i}^2}{\rri} \right) (\ri\ui - \rri\bar{\ui}) dxds =0,\label{lp1-diff}
	\end{align}
	\begin{align}
	&\int_{\T} \sum_{i=1}^{n} \frac{\bar{m}_i}{\rri}(\ri \ui- \rri \bar{\ui})\Big|_{s=0}^t - \int_{0}^{t} \int_{\T} \sum_{i=1}^{n} \partial_s\left( \frac{\bar{m}_i}{\rri} \right) (\ri \ui- \rri \bar{\ui}) dxds 
	\nonumber\\
	& - \frac{1}{\e}\int_{0}^{t} \int_{\T} \sum_{i=1}^{n}\left(( \ri \ui \otimes \ui - \rri \bar{\ui} \otimes \bar{\ui}): \nabla \left(\frac{\bar{m}_i}{\rri} \right) +  (p(\ri)-p(\rri))\dive\left(\frac{\bar{m}_i}{\rri} \right)     \right) dxds 
	\nonumber\\
	& - \frac{1}{\e} \int_{0}^{t} \int_{\T} \sum_{i=1}^{n} (\mu(\ri)v_i - \mu(\rri)\bar{v}_i)\nabla \dive \left(\frac{\bar{m}_i}{\rri} \right) \nonumber
	\\
	& \hspace{5cm} + (\nabla \mu(\ri)v_i - \nabla \mu(\rri)\bar{v}_i): \nabla\left(\frac{\bar{m}_i}{\rri} \right) dxds 
	\nonumber\\
	& - \frac{1}{2\e} \int_{0}^{t} \int_{\T} \sum_{i=1}^{n} (\nabla \lambda(\ri)v_i - \nabla \lambda(\rri)\bar
	v_i)\dive \left(\frac{\bar{m}_i}{\rri} \right)  \nonumber
	\\
	&\hspace{5cm} + (\lambda(\ri)v_i - \lambda(\rri)\bar{v}_i) \nabla\dive\left(\frac{\bar{m}_i}{\rri} \right)dxds 
	\nonumber\\
	& + \frac{\nu}{\e}\int_{0}^{t} \int_{\T} \sum_{i,j=1}^{n}2\mu_L(\ri) D(\ui):D(\uui) + \lambda_L(\ri)  \dive \ui \dive \uui  dxds
	\nonumber\\	
	& = -\frac{1}{\e^2}\int_{0}^{t} \int_{\T} \sum_{i,j=1}^{n} b_{i,j}(\ri \rj (\ui-\uj) - \rri \rrj (\uui-\uuj)) \bar{\ui} dxds \nonumber \\
	& - \frac{1}{\e^2}\int_{0}^{t} \int_{\T} \sum_{i=1}^{n} M_i (\ri \ui - \rri \bar{\ui} )\cdot \bar{\ui}\; dxds  -\int_{0}^{t} \int_{\T} \sum_{i=1}^{n} \bar{e_i} \bar{\ui}  dxds,	\label{lp2-diff}
	\end{align}
and
	\begin{align}
		&\int_{\T} \sum_{i=1}^{n} \frac{\bar{J}_i}{\rri}(\ri v_i- \rri \bar{v}_i)\Big|_{s=0}^t - \int_{0}^{t} \int_{\T} \sum_{i=1}^{n} \partial_s\left( \frac{\bar{J}_i}{\rri} \right) (\ri v_i- \rri \bar{v}_i)\; dxds 
		\nonumber\\
		& \ - \frac{1}{\e}\int_{0}^{t} \int_{\T} \sum_{i=1}^{n}\left(( \ri \ui \otimes v_i - \rri \bar{u} \otimes \bar{v}_i): \nabla \left(\frac{\bar{J}_i}{\rri} \right) \right)\;dxds 
		\nonumber\\
		& \ + \frac{1}{\e}\int_{0}^{t} \int_{\T} \sum_{i=1}^{n} (\mu(\ri) \ui - \mu(\rri)\bar{\ui})\nabla \dive \left(\frac{\bar{J}_i}{\rri} \right)  
		\nonumber\\
		&\hspace{5cm} + (\nabla \mu(\ri)u_i - \nabla \mu(\rri)\bar{\ui}): \nabla \left(\frac{\bar{J}_i}{\rri} \right) dxds 
		\nonumber\\
		&\  + \frac{1}{2\e} \int_{0}^{t} \int_{\T} \sum_{i=1}^{n} (\nabla \lambda(\ri)\ui - \nabla \lambda(\rri)\bar \ui)\dive \left(\frac{\bar{J}_i}{\rri} \right) 
		\nonumber\\
		&\hspace{5cm}+ (\lambda(\ri)\ui - \lambda(\rri)\bar{\ui}) \nabla\dive\left(\frac{\bar{J}_i}{\rri} \right) dxds 
		\nonumber\\
		&=0.\label{lp3-diff}
\end{align}
Recalling the definition of the relative entropy $\eqref{ret-diff}$, and using $\eqref{eq:re1-diff}, \eqref{eq:re2-diff}$, $\eqref{lp1-diff}$, $\eqref{lp2-diff}$, $\eqref{lp3-diff}$ we obtain:
\begin{align*}
& \Sigma_{tot}(\hat \rho^{\e}, \hat m^{\e}, \hat J^{\e} | \bar{ \hat \rho}, \bar{ \hat m},\bar{\hat J}) \Big|_{s=0}^t\\
&\ \leq 
- \int_{0}^{t} \int_{\T} \sum_{i=1}^{n} \partial_s\left( h'(\rri) - \frac{1}{2} \frac{|\bar{m_i}|^2}{\rri} - \frac{1}{2} \frac{|\bar{J_i}|^2}{\rri} \right) (\ri- \rri) 
\\
&
\hspace{5cm} + \partial_s\left( \frac{\bar{m}_i}{\rri} \right) (\ri \ui- \rri \bar{\ui}) +  \partial_s\left( \frac{\bar{J}_i}{\rri} \right) (\ri v_i- \rri \bar{v}_i)dxds \\
&\phantom{\leq} - \frac{1}{\e}\int_{0}^{t} \int_{\T} \sum_{i=1}^{n} \nabla_x \left( h'(\rri) - \frac{1}{2} \frac{\bar{m_i}^2}{\rri} - \frac{1}{2} \frac{\bar{J_i}^2}{\rri} \right) (\ri\ui - \rri\bar{\ui}) 
\\
& \hspace{5cm}
+(p(\ri) - p(\rri)) \dive \left(\frac{\bar{m}_i}{\rri} \right)dxds
\\
&\phantom{\leq} - \frac{1}{\e}\int_{0}^{t} \int_{\T} \sum_{i=1}^{n} (\ri \ui \otimes \uj - \rri \bar{\ui} \otimes \bar{\ui}) : \nabla \bar{\ui} + (\ri \ui \otimes v_j - \rri \bar{\ui} \otimes \bar{v}_j): \nabla \bar{v}_idxds
\\
&\phantom{\leq} - \frac{1}{\e}\int_{0}^{t} \int_{\T} \sum_{i=1}^{n} \mu(\ri)[(v_i - \bar{v}_i)\nabla \dive\bar{\ui}- (\ui - \bar{\ui})\nabla \dive \bar{v}_i] 
\\ 
& \hspace{5cm}
+ \nabla\mu(\ri)[\nabla \bar{\ui}(v_i -\bar{v}_i)-\nabla \bar{v}_i (\ui -\bar{\ui})] dxds 
\\
&\phantom{\leq} -\frac{1}{2\e}\int_{0}^{t} \int_{\T} \sum_{i=1}^{n} \nabla \lambda(\ri)[(v_i -\bar{v}_i)\dive\bar{\ui} - (\ui -\bar{\ui})\dive \bar{v}_i]
\\
&  \hspace{5cm} +
\lambda(\ri)[(v_i -\bar{v}_i)\nabla \dive \bar{\ui} - (\ui-\bar{\ui})\nabla \dive \bar{v}_i] dxds 
\\
&\phantom{\leq}+\frac{1}{\e}\int_{0}^{t} \int_{\T} \sum_{i=1}^{n} (\mu(\ri)-\mu(\rri))(\bar{\ui}\nabla \dive \bar{v}_i    -\bar{v}_i\nabla \dive\bar{\ui} )
\\
& \hspace{5cm} + (\nabla\mu(\ri)- \nabla \mu (\rri))[\nabla \bar{v}_i \bar{u}- \nabla \bar{\ui}\bar{v}_i] dxds
\\
&\phantom{\leq}+\frac{1}{2\e}\int_{0}^{t} \int_{\T} \sum_{i=1}^{n} (\lambda(\ri)-\lambda(\rri))[\bar{\ui}\nabla \dive \bar{v}_i - \bar{v}_i\nabla \dive \bar{\ui}] \\
&\hspace{5cm}+ (\nabla \lambda(\ri)- \nabla \lambda(\rri))(\bar{\ui}\dive\bar{v}_i - \bar{v}_i \dive \bar{\ui}) dxds 
\\
&\phantom{\leq}- \frac{2\nu}{ \e} \int_{0}^{t} \int_{\T} \sum_{i=1}^{n}\mu_L(\ri) |D(\ui)|^2 dxds -  \frac{\nu}{ \e} \int_{0}^{t} \int_{\T} \sum_{i=1}^{n}\lambda_L(\ri) |\dive \ui|^2 dxds 
\\
&\phantom{\leq}+ \frac{\nu}{\e}\int_{0}^{t} \int_{\T} \sum_{i,j=1}^{n}2\mu_L(\ri) D(\ui):D(\uui) + \lambda_L(\ri)  \dive \ui \dive \uui dxds
\\
&\phantom{\leq}-\frac{1}{\e^2}\int_{0}^{t} \int_{\T} \sum_{i,j=1}^{n} b_{i,j}(\ri \rj (\ui-\uj) - \rri \rrj (\uui-\uuj)) \bar{\ui} dxds
\\
&\phantom{\leq}-\frac{1}{2 \e^2} \int_{0}^{t} \int_{\T} \sum_{i,j=1}^{n} b_{i,j}\ri \rj |\ui-\uj|^2 dxds + \frac{1}{2 \e^2} \int_{0}^{t} \int_{\T}\sum_{i,j=1}^{n} b_{i,j}\rri \rrj |\uui-\uuj|^2 dxds
\\
&\phantom{\leq}- \frac{1}{\e^2}\int_{0}^{t} \int_{\T} \sum_{i=1}^{n} M_i \ri \ui^2 dxds + \frac{1}{\e^2}\int_{0}^{t} \int_{\T} \sum_{i=1}^{n}  M_i \rri \uui^2 dxds  
\\
& \hspace{5cm}+ \frac{1}{\e^2}\int_{0}^{t} \int_{\T} \sum_{i=1}^{n} M_i (\ri \ui - \rri \uui) \uui dxds 
\\
&\ =: \sum_{i=1}^{12} I_i.
\end{align*}
Let us observe that the $I_6 + I_7$ are equal to zero thanks to the property $\eqref{propeSK}$ of the stress tensors $S_i$, $K_i$. Using the relation $h''(\rri)= p'(\rri)/\rri$ we find that:
\begin{equation*}
	\begin{aligned}
	&- \int_{0}^{t} \int_{\T} \sum_{i=1}^{n} \partial_s h'(\rri)  (\ri- \rri)+ \frac{1}{\e}\nabla_x  h'(\rri)(\ri \ui - \rri \bar{\ui}) dxds\\
	& =
	 \frac{1}{\e}\int_{0}^{t} \int_{\T} \sum_{i=1}^{n} p'(\rri)(\ri-\rri)\dive \bar{\ui} + \frac{\ri}{\rri}\nabla p(\rri)(\bar{\ui}-\ui)dxds.
	\end{aligned}
\end{equation*}
Then we notice that the following quantity:
\begin{align*}
&\int_{0}^{t} \int_{\T} \sum_{i=1}^{n} \left( \partial_s\left( \frac{1}{2}|\bar{\ui}|^2 \right)(\ri- \rri) + \frac{1}{\e}\nabla_x\left( \frac{1}{2}|\bar{\ui}|^2 \right)(\ri \ui- \rri \bar{\ui}) \right)dxds\\
& + 
\int_{0}^{t} \int_{\T} \sum_{i=1}^{n} \left( - \partial_s \bar{u}(\ri \ui- \rri \bar{\ui})- \frac{1}{\e}\nabla_x \bar{\ui}:( \ri \ui \otimes  \ui - \rri \uui \otimes \uui ) \right)dxds
\end{align*}
can be rearranged multiply the momentum equations of the strong solution $\bar{\ui},\bar{v}_i$, by $\ri(\bar{\ui} - \ui)$ and $\ri(\bar{v}_i - v_i)$ respectively. We get:
\begin{align*}
&\int_{0}^{t} \int_{\T} \sum_{i=1}^{n} \left( \partial_s\left( \frac{1}{2}|\bar{\ui}|^2 \right)(\ri- \rri) + \frac{1}{\e}\nabla_x\left( \frac{1}{2}|\bar{\ui}|^2 \right)(\ri \ui- \rri \bar{\ui}) - \partial_s \bar{\ui}(\ri \ui- \rri \bar{\ui})  \right)dxds\\
&\ \  -\int_{0}^{t} \int_{\T} \sum_{i=1}^{n}  \frac{1}{\e}\nabla_x \bar{\ui}:( \ri \ui \otimes  \ui - \rri \uui \otimes \uui )dxds\\
&\ = - \int_{0}^{t} \int_{\T} \sum_{i=1}^{n}  \frac{1}{\e}\ri \nabla_x \bar{\ui} : (\ui -\bar{\ui}) \otimes (\ui -\bar{\ui}) + \frac{1}{\e}\frac{\ri}{\rri} \nabla p(\rri) (\bar{\ui}-\ui) - \frac{1}{\e}\frac{\ri}{\rri} \dive \bar{S}_i (\bar{\ui}-\ui)
\\
&\hspace{5cm} - \frac{\ri}{\rri} \bar{e}(\bar{\ui}-\ui)dxds\\
&-\int_{0}^{t} \int_{\T} \sum_{i=1}^{n} \left(  \frac{1}{\e^2}M_i \ri (\uui-\ui) + \frac{1}{\e^2} b_{i,j} \ri \rrj (\uui - \ui)(\uui -\uuj) \right)\;dxds,
\end{align*}
\newline
and
\begin{equation*}
	\begin{aligned}
		&\int_{0}^{t} \int_{\T} \sum_{i=1}^{n} \left( \partial_s\left( \frac{1}{2}|\bar{v}_i|^2 \right)(\ri- \rri) + \frac{1}{\e}\nabla_x\left( \frac{1}{2}|\bar{v}_i|^2 \right)(\ri \ui- \rri \bar{\ui}) \right)dxds
		\\
		& + \int_{0}^{t} \int_{\T} \sum_{i=1}^{n} \left(- \partial_s \bar{v}_i(\ri \ui- \rri \bar{\ui})- \frac{1}{\e}\nabla_x \bar{v}_i :( \ri \ui \otimes  v_i - \rri \uui \otimes \bar{v}_i ) \right)dxds
		\\
		&= - \int_{0}^{t} \int_{\T} \sum_{i=1}^{n} \left(  \frac{1}{\e}\ri \nabla \bar{v}_i : (\ui -\bar{\ui}) \otimes (v_i -\bar{v}_i) + \frac{1}{\e}\frac{\ri}{\rri} \dive \bar{K}_i (\bar{v_i}- v_i) \right)\;dxds.
	\end{aligned}
\end{equation*}
From the previous calculations the relative entropy becomes:
\begin{align*}
		& \Sigma_{tot}(\hat \rho^{\e}, \hat m^{\e}, \hat J^{\e} | \bar{\hat \rho}, \bar{\hat m},\bar{\hat J}) \Big|_{s=0}^t
		\\
		&\ \leq 
		 - \frac{1}{\e} \int_{0}^{t} \int_{\T} \sum_{i=1}^{n} p(\ri | \rri) \dive \uui dxds 
		\\
		&\phantom{\leq}-  \frac{1}{\e}\int_{0}^{t} \int_{\T} \sum_{i=1}^{n} \left(  \ri \nabla \uui : (\ui -\uui) \otimes (\ui -\bar{\ui}) + \ri \nabla \bar{v}_i : (\ui -\uui) \otimes (v_i -\bar{v}_i) \right) dxds
		\\
		& \phantom{\leq}-  \frac{1}{\e}\int_{0}^{t} \int_{\T} \sum_{i=1}^{n} \mu(\ri)[(v_i - \bar{v}_i)\nabla \dive\uui- (\ui - \uui)\nabla \dive \bar{v}_i] 
		\\
		& \hspace{5cm} + \nabla\mu(\ri)[\nabla \bar{\ui}(v_i -\bar{v}_i)-\nabla \bar{v}_i (\ui -\uui)]dxds 
		\\
		&\phantom{\leq}-\frac{1}{2\e}\int_{0}^{t} \int_{\T} \sum_{i=1}^{n}  \nabla \lambda(\ri)[(v_i -\bar{v}_i)\dive\uui - (\ui -\uui)\dive \bar{v}_i] 
		\\
		&  \hspace{5cm} + \lambda(\ri)[(v_i -\bar{v}_i)\nabla \dive \uui - (\ui-\uui)\nabla \dive \bar{v}_i]       dxds
		\\
		&\phantom{\leq} +  \frac{1}{\e}\int_{0}^{t} \int_{\T} \sum_{i=1}^{n} \frac{\ri}{\rri} \Big ( \mu(\rri)\dive\nabla \bar{v}_i + {}^t \nabla \mu(\rri) {}^t \nabla \bar{v}_i  
		\\
		&  \hspace{5cm} + \frac{1}{2} \nabla \lambda(\rri) \dive \bar{v}_i + \frac{1}{2}\lambda(\rri) \nabla \dive \bar{v}_i \Big ) (\uui- \ui)dxds
		\\
		&\phantom{\leq}-  \frac{1}{\e}\int_{0}^{t} \int_{\T} \sum_{i=1}^{n}\frac{\ri}{\rri}\Big (\mu(\rri)\dive {}^t \nabla \uui + {}^t \nabla \mu(\rri)  \nabla \uui  
		\\
		&  \hspace{5cm} + \frac{1}{2} \nabla \lambda(\rri) \dive \uui + \frac{1}{2} \lambda(\rri) \nabla \dive \uui \Big )(\bar{v}_i- v_i)dxds
		\\
		&\phantom{\leq}-\frac{1}{2 \e^2} \int_{0}^{t} \int_{\T} \sum_{i,j=1}^{n} b_{i,j}\ri \rj |\ui-\uj|^2 dxds + \frac{1}{2 \e^2} \int_{0}^{t} \int_{\T}\sum_{i,j=1}^{n} b_{i,j}\rri \rrj |\uui-\uuj|^2 dxds 
		\\
		& \phantom{\leq}- \frac{1}{\e^2}\int_{0}^{t} \int_{\T} \sum_{i,j=1}^{n} b_{i,j}(\ri \rj (\ui-\uj) - \rri \rrj (\uui-\uuj)) \uui dxds
		\\
		&\hspace{5cm} -\frac{1}{\e^2}\int_{0}^{t} \int_{\T} \sum_{i,j=1}^{n}b_{i,j} \ri \rrj(\uui - \uuj)(\uui -\ui) dxds 
		\\ 
		&\phantom{\leq}- \frac{1}{\e^2}\int_{0}^{t} \int_{\T} \sum_{i=1}^{n} M_i \ri \ui^2 dxds + \frac{1}{\e^2}\int_{0}^{t} \int_{\T} \sum_{i=1}^{n}  M_i \rri \uui^2 dxds 
		\\
		&\hspace{5cm} + \frac{1}{\e^2}\int_{0}^{t} \int_{\T} \sum_{i=1}^{n} M_i (\ri \ui - \rri \uui) \uui dxds
		\\ 
		& \phantom{\leq}- \frac{1}{\e^2}\int_{0}^{t} \int_{\T} \sum_{i=1}^{n} M_i \ri \uui(\uui-\ui) dxds 
		\\
		&\phantom{\leq}+\int_{0}^{t} \int_{\T} \sum_{i=1}^{n}  \frac{\ri}{\rri} \bar{e}(\uui-\ui) dxds
		\\
		&\phantom{\leq}- \frac{2\nu}{ \e} \int_{0}^{t} \int_{\T} \sum_{i=1}^{n}\mu_L(\ri) |D(\ui)|^2 dxds -  \frac{\nu}{ \e} \int_{0}^{t} \int_{\T} \sum_{i=1}^{n}\lambda_L(\ri) |\dive \ui|^2 dxds
		\\
	&\phantom{\leq} + \frac{2\nu}{ \e} \int_{0}^{t} \int_{\T} \sum_{i=1}^{n} \mu_L(\ri) D\ui  : D\bar{\ui} dxds + \frac{\nu}{ \e}\int_{0}^{t} \int_{\T} \sum_{i=1}^{n}\lambda_L(\ri) \dive \ui \dive \bar{\ui} dxds    
	\\
	&=: \sum_{i=1}^{13} I_i,
\end{align*}
where, with a slight abuse of notation, we redefine the right hand side of the above relation using the same simbols.
Let us mention that $\dive {}^t \nabla \bar{u} = \nabla \dive \bar{u}$ while $\dive \nabla \bar{v}_i = \nabla \dive \bar{v}_i$ 
since $\nabla \bar{v}_i$ is symmetric being $\bar{v}_i$ a gradient.
We define $\tilde{I_2}$:
\begin{equation*}
\begin{aligned}
\tilde{I_2}: = & - \frac{1}{\e} \int_{0}^{t} \int_{\T} \sum_{i=1}^{n} \left( \mu(\ri) - \frac{\ri}{\rri} \mu(\rri) \right)(\nabla \dive \bar{\ui}(v_i -\bar{v}_i)- \nabla \dive \bar{v}_i(\ui-\bar{\ui}))\;dxds \\
& - \frac{1}{\e}\int_{0}^{t} \int_{\T} \sum_{i=1}^{n} \left( \nabla \mu(\ri) - \frac{\ri}{\rri} \nabla \mu(\rri) \right) \cdot (\nabla \bar{\ui}(v_i- \bar{v}_i) - \nabla \bar{v}_i(\ui -\bar{\ui}))\;dxds,
\end{aligned}
\end{equation*}
since $v_i= \nabla \mu(\ri) / \ri$ then $\tilde{I_2}$ reads as:
\begin{equation*}
	\begin{aligned}
		\tilde{I_2}  = & - \frac{1}{\e}\int_{0}^{t} \int_{\T} \sum_{i=1}^{n} \ri \left( \frac{\mu(\ri)}{\ri} - \frac{\mu(\rri)}{\rri} \right)(\nabla \dive \bar{\ui}(v_i -\bar{v}_i)- \nabla \dive \bar{v}_i(\ui-\bar{\ui}))\;dxds \\
		& - \frac{1}{\e}\int_{0}^{t} \int_{\T} \sum_{i=1}^{n} \ri \left( v_i -\bar{v}_i \right) \cdot (\nabla \bar{\ui}(v_i- \bar{v}_i) - \nabla \bar{v}_i(\ui -\bar{\ui}))\;dxds.
	\end{aligned}
\end{equation*}
We define
\begin{equation*}
\begin{aligned}
\tilde{I_3} := & - \frac{1}{2\e} \int_{0}^{t} \int_{\T} \sum_{i=1}^{n} \left( \lambda(\ri) - \frac{\ri}{\rri} \lambda(\rri)\right)((v_i -	\bar{v}_i)\nabla \dive \bar{\ui}- (u-\bar{\ui})\nabla \dive \bar{v})\;dxds \\
& - \frac{1}{2\e} \int_{0}^{t} \int_{\T} \sum_{i=1}^{n} \left( \nabla \lambda(\ri) - \frac{\ri }{\rri} \nabla \lambda(\rri) \right)((v-\bar{v}_i)\dive \bar{\ui} - (\ui-\bar{u})\dive \bar{v}_i)dxds,
\end{aligned}
\end{equation*}
since $\lambda(\ri) = 2(\ri \mu'(\ri)- \mu(\ri))$ one has:
\begin{equation*}
	\begin{aligned}
		\tilde{I_3} = & - \frac{1}{2\e} \int_{0}^{t} \int_{\T} \sum_{i=1}^{n}  \ri \left( \frac{\lambda(\ri)}{\ri} - \frac{\lambda(\rri)}{\rri} \right)((v_i -	\bar{v}_i)\nabla \dive \bar{\ui}- (\ui-\bar{\ui})\nabla \dive \bar{v_i})\;dxds \\
		& - \frac{1}{\e}\int_{0}^{t} \int_{\T} \sum_{i=1}^{n} \ri(\mu''(\ri)\nabla \ri - \mu''(\rri)\nabla \rri) ((v_i-\bar{v}_i)\dive \bar{\ui} - (\ui-\bar{\ui})\dive \bar{v}_i)dxds.
	\end{aligned}
\end{equation*}
Therefore $I_3+I_4+I_5+I_6= \tilde{I_2}+\tilde{I_3}$ becomes:
\begin{equation*}
\begin{aligned}
\tilde{I_2}+\tilde{I_3} = & - \frac{1}{\e}\int_{0}^{t} \int_{\T} \sum_{i=1}^{n} \ri(\mu''(\ri)\nabla \ri - \mu''(\rri)\nabla \rri) ((v_i-\bar{v}_i)\dive \bar{\ui} + (\bar{\ui}-\ui)\dive \bar{v}_i)dxds \\
& - \frac{1}{\e}\int_{0}^{t} \int_{\T} \sum_{i=1}^{n} \ri(\mu'(\ri) - \mu'(\rri)) ((v_i-\bar{v}_i) \nabla \dive \bar{\ui} + (\bar{\ui}-\ui) \nabla \dive \bar{v}_i)dxds\\
& - \frac{1}{\e}\int_{0}^{t} \int_{\T} \sum_{i=1}^{n} \ri \left[(v_i-\bar{v}_i) \nabla \bar{\ui} (v_i-\bar{v}_i)- (v_i-\bar{v}_i) \nabla \bar{v}_i (\ui-\bar{\ui}) \right]\;dxds.
\end{aligned}
\end{equation*}
	
Let us manage the remaining terms, starting with the interaction terms $I_7 +I_8$. Following the computations in \cite{HJT} we get:
\begin{align*}
 I_8 &=  \frac{1}{\e^2}\int_{0}^{t} \int_{\T} \sum_{i,j=1}^{n}b_{i,j}(\ri \rj (\ui-\uj) - \rri \rrj (\uui-\uuj)) \uui dxds
 \\
 & \ - \frac{1}{\e^2}\int_{0}^{t} \int_{\T} \sum_{i,j=1}^{n}b_{i,j} \ri \rrj(\uui - \uuj)(\uui -\ui) dxds  
 \\
&  = -\frac{1}{\e^2}\int_{0}^{t} \int_{\T} \sum_{i,j=1}^{n}b_{i,j}\ri \rj (\ui-\uj)(\ui -\uui)dxds 
\\
&\ + \frac{1}{\e^2}\int_{0}^{t} \int_{\T} \sum_{i,j=1}^{n}b_{i,j}\ri \rj(\ui-\uj) \ui dxds 
\\
&-\frac{1}{\e^2}\int_{0}^{t} \int_{\T} \sum_{i,j=1}^{n}b_{i,j} \rri \rrj (\uui -\uuj) \uui dxds
\\
&\ - \frac{1}{\e^2}\int_{0}^{t} \int_{\T} \sum_{i,j=1}^{n}b_{i,j} \ri \rrj(\uui - \uuj)(\uui -\ui) dxds
\\
&  = -\frac{1}{\e^2}\int_{0}^{t} \int_{\T} \sum_{i,j=1}^{n}b_{i,j}\ri \rj (\ui-\uj)(\ui -\uui)dxds
\\
&\ + \frac{1}{2 \e^2}\int_{0}^{t} \int_{\T} \sum_{i,j=1}^{n}b_{i,j}\ri \rj|\ui-\uj|^2dxds 
\\
& \ -\frac{1}{2 \e^2}\int_{0}^{t} \int_{\T} \sum_{i,j=1}^{n}b_{i,j} \rri \rrj |\uui -\uuj|^2 dxds 
\\
& \ - \frac{1}{\e^2}\int_{0}^{t} \int_{\T} \sum_{i,j=1}^{n}b_{i,j}\ri(\rj + \rrj - \rj) (\uui-\uuj)(\uui -\ui) dxds
\\
& = \frac{1}{2 \e^2}\int_{0}^{t} \int_{\T} \sum_{i,j=1}^{n}b_{i,j}\ri \rj|\ui-\uj|^2dxds - \frac{1}{2 \e^2}\int_{0}^{t} \int_{\T} \sum_{i,j=1}^{n}b_{i,j} \rri \rrj |\uui -\uuj|^2 dxds 
\\
& \  -\frac{1}{\e^2}\int_{0}^{t} \int_{\T} \sum_{i,j=1}^{n}b_{i,j}\ri(\rrj-\rj)(\ui - \uui)(\uui-\uuj)dxds
\\
& \ - \frac{1}{\e}\int_{0}^{t} \int_{\T} \sum_{i,j=1}^{n}b_{i,j} \ri \rj(\ui-\uui)[(\ui-\uj)-(\uui-\uuj)]dxds
\\
& =  \frac{1}{2 \e^2}\int_{0}^{t} \int_{\T} \sum_{i,j=1}^{n}b_{i,j}\ri \rj|\ui-\uj|^2dxds - \frac{1}{2 \e^2}\int_{0}^{t} \int_{\T} \sum_{i,j=1}^{n}b_{i,j} \rri \rrj |\uui -\uuj|^2 dxds
\\
&  -\frac{1}{\e^2}\int_{0}^{t} \int_{\T} \sum_{i,j=1}^{n}b_{i,j}\ri(\rrj- \rj)(\uui-\ui)(\uui-\uuj)dxds
\\
& \  -  \frac{1}{2 \e^2}\int_{0}^{t} \int_{\T} \sum_{i,j=1}^{n}b_{i,j} \ri \rj |(\ui-\uj)-(\uui-\uuj)|^2 dxds.
\end{align*}
Therefore $I_7 + I_8$ reduces to:
\begin{align*} &-\frac{1}{\e^2}\int_{0}^{t} \int_{\T} \sum_{i,j=1}^{n}b_{i,j}\ri(\rj -\rrj)(\uui-\ui)(\uui-\uuj)\;dxds \\
&-  \frac{1}{2 \e^2}\int_{0}^{t} \int_{\T} \sum_{i,j=1}^{n}b_{i,j} \ri \rj |(\ui-\uj)-(\uui-\uuj)|^2 \;dxds.
\end{align*}
The friction terms and the error part $I_9 + I_{10}+I_{11}$ becomes:
\begin{equation*}
I_9 + I_{10}+I_{11} = -\frac{1}{\e^2}\int_{0}^{t} \int_{\T} \sum_{i=1}^{n} M_i \ri |\ui-\uui|^2 \;dxds + \int_{0}^{t} \int_{\T} \sum_{i=1}^{n}  \frac{\ri}{\rri} \bar{e}(\bar{u}_i-\ui) \;dxds.
\end{equation*}
Finally the viscosity parts $I_{12}+I_{13}$ rewrite as:
\begin{align*}
     &  	- \frac{2\nu}{ \e} \int_{0}^{t} \int_{\T} \sum_{i=1}^{n}\mu_L(\ri) |D(\ui)|^2 dxds -  \frac{\nu}{ \e} \int_{0}^{t} \int_{\T} \sum_{i=1}^{n}\lambda_L(\ri) |\dive \ui|^2 dxds \\
	&+ \frac{2\nu}{ \e} \int_{0}^{t} \int_{\T} \sum_{i=1}^{n} \mu_L(\ri) D\ui  : D\bar{u}_i dxds + \frac{\nu}{ \e} \int_{0}^{t} \int_{\T} \sum_{i=1}^{n}\lambda_L(\ri) \dive \ui \dive \bar{u}_i dxds \\
	& = 
	 - \frac{2\nu}{ \e} \int_{0}^{t} \int_{\T} \sum_{i=1}^{n}\mu_L(\ri) |D(\ui - \bar{u}_i)|^2 dxds - \frac{\nu}{ \e} \int_{0}^{t} \int_{\T} \sum_{i=1}^{n}\lambda_L(\ri)|\dive (\ui -\bar{u}_i)|^2 dxds \\
	& \phantom{=}- \frac{2\nu}{ \e} \int_{0}^{t} \int_{\T} \sum_{i=1}^{n}\mu_L(\ri) D(\bar{u}_i): D(\ui -\bar{u}_i) dxds 
	\\
	&\phantom{=} \ - \frac{\nu}{ \e} \int_{0}^{t} \int_{\T} \sum_{i=1}^{n}\lambda_L(\ri) \dive \bar{u}_i(\dive \ui - \dive \bar{u}_i)dxds
\end{align*}
and we end up with $\eqref{eq:relent2}$:
\begin{align*}
	 &\Sigma_{tot}(\hat \rho^{\e}, \hat m^{\e}, \hat J^{\e} | \bar{\hat \rho}, \bar{\hat m},\bar{\hat J}) \Big|_{s=0}^t + \frac{1}{2 \e^2}\int_{0}^{t} \int_{\T} \sum_{i,j=1}^{n}b_{i,j} \ri \rj |(\ui-\uj)-(\uui-\uuj)|^2dxds \\
	&+ \frac{1}{\e^2}\int_{0}^{t} \int_{\T} \sum_{i=1}^{n} M_i \ri |\ui-\uui|^2dxds    + \frac{2\nu}{ \e} \int_{0}^{t} \int_{\T} \sum_{i=1}^{n}\mu_L(\ri) |D(\ui - \bar{u}_i)|^2 dxds
	\\
	& + \frac{\nu}{ \e} \int_{0}^{t} \int_{\T} \sum_{i=1}^{n}\lambda_L(\ri)|\dive (\ui -\bar{u}_i)|^2 dxds
	\\
	& \leq 
	 - \frac{1}{\e}\int_{0}^{t} \int_{\T} \sum_{i=1}^{n}p(\ri | \rri) \dive \uui dxds
	 \\
	& \phantom{\leq}  - \frac{1}{\e}\int_{0}^{t} \int_{\T} \sum_{i=1}^{n} \left(  \ri \nabla \uui : (\ui -\uui) \otimes (\ui -\uui) \right) dxds
	\\
	& \phantom{\leq} - \frac{1}{\e}\int_{0}^{t} \int_{\T} \sum_{i=1}^{n} \ri (v_i-\bar{v}_i) \nabla \uui (v_i-\bar{v}_i) dxds
	\\
	&\phantom{\leq}  - \frac{1}{\e}\int_{0}^{t} \int_{\T} \sum_{i=1}^{n} \ri(\mu''(\ri)\nabla \ri - \mu''(\rri)\nabla \rri) ((v-\bar{v}_i)\dive \uui - (\ui-\uui)\dive \bar{v}_i)dxds 
	\\
	&\phantom{\leq} - \frac{1}{\e}\int_{0}^{t} \int_{\T} \sum_{i=1}^{n} \ri(\mu'(\ri) - \mu'(\rri)) ((v-\bar{v}_i) \nabla \dive \uui - (\ui-\uui) \nabla \dive \bar{v}_i)dxds
	\\
	&\phantom{\leq} -\frac{1}{\e^2}\int_{0}^{t} \int_{\T} \sum_{i,j=1}^{n}b_{i,j}\ri(\rrj - \rj)(\uui-\ui)(\uui-\uuj)dxds 
	\\
	&\phantom{\leq} +\int_{0}^{t} \int_{\T} \sum_{i=1}^{n}  \frac{\ri}{\rri} \bar{e}(\bar{u}_i-\ui) dxds
	\\
	&\phantom{\leq} - \frac{2\nu}{ \e} \int_{0}^{t} \int_{\T} \sum_{i=1}^{n}\mu_L(\ri) D(\bar{u}_i): D(\ui -\bar{u}_i) dxds 
	\\
	&\phantom{\leq} - \frac{\nu}{ \e} \int_{0}^{t} \int_{\T} \sum_{i=1}^{n}\lambda_L(\ri) \dive \bar{u}_i(\dive \ui - \dive \bar{u}_i)dxds.
\end{align*}
\end{proof}

\subsection{Stability result and convergence of the diffusive limit}
In this part we prove the stability result using the estimate contained in Proposition \ref{prop:2}. Let us present the  ``distance'' used to measure the difference between a weak dissipative solution of $\eqref{eq:sin2}$ and a strong solution of $\eqref{gfeq}$
\begin{equation}\label{defpsi}
\Psi(t):= \int_{\T} \sum_{i=1}^{n}\left( \frac{1}{2}\rho_i \left|\frac{m_i}{\ri} - \frac{\bar{m}_i}{\rri} \right|^2 +  \frac{1}{2} \rho_i \left|\frac{J_i}{\ri} - \frac{\bar{J}_i}{\rri} \right|^2 + h(\ri |  \rri) \right)dx.
\end{equation}
The proof presented here follows the ones contained in \cite{CL,OR}, in particular generalizing the results of the latter by including in the model  under investigations also capillarity and viscosity terms. In addition, we shall also take advantage the enlarged formulation $\eqref{eq:nskv}$ in terms of the drift velocity \cite{BL} to be able to handle more general  Korteweg tensors with respect to the ones considered in \cite{HJT}.  We stress that, under the hypothesis of $\gamma$-law pressure, with $\gamma >1$, we have 
\begin{equation*}
h(\ri) = \frac{\ri^{\gamma}}{\gamma - 1},
\end{equation*}
and therefore
\begin{equation*}
p(\ri | \rri) = (\gamma - 1)h(\ri | \rri),
\end{equation*}
this means that the term in $\eqref{eq:relent2}$ involving the pressure can be controlled in terms of the relative entropy. 
To control the lines involving the Korteweg tensor in $\eqref{eq:relent}$ we make use of the results contained in \cite{BL}, we state below for the sake of completeness.
\begin{lemma}\cite[Lemma 35]{BL}\label{lemma8}
	Let assume $\mu(\rho)= \rho^{\frac{s+3}{2}}$ with $\gamma \geq s+2$ and $s \geq -1$. We have $$\rho |\mu'(\rho)-\mu'(\bar{\rho})|^2 \leq C(\bar{\rho})h(\rho|\bar{\rho}),$$ with $C(\bar{\rho})$ uniformly bounded for $\bar{\rho}$ belonging to compact sets in $\mathbb{R}^+ \times \mathbb{T}^n$.
\end{lemma}
Finally,  to take full advantage of the diagonal relaxation term in \eqref{eq:relent2}, as already done in \cite{OR}, we strengthen assumption A2) as follows: 
\begin{itemize}
    \item[A2b)] Let $\{b_{i,j}\}_{i,j=1}^{n}$ be such that
    hypothesis A1) holds and assume that for any index $i~\in~\{1, \cdots, n\}$ we have   $M_i\rho_i>0$.
\end{itemize}
In view of A2b), in the proof of the main result below,  we  shall use $\widehat{M}~:=~\min_{i\in\{1,\dots,n\}} M_i>0$.
\begin{theorem}\label{theo: stab2}
Let $T>0$ be fixed and let $(\hat \rho^{\e}, \hat m^{\e}, \hat J^{\e} )$ be as in Definition \ref{def:ws-diff} and $(\bar{\hat \rho},\bar{\hat m},\bar{\hat J})$   be a smooth solution of $\eqref{gfeq}$. Assume the pressure $p(\ri)$ is given by the $\gamma$--law $\ri^{\gamma}$ such that $\gamma > 1$. Let $\mu(\ri)= \rho^{\frac{s+3}{2}}$ with $\gamma \geq s+2$ and $s \geq -1$, and assume
\begin{equation*}
    ||\lambda_L(\ri)||_{L^{\infty}((0,t); L^1(\T))}, ||\mu_L(\ri)||_{L^{\infty}((0,t); L^1(\T))} \leq \tilde{E}
\end{equation*}
for a positive constant $\tilde{E}$ independent on $\e$. Let us also assume  that $\ri^\e \in L^{\infty}([0,T]; L^{\infty}(\T))$ and there exist two constants $k, N$ such that:
\begin{equation}\label{boundrho}
    0< k \leq \ri^\e \leq N \text{ in } \R^3, \; 0<t<T.
\end{equation} Finally, assume condition A2b) holds. Then,  for any $t \in [0,T]$,  the stability estimate
\begin{equation}\label{stab2}
	\Psi(t) \leq (\Psi(0) + C\nu\e + C \e^4)\exp^{\hat C t},
\end{equation}
holds true, where $\hat C>0$ depends on $\widehat{M}$ and $C$ is a positive constant depending on $T$, $K$, $k$, $N$, $\widehat{M}$, the $L^1$ bound (\eqref{A1-diff}) for $\ri^{\e}$, assumed to be uniform in $\epsilon$, $\rri$  and its derivatives. Moreover, if $\Psi(0) \rightarrow 0$ as $\epsilon \rightarrow 0$, then as $\epsilon \rightarrow 0$
\begin{equation*}
	\sup_{t \in [0,T]} \Psi(t) \rightarrow 0.
\end{equation*}
\end{theorem}
\begin{proof}
In the following,    we shall estimate  the right--hand--side of \eqref{eq:relent2}
\begin{align*}
 \sum_{l=1}^8 \mathcal{I}_l &:=  - \frac{1}{\e}\int_{0}^{t} \int_{\T} \sum_{i=1}^{n}p(\ri | \rri) \dive \uui dxds 
 \\
		&\  - \frac{1}{\e}\int_{0}^{t} \int_{\T} \sum_{i=1}^{n} \left(  \ri \nabla \uui : (\ui -\uui) \otimes (\ui -\uui) \right) dxds
		\\
		& \ - \frac{1}{\e}\int_{0}^{t} \int_{\T} \sum_{i=1}^{n} \left(  \ri \nabla \uui : (v_i -\bar{v}_i) \otimes (v_i -\bar{v}_i) \right)dxds
		\\
		&\ - \frac{1}{\e}\int_{0}^{t} \int_{\T} \sum_{i=1}^{n} \ri(\mu''(\ri)\nabla \ri - \mu''(\rri)\nabla \rri) \cdot ((v-\bar{v}_i)\dive \uui - (\ui-\uui)\dive \bar{v}_i)dxds
		\\
		&\ - \frac{1}{\e}\int_{0}^{t} \int_{\T} \sum_{i=1}^{n} \ri(\mu'(\ri) - \mu'(\rri)) ((v-\bar{v}_i)\cdot \nabla \dive \uui - (\ui-\uui) \cdot\nabla \dive \bar{v}_i)dxds
		\\
		&\ -\frac{1}{\e^2}\int_{0}^{t} \int_{\T} \sum_{i,j=1}^{n}b_{i,j}\ri(\rj -\rrj)(\uui-\ui)\cdot(\uui-\uuj)dxds 
		\\
		&\ +\int_{0}^{t} \int_{\T} \sum_{i=1}^{n}  \frac{\ri}{\rri} \bar{e}_i\cdot(\bar{u}_i-\ui) dxds
		\\
		&\ - \frac{2\nu}{ \e} \int_{0}^{t} \int_{\T} \sum_{i=1}^{n}\mu_L(\ri) D(\bar{u}_i): D(\ui -\bar{u}_i) dxds
		\\
		&\hspace{5cm} - \frac{\nu}{ \e} \int_{0}^{t} \int_{\T} \sum_{i=1}^{n}\lambda_L(\ri) \dive \bar{u}_i(\dive \ui - \dive \bar{u}_i)dxds
\end{align*}
term by term,  using the strict positive terms  at the left--hand--side and in terms of the quantity \eqref{defpsi}. Following standard arguments \cite{LT,LT2,CL}, and using in particular that $\bar u _i$ and their derivatives are $O(\e)$, we readily obtain
\begin{equation*}
\begin{aligned}
 |\mathcal{I}_1| &\leq \frac{1}{\e}\int_{0}^{t} \int_{\T} \sum_{i=1}^{n} |	p(\ri | \rri) \dive \bar{\ui}| dxds \leq C \int_{0}^{t} \int_{\T} \sum_{i=1}^{n} h(\ri | \rri) dxds;\\
|\mathcal{I}_2| &\leq \frac{1}{\e}\int_{0}^{t} \int_{\T} \sum_{i=1}^{n}  \left | \ri \nabla \bar{\ui} : (\ui -\bar{\ui}) \otimes (\ui -\bar{\ui}) \right |dxds \leq C  \int_{0}^{t} \int_{\T} \sum_{i=1}^{n}\ri \left| \frac{m_i}{\ri}- \frac{\bar{m}_i}{\rri} \right|^2 dxds;\\
	|\mathcal{I}_3| &\leq \frac{1}{\e}\int_{0}^{t} \int_{\T} \sum_{i=1}^{n}\left | \ri \nabla \bar{\ui} : (v_i-\bar{v}_i) \otimes (v_i-\bar{v}_i) \right | dxds \leq C\int_{0}^{t} \int_{\T} \sum_{i=1}^{n} \ri \left| \frac{J_i}{\ri} - \frac{\bar{J}_i}{\rri} \right|^2 dxds.
	\end{aligned}
\end{equation*}
Then we bound $\mathcal{I}_4$   by means of  Young's inequality as follows:
\begin{align*}
	|\mathcal{I}_4|  &\leq   \frac{1}{\e}\int_{0}^{t} \int_{\T} \sum_{i=1}^{n} \left | \ri(\mu''(\ri)\nabla \ri - \mu''(\rri)\nabla \rri\right | \left | (v-\bar{v}_i)\dive \bar{\ui} -  (\ui-\bar{\ui})\dive \bar{v}_i) \right | dxds \\
&\leq C\int_{0}^{t} \int_{\T} \sum_{i=1}^{n} \ri\left |\mu''(\ri)\nabla \ri - \mu''(\rri)\nabla \rri\right | \left |v_i-\bar{v}_i\right |dxds \\
& \ + \frac{C}{\e}\int_{0}^{t} \int_{\T}  \sum_{i=1}^{n} \ri\left |\mu''(\ri)\nabla \ri - \mu''(\rri)\nabla \rri\right| \left |\ui-\bar{\ui}\right | dxds \\
& \leq \frac{C}{\widehat{M}} \int_{0}^{t} \int_{\T}  \sum_{i=1}^{n}\ri \left|\frac{J_i}{\ri} - \frac{\bar{J}_i}{\rri} \right|^2  dxds + \frac{1}{12\e^2} \int_{0}^{t} \int_{\T}  \sum_{i=1}^{n}M_i  \ri \left|\frac{m_i}{\ri} - \frac{\bar{m}_i}{\rri} \right|^2 dxds,
\end{align*}
being $\mu''(\ri)\nabla \ri - \mu''(\rri)\nabla \rri =\frac{s+1}{2} (v_i-\bar{v}_i)$. 
For $\mathcal{I}_5$ we use Lemma \ref{lemma8} to obtain:
\begin{align*}
|\mathcal{I}_5|& \leq \frac{1}{\e}\int_{0}^{t} \int_{\T}\sum_{i=1}^n \left |\ri(\mu'(\ri) - \mu'(\rri)\right | \left | (v_i-\bar{v}_i) \cdot \nabla \dive \bar{u}_i  -  (\ui-\bar{\ui}) \cdot \nabla \dive \bar{v}_i \right | dxds 
\\
&\leq C
\int_{0}^{t} \int_{\T}\sum_{i=1}^n \left |\ri(\mu'(\ri) - \mu'(\rri)\right | \left | v_i-\bar{v}_i \right | dxds \\
&\ + \frac{C}{\e} 
\int_{0}^{t} \int_{\T}\sum_{i=1}^n \left |\ri(\mu'(\ri) - \mu'(\rri)\right | \left | \ui-\bar{\ui} \right | dxds 
\\
&\leq \frac{C}{\widehat{M}} \int_{0}^{t} \int_{\T}
\sum_{i=1}^n \left ( h(\ri | \rri) + \ri \left| \frac{J_i}{\ri} - \frac{\bar{J}_i}{\rri} \right|^2 \right ) dxds 
\\
&\ + \frac{1}{12\e^2}\int_{0}^{t} \int_{\T} \sum_{i=1}^nM_i \ri \left|\frac{m_i}{\ri} - \frac{\bar{m}_i}{\rri} \right|^2  dxds.
\end{align*}
We estimate the interaction term $\mathcal{I}_6$ in \eqref{eq:relent2} when $b_{i,j}>0$ as in \cite{OR}:
\begin{align*}
	|\mathcal{I}_6| & \leq \frac{1}{\e^2}\int_{0}^{t} \int_{\T} \sum_{i,j=1}^{n} \left | b_{i,j}\ri (\rrj -\rj)(\ui- \uui)\right | \left |\uui-\uuj\right |dxds \\
	&\leq  
	 \frac{C}{\e}\int_{0}^{t} \int_{\T} \sum_{i,j=1}^{n} \left |b_{i,j} \ri ( \rrj-\rj)(\ui- \uui) \right |dxds \\
	&\leq  \frac{C}{\widehat{M}} \int_{0}^{t} \int_{\T} \sum_{i,j=1}^{n} |\rj - \rrj|^2dxds + \frac{1}{12 \e^2}\int_{0}^{t} \int_{\T} \sum_{i=1}^{n} M_i \ri |\ui-\uui|^2 dxds 
	\\
	&\leq  C \int_{0}^{t} \int_{\T} \sum_{i=1}^{n} h(\ri | \rri) dxds + \frac{1}{12 \e^2}\int_{0}^{t} \int_{\T} \sum_{i=1}^{n} M_i \ri |\ui-\uui|^2 dxds.
\end{align*}
In the first line we use again $ \uui-\uuj = O(\e)$ and in  the second  \eqref{boundrho} and Young's inequality, while   the last inequality follows from the strict convexity of $\Psi$, which is again a consequence of \eqref{boundrho}. Indeed, in the framework of  solutions  with bounded density, 
 $|\rho_i - \bar\rho_i|^2$ can be controlled in terms of $h(\ri | \rri)$, as shown also in \cite{LT,GLT}. 
For $\mathcal{I}_7$, 
since $\bar{e}_i= O(\e)$,
using 
Young's inequality we get:
\begin{align*}
|\mathcal{I}_7| & \leq \int_{0}^{t} \int_{\T} \sum_{i=1}^{n} \frac{\ri}{\rri} |\bar{e}_i| |\ui - \bar{\ui}| dxds \\
&\leq \frac{C\e^2}{ \widehat{M}} \int_{0}^{t} \int_{\T} \sum_{i=1}^{n} \frac{\ri}{\rri^2} |\bar{e}|^2 dxds + \frac{1}{4 \e^2}\int_{0}^{t} \int_{\T} \sum_{i=1}^{n} M_i \ri|\ui-\uui|^2dxds \\
&\leq C \e^4T + \frac{1}{4 \e^2}\int_{0}^{t} \int_{\T} \sum_{i=1}^{n} M_i \ri|\ui-\uui|^2dxds.
\end{align*}
 Finally, as in \cite{CL}, the viscosity terms in $\mathcal{I}_8$ can be controlled thanks to the uniform bounds of $\mu_L(\ri)$ and $\lambda_L(\ri)$, and again using that $\bar{\ui}$  is $O(\e)$. Indeed:
\begin{align*}
\mathcal{I}_{8} & = \mathcal{I}_{8,1} + \mathcal{I}_{8,2} := - \frac{2\nu}{ \e} \int_{0}^{t} \int_{\T} \sum_{i=1}^{n}\mu_L(\ri) D(\bar{u}_i): D(\ui -\bar{u}_i) dxds  \\
& \ - \frac{\nu}{ \e} \int_{0}^{t} \int_{\T} \sum_{i=1}^{n}\lambda_L(\ri) \dive \bar{u}_i(\dive \ui - \dive \bar{u}_i)dxds
\end{align*}
and 
\begin{align*}
    |\mathcal{I}_{8,1}| &\leq \frac{\nu}{\e} \int_0^t \int_{\T} \sum_{i=1}^{n} \mu_L(\ri) |D(\ui-\bar{\ui})|^2 \;dxds + \nu C T\e;
\\
    |\mathcal{I}_{8,2}| &\leq \frac{\nu}{2 \e} \int_0^t \int_{\T} \sum_{i=1}^{n} \lambda_L(\ri) |\dive(\ui-\bar{\ui})|^2 \;dxds + \nu C T\e.
\end{align*}

Therefore, collecting all estimates above,  we obtain:
\begin{align*}
		& \Psi(t) + \frac{1}{2 \e^2}\int_{0}^{t} \int_{\T} \sum_{i=1}^{n} M_i \ri|\ui-\uui|^2dxds 
		\\
		&\ \ + \frac{1}{2 \e^2}\int_{0}^{t} \int_{\T} \sum_{i,j=1}^{n}b_{i,j} \ri \rj |(\ui-\uj)-(\uui-\uuj)|^2dxds 
		\\
		& \ \   +
		 \frac{\nu}{\e}\int_0^t \int_{\T} \sum_{i=1}^{n} \mu_L(\ri) |D(\ui-\bar{\ui})|^2 dxds + \frac{\nu}{2 \e} \int_0^t \int_{\T} \sum_{i=1}^{n} \lambda_L(\ri) |\dive(\ui-\bar{\ui})|^2 dxds 
		 \\
		& \ \leq \Psi(0) + C \e^4 T + \nu \bar{C} T \e + \hat{C}\int_0^t \Psi(s) ds,
	\end{align*}
and, since from  relation $\eqref{lamee}$   
we have
\begin{align*}
     0 & \leq  \frac{\nu}{2\e}\int_0^t \int_{\T} \sum_{i=1}^{n}   \left( \lambda_L(\ri) + \frac{2}{n} \mu_L(\ri) \right)|\dive(\ui- \bar{\ui})|^2  dxds \\
     & \leq \frac{\nu}{\e}\int_0^t \int_{\T} \sum_{i=1}^{n} \mu_L(\ri) |D(\ui-\bar{\ui})|^2 dxds +
    \frac{\nu}{2 \e} \int_0^t \int_{\T}\sum_{i=1}^{n} \lambda_L(\ri) |\dive(\ui-\bar{\ui})|^2dxds, 
\end{align*}
see also \cite{CL},
the Gronwall Lemma concludes the proof.
\end{proof}
\begin{remark}\label{rem4}
In the above estimate, we keep the dependence on the viscosity constant $\nu$ explicit, showing in this way the final estimate \eqref{stab2} is indeed uniform for $0\leq\nu\ll 1$. In other words, our results applies also for the Euler-Korteweg case, and the corresponding estimate is given by taking the limit $\nu\to0^+$ in \eqref{stab2}, that is   
\begin{equation*}
    \Psi(t) \leq (\Psi(0) + C \e^4)\exp^{\hat C t}
\end{equation*} 
\end{remark}
\section{High friction limit in the case $M_i=0$}\label{sec:3}
This section is devoted to the study of the high--friction limit when no diagonal damping is present, and without diffusive scaling. Hence, referring to the discussions of  Section \ref{sec:1}, the system we are going to study is the following:
\begin{equation}\label{ek:1}
	\left\{ \begin{aligned} 
		&\partial_t\rho_i + \dive(\rho_i u_i ) = 0 \\ 
		&\partial_t(\rho_i u_i ) +\dive (\rho_i u_i \otimes u_i)- 2 \nu  \dive(\mu(\ri) D( \ui)) - \nu \nabla(\lambda(\ri)\dive \ui) + \nabla \rho_i^{\gamma}
		\\
		&\ = \dive S_i  -\frac{1}{\e} \sum_{j=1}^{n}b_{i,j}\rho_i \rho_j (u_i-u_j)\\
		&\partial_t(\rho_i v_i ) +\dive (\rho_i v_i \otimes u_i) + \dive K_i= 0.
	\end{aligned} \right.
\end{equation}
The argument we shall use are closely related to  the ones in \cite{HJT}; however, in the present work, we consider viscosity effects and a more general class of capillarity coefficients, the latter thanks to the augmented formulation already presented in the previous section. Thus, for readers' convenience, we recall the following relations involving in particular the drift velocities $v_i$:
\begin{align*}
	& v_i=\frac{\nabla \mu(\ri)}{\ri}, \ m_i = \rho_i u_i,\  J_i = \rho_i v_i= \nabla \mu(\ri);\\
	& \dive S_i = \rho_i \nabla \left(k(\rho_i) \Delta \ri + \frac{1}{2}k'(\rho_i)|\nabla \rho_i|^2 \right) = \dive(\mu(\ri)\nabla v_i) + \frac{1}{2} \nabla (\lambda(\ri) \dive v_i);\\
	& \dive K_i =  \dive(\mu(\ri)^t\nabla u_i) + \frac{1}{2} \nabla (\lambda(\ri) \dive u_i).
\end{align*}
As already emphasized in Section \ref{sec:1}, in system \eqref{ek:1} we made  the particular choice $\mu_L(\ri)= \mu(\ri)$ and $\lambda_L(\ri)=\lambda(\ri)= 2(\mu'(\ri)\ri - \mu(\ri))$, as in \cite{BL}. This  allows us to rewrite the viscosity terms in the forthcoming relative entropy calculation \eqref{eq:relent} in terms of the quadratic quantity already used \eqref{defpsi}; see subsections below. Indeed,  we recall that from the Hilbert expansion of \eqref{ek:1} performed in Section \ref{sec:1}, in the present high friction limit without time scaling, which does not lead to a diffusive limit,  the viscosity part is not an  $O(\e)$ error term, and therefore it can not be managed in the same way as in Section \ref{sec:4}.

Hence, let us now  compare via a relative entropy inequality weak solutions of $\eqref{ek:1}$ with strong solutions of system \eqref{sys1}, that is:
\begin{equation}\label{eklim}
	\left\{ \begin{aligned} 
		&\partial_t \rri + \dive(\rri \bar{u} ) = 0 \\ 
		&\partial_t(\rri \bar{u} ) +\dive (\rri \bar{u} \otimes \bar{u})- 2 \nu  \dive(\mu(\rri) D(\bar u)) - \nu \nabla(\lambda(\rri)\dive \bar{u})+ \nabla \rri^{\gamma} = \dive \bar{S}_i + \bar{R}_i \\
		&\partial_t(\rri \bar{v}_i ) +\dive (\rri \bar{v}_i \otimes \bar{u}) + \dive \bar{K}_i= 0
	\end{aligned} \right.
\end{equation}
where $\bar{m_i}= \rri \bar{u}$, $\bar{J_i}= \rri \bar{v_i}= \nabla \mu(\rri)$, 
and $\bar{R}_i$ is defined as in \eqref{defri}:
\begin{align*}
\bar{R_i}  &=  \frac{\ri}{\bar{\rho}}\sum_{j=1}^n\left[ \dive{\bar{S_j}} - \nabla  p(\rrj)\right]  - \dive{\bar{S_i}}   + \nabla p(\rri) \\
& \ + \frac{\ri}{\bar{\rho}}\sum_{j=1}^n \left[2\nu \dive{\mu(\rrj)D(\bar{u})}+ \nu \nabla (\lambda(\rrj)\dive{\bar{u}})\right] - 2\nu \dive{\mu(\rri)D(\bar{u})} - \nu \nabla (\lambda(\rri)\dive{\bar{u}})\\
&=  \frac{\ri}{\bar{\rho}} \dive \bar{T} - \dive \bar{T_i} + \nu \dive \bar{D}\frac{\rri}{\bar{\rho}} - \nu \dive \bar{D}_i;
\\
     \dive{\bar{T_i}}&= \dive{\bar{S_i}} - \nabla p(\rri); 
     \\
     \dive{\bar{T}} &= \sum_{i=1}^n \dive{\bar{T_i}}; \\
    \dive{\bar{D_i}}&= \dive{(\mu(\rri)D(\bar{u}))} + \nabla(\lambda(\rri) \dive{\bar{u}});
    \\
    \dive{\bar{D}} &= \sum_{i=1}^n \dive{\bar{D_i}}.
\end{align*}

The definition of weak solutions of system $\eqref{ek:1}$ is the following.
\begin{definition}[weak and dissipative weak solution]\label{def:ws}
A triple $(\hat \rho^{\e}, \hat m^{\e}, \hat J^{\e})$ (where $\hat \rho^\e= (\rho_1^\e, \cdots, \rho_n^\e)$, $\hat m^\e= (m_1^\e, \cdots, m_n^\e)$, $\hat J^\e= (J_1^\e, \cdots, J_n^\e)$) such that
\begin{align*}
	& 0 \leq \ri^{\e} \in  C^0([0,\infty);L^{1}(\T)); \\
	&(\ri^{\e} \ui^{\e},\ri^{\e} v_i^{\e})  \in  C^0([0,\infty);L^{1}(\T; \R^3))\\
\end{align*}
is called a periodic \emph{weak  solution} of $\eqref{ek:1}$ if for any $i=1,\cdots,n$:
\begin{align*}
	&\sqrt{\ri^{\e}} \ui^{\e} \in L^{\infty}((0,T);L^2(\T)^3) \\
	&\sqrt{\ri^{\e}} v_i^{\e} \in L^{\infty}((0,T);L^2(\T)^3) \\
	&\ri^{\e} \in  C^0([0,\infty);L^{\gamma}(\T))\\
	&\mu(\ri^{\e}) D (\ui^\e) \in L^1((0,T);L^1(\T)^{3\times3}) \\
	&\lambda(\ri^\e) \dive \ui^\e \in L^1((0,T);L^1(\T)) \\
	& (\ri^{\e})^2 u_i^{\e} \in L^{\infty}((0,T);L^1(\T)^3) 
\end{align*}
and $(\ri^{\e},\ui^{\e},v_i^{\e})$ satisfy for all
$\psi_i \in C_0^{1}([0,\infty); C^{1}(\R^3))$ and $\phi_i, \varphi_i \in C_0^1([0,\infty); C^1(\T)^3)$:
\begin{equation*}
- \int_{0}^{\infty} \int_{\T} \sum_{i=1}^{n}(\ri^{\e} \partial_t\psi_i + \ri^{\e}\ui^{\e} \cdot \nabla \psi_i) \; dxdt = \int_{\T} \sum_{i=1}^{n} \ri^{\e}(x,0) \psi_i(x,0)\;dx,
\end{equation*}
\begin{align*}
& - \int_{0}^{\infty} \int_{\T} \sum_{i=1}^{n}\ri^{\e} \ui^{\e} \partial_t \phi_i + \ri^{\e}\ui^{\e} \otimes \ui^{\e} : \nabla \phi_i + {\rho^{\e}_i }^{\gamma} \dive \psi_i                 + \mu(\ri^{\e})v_i \cdot \nabla \dive \phi_i 
\\
& \hspace{5cm} + \nabla \mu(\ri^{\e}) \cdot (\nabla \phi_i v_i) dxdt
\\ 
& \   - \int_{0}^{\infty} \int_{\T} \sum_{i=1}^{n} \left( \frac{1}{2} \nabla \lambda(\ri^{\e}) \cdot v_i^{\e} \dive \phi_i + \frac{1}{2} \lambda(\ri^{\e})v_i^{\e} \cdot \nabla \dive \phi_i \right) dxdt
\\
& \ +  \nu \int_{0}^{\infty} \int_{\T} \sum_{i=1}^{n} \left(2\mu(\ri)D(\ui):\nabla \phi_i +   \lambda(\ri) \dive{u_i} \dive \phi_i \right)  dxdt
\\
& \  + \frac{1}{\e}\int_{0}^{\infty} \int_{\T} \sum_{i,j=1}^{n} b_{i,j}\ri^{\e}\rj^{\e}(\ui^{\e}-\uj^{\e})\cdot \phi_i dxdt 
\\
& = \int_{\T} \sum_{i=1}^{n} (\ri^{\e}\ui^{\e})(x,0) \cdot \phi_i(x,0)dx,
\end{align*}
\begin{align*}
	&  - \int_{0}^{\infty} \int_{\T} \sum_{i=1}^{n}(\ri^{\e} v_i^{\e} \partial_t \varphi_i + \ri^{\e}\ui^{\e} \otimes v_i^{\e} : \nabla \varphi_i + \mu(\ri^{\e})v_i \cdot \nabla \dive\varphi_i + \nabla \mu(\ri^{\e}) \cdot (\nabla \varphi_i v_i) ) dxdt 
	\\
	& \  - \int_{0}^{\infty} \int_{\T} \sum_{i=1}^{n} \left( \frac{1}{2} \nabla \lambda(\ri^{\e}) \cdot v_i^{\e} \dive \varphi_i + \frac{1}{2} \lambda(\ri^{\e})v_i^{\e} \cdot \nabla \dive\varphi_i \right) dxdt
	\\
	& = \int_{\T} \sum_{i=1}^{n} (\ri^{\e} v_i^{\e})(x,0) \cdot \varphi_i(x,0)dx.	
\end{align*}

If in addition $\eta_{tot}(\hat \rho^{\e}, \hat m^{\e},\hat J^{\e}) \in C([0,\infty);L^1(\T))$ 
and the integrated energy inequality of \eqref{sigmatot}
\begin{equation}\label{ene1}
\begin{aligned}
& - \int_{0}^{\infty} \Sigma_{tot}(\hat \rho^{\e}(t)     ,\hat m^{\e}(t),\hat J^{\e}(t)) \theta'(t)dt  +  \frac{1}{2 \e}\int_{0}^{\infty} \int_{\T} \sum_{i,j=1}^{n} b_{i,j} \ri^{\e} \rj^{\e}|\ui^{\e}-\uj^{\e}|^2\theta(t)dxdt  
\\
&\  + 2\nu\int_{0}^{\infty} \int_{\T} \sum_{i=1}^n \mu(\ri^\e)| D(\ui^\e)|^2 \; dxdt + \nu \int_{0}^{\infty} \int_{\T} \sum_{i=1}^n \lambda(\ri^\e)| \dive \ui^\e|^2 dxdt 
\\
& \leq \Sigma_{tot}(\hat\rho^{\e}(0),\hat m^{\e}(0),\hat J^{\e}(0))\theta(0)
\end{aligned}	
\end{equation}
holds for any $\theta \in W^{1,\infty}([0,\infty))$ compactly supported in $[0,\infty)$, then we call $(\hat \rho^{\e}, \hat m^{\e}, \hat J^{\e})$ a periodic \emph{dissipative weak solution}.

We say that a weak periodic (dissipative) weak solution $(\hat \rho^{\e}, \hat m^{\e},\hat J^{\e})$ for $\eqref{ek:1}$ has finite total mass and finite total energy if for any $T>0$ there exists a constant $K>0$ indipendent of $\e$ such that:
\begin{equation}\label{A1}
\begin{aligned}
& \sup_{t \in (0,T)} \int_{\T} \sum_{i=1}^{n} \ri^{\e}dx \leq K;
\\
& \sup_{t \in (0,T)} \int_{\T} \sum_{i=1}^{n} \left( \frac{1}{2}\frac{{m_i^{\e}}^2}{\ri^{\e}} + \frac{1}{2} \frac{{J_i^{\e}}^2}{\ri^{\e}} + \frac{{\ri^{\e}}^{\gamma}}{\gamma - 1} \right)dx \leq K.
\end{aligned}
\end{equation}
\end{definition}
We recall the definition of the relative entropy where the subscripts $\e$ have been omitted for simplicity:
\begin{equation}\label{ret}
\begin{aligned}
&\eta_{tot}(\hat \rho,\hat m,\hat J | \bar{\hat \rho}, \bar{\hat m},\bar{\hat J})=  \sum_{i=1}^{n} \eta(\ri,m_i,J_i) - \sum_{i=1}^{n} \eta(\rri, \bar{m}_i,\bar{J}_i) - \sum_{i=1}^{n} \eta_{\ri}(\bar{\ri}, \bar{m_i},\bar{J_i}) (\ri-\rri) \\
& - \sum_{i=1}^{n} \eta_{m_i}(\bar{\ri}, \bar{m_i},\bar{J_i})(m_i-\bar{m}_i) -  \sum_{i=1}^{n} \eta_{J_i}(\bar{\ri}, \bar{m_i},\bar{J_i})(J_i - \bar{J}_i).
\end{aligned}
\end{equation}
\begin{proposition}\label{prop:1}
Let $( \hat \rho^{\e}, \hat m^{\e}, \hat J^{\e})$ be a dissipative weak periodic solution to $\eqref{ek:1}$ as in Definition \ref{def:ws} with finite total mass and energy \eqref{A1} such that A1) holds. Let $\bar{w}_i=(\bar{\ri}, \bar{m}_i=\bar{\ri}\bar{u},\bar{J}_i = \bar{\ri}\bar{v}_i)$ be a smooth solution to $\eqref{eklim}$ such that:
\begin{equation*}
\bar{w}_i, \partial_t\bar{w}_i, \nabla \bar{w}_i, D^2\bar{w}_i, D^3\bar{\rho}_i \in L^{\infty}([0,T];L^{\infty}(\T)), \text{ for any } i=1,\cdots,n
\end{equation*}
then:
\begin{align}\label{eq:relent}
\nonumber		&\Sigma_{tot}(\hat \rho^\e,\hat m^\e,\hat J^\e | \bar{\hat \rho}, \bar{\hat m},\bar{\hat J}) \Big|_{s=0}^t + \frac{1}{2\e} \int_{0}^{t} \int_{\T} \sum_{i,j=1}^{n} b_{i,j} \ri \rj|\ui -\uj|^2  dxds \\ \nonumber
		&\ \ + 2\nu\int_{0}^{t} \int_{\T} \sum_{i=1}^n \mu(\ri) |D(\ui-\bar{u})|^2 dxds
		+ \nu \int_{0}^{t} \int_{\T} \sum_{i=1}^n \lambda(\ri) |\dive{(\ui-\bar{u})}|^2 dxds 
		\\ \nonumber
& \ \leq  2\nu \int_0^t \int_{\T}\sum_{i=1}^n \ri D(\bar{u}): (v_i- \bar{v_i}) \otimes (\ui-\bar{u}) dxds \\ \nonumber
  &\ \ + 2\nu \int_0^t \int_{\T}\sum_{i=1}^n \ri \nabla (\dive \bar{u})( \mu'(\ri) - \mu'(\rri)) (\ui-\bar{u}) dxds \\ \nonumber
  &\ \  +  2\nu \int_0^t \int_{\T}\sum_{i=1}^n \ri( \mu''(\ri) \nabla \ri -  \mu''(\rri)\nabla \rri) \dive \bar{u} (\ui-\bar{u}) dxds
  \\ \nonumber
		& \ \  +\int_{0}^{t} \int_{\T} \sum_{i=1}^{n}  \frac{\ri}{\rri} \bar{R}_i(\bar{u}-\ui)  dxds \\ \nonumber
		&\ \  -\int_{0}^{t} \int_{\T} \sum_{i=1}^{n}p(\ri | \rri) \dive \bar{u} dxds \\ \nonumber
		&\ \  - \int_{0}^{t} \int_{\T} \sum_{i=1}^{n} \left(  \ri \nabla \bar{u} : (\ui -\bar{u}) \otimes (\ui -\bar{u}) \right) dxds\\ \nonumber
		&\ \  - \int_{0}^{t} \int_{\T} \sum_{i=1}^{n} \ri \nabla \bar{u}: (v_i-\bar{v}_i) \otimes  (v_i-\bar{v}_i) dxds\\ \nonumber
		&\ \  - \int_{0}^{t} \int_{\T} \sum_{i=1}^{n} \ri(\mu''(\ri)\nabla \ri - \mu''(\rri)\nabla \rri) \cdot ((v-\bar{v}_i)\dive \bar{u} - (\ui-\bar{u})\dive \bar{v}_i)dxds \\
		&\ \  - \int_{0}^{t} \int_{\T} \sum_{i=1}^{n} \ri(\mu'(\ri) - \mu'(\rri)) ((v-\bar{v}_i) \cdot \nabla \dive \bar{u} - (\ui-\bar{u}) \cdot \nabla \dive \bar{v}_i)dxds.
	\end{align}
\end{proposition}
\begin{proof}
Using the definition of periodic dissipative weak  solution as in Definition \ref{def:ws} and the regularity of strong solution of $\eqref{eklim}$ we compute the first two terms of \eqref{ret}, while the weak formulation allows us to estimate the remaining linear parts.
\begin{flushleft}
\emph{Step I: The Energy Inequalities.}	
\end{flushleft}

As in Section \ref{sec:4}, we use the following test function in $\eqref{ene1}$:
\begin{equation}\label{theta}
\theta(s) = \left\{ \begin{aligned}
 1 \quad & \text{ for } \quad 0 \leq s< t \\
\frac{t-s}{\delta} + 1  \quad & \text{ for } \quad t \leq s < t+ \delta \\
0 \quad & \text{ for } \quad s > t+ \delta.
\end{aligned}\right.
\end{equation}
Passing to the limit as $\delta \rightarrow 0$ we get:
\begin{equation}\label{eq: re1}
\begin{aligned}
&\Sigma_{tot}(\hat \rho^{\e}(t),\hat m^{\e}(t),\hat J^{\e}(t)) + \frac{1}{2 \e} \int_{0}^{t} \int_{\T} \sum_{i,j=1}^n b_{i,j}\ri \rj |\ui-\uj|^2 dxds  \\
&\ + 2\nu\int_{0}^{t} \int_{\T} \sum_{i=1}^n \mu(\ri) |D(\ui)|^2 dxds 
 + \nu \int_{0}^{t} \int_{\T} \sum_{i=1}^n \lambda(\ri) |\dive{\ui}|^2 dxds
\\
& \leq  \Sigma_{tot}(\hat \rho^{\e}(0),\hat m^{\e}(0),\hat J^{\e}(0)).
\end{aligned}
\end{equation}
We compute the energy associated to the strong solution $(\bar{\ri}, \bar{m}_i,\bar{J}_i)$ multiplying equations $\eqref{eklim}_{2}$ and $\eqref{eklim}_{3}$ by $\bar{u}$ and $\bar{v}_i$ respectively. If we sum all the contributes integrating over $(0,T) \times \T$ we obtain:
\begin{equation}\label{eq:re2}
\begin{aligned}
&\Sigma_{tot}(\bar{ \hat \rho}(t),\bar{\hat m}(t),\bar{\hat J}(t))  + 2\nu\int_{0}^{t} \int_{\T} \sum_{i=1}^n \mu(\rri) |D(\bar{u})|^2 dxds 
 + \nu \int_{0}^{t} \int_{\T} \sum_{i=1}^n \lambda(\rri) |\dive{\bar{u}}|^2 dxds 
\\
& \leq  \Sigma_{tot}(\bar{\hat \rho}(0),\bar{\hat m}(0),\bar{\hat J}(0)) + \int_{0}^{t} \int_{\T} \sum_{i=1}^{n} \bar{R}_i \bar{u}dxds.
\end{aligned}
\end{equation}
\begin{flushleft}
	\emph{Step II: Equations for the difference.}
\end{flushleft}

Now we evaluate the linear part of the relative entropy using suitable test functions in the weak formulation according to  Definition \ref{def:ws}. If we take the differences $(\ri- \rri, \rho_i\ui - \bar\rho_i\bar{u}, \rho_i v_i - \bar\rho_i \bar{v}_i)$ they have to satisfy:
\begin{align*}
- \int_{0}^{t} \int_{\T} \sum_{i=1}^{n}(\ri- \rri) \partial_s\psi_{i} + (\ri\ui - \rri\bar{u})\cdot \nabla\psi_i dxds
= \int_{\T} \sum_{i=1}^{n}(\ri(x,0)- \rri(x,0)\psi_i(x,0)) dx,
\end{align*}
\begin{align*}
	&- \int_{0}^{t} \int_{\T} \sum_{i=1}^{n}(\ri \ui- \rri \bar{u}) \cdot \partial_s\phi_{i} + (\ri\ui \otimes \ui - \rri\bar{u} \otimes \bar{u}): \nabla\phi_i + (\ri^{\gamma} - \rri^{\gamma}) \dive \phi_i   dxds \\
	&\ - \int_{0}^{t} \int_{\T} \sum_{i=1}^{n} (\mu(\ri)v_i - \mu(\rri)\bar{v}_i)\nabla \dive \phi_i  + (\nabla \mu(\ri)v_i - \nabla \mu(\rri)\bar{v}_i): \nabla \phi_i dxds
	\\
	&\  - \frac{1}{2} \int_{0}^{t} \int_{\T} \sum_{i=1}^{n} (\nabla \lambda(\ri)v_i - \nabla \lambda(\rri)\bar
	v_i)\dive \phi_i + (\lambda(\ri)v_i - \lambda(\rri)\bar{v}_i) \nabla\dive\phi_i dxds
	\\
	&\ +  \nu \int_{0}^{t} \int_{\T} \sum_{i=1}^{n} \left(2\mu(\ri)D(\ui) - 2\mu(\rri)D(\bar{u})\right) :\nabla \phi_i + \left(  \lambda (\ri) \dive{u_i} - \lambda(\rri) \dive \bar{u} \right)\dive \phi_i dxds
	\\
	&\ +\frac{1}{\e} \int_{0}^{t} \int_{\T} \sum_{i,j=1}^{n} b_{i,j} \ri \rj(\ui -\uj) \phi_i dxds +\int_{0}^{t} \int_{\T} \sum_{i=1}^{n} \bar{R}_i \phi_i  dxds 
	\\
	&=\int_{\T} \sum_{i=1}^{n} ((\ri \ui)(x,0)- (\rri \uui)(x,0)\phi_i(x,0))dx,
	\end{align*}
and
\begin{align*}
	&- \int_{0}^{t} \int_{\T} \sum_{i=1}^{n}(\ri v_i- \rri \bar{v}_i) \cdot \partial_s\varphi_{i} + (\ri\ui \otimes v_i - \rri\bar{u} \otimes \bar{v}_i): \nabla\varphi_i dxds
	\\
	& \ + \int_{0}^{t} \int_{\T} \sum_{i=1}^{n} (\mu(\ri) \ui - \mu(\rri)\bar{u})\nabla \dive \varphi_i  + (\nabla \mu(\ri)u_i - \nabla \mu(\rri)\bar{u}): \nabla \varphi_i dxds
	\\
	& \ + \frac{1}{2} \int_{0}^{t} \int_{\T} \sum_{i=1}^{n} (\nabla \lambda(\ri)\ui - \nabla \lambda(\rri)\bar u)\dive \varphi_i + (\lambda(\ri)\ui - \lambda(\rri)\bar{u}) \nabla\dive\varphi_i  dxds  \\
	&=  \int_{\T} \sum_{i=1}^{n} ((\ri v_i)(x,0)- (\rri \bar{v}_i)(x,0))\phi_i(x,0)dx,
	\end{align*}
where $\psi,\phi, \varphi$ are Lipschitz test functions ($\phi,\varphi$ vector-valued) compactly supported in $[0,\infty)$ in time and periodic in space. In the above formulation we choose:
\begin{align*}
& \psi_i = \theta(s) \left( h'(\rri) - \frac{1}{2} \frac{\bar{m_i}^2}{\rri} - \frac{1}{2} \frac{\bar{J_i}^2}{\rri} \right),
\\
& \phi_i = \theta(s) \left(\frac{\bar{m}_i}{\rri} \right) \quad \varphi_i = \theta(s) \left(\frac{\bar{J}_i}{\rri} \right),
\end{align*}
where $\theta$ is defined in $\eqref{theta}$ and we recall that $\bar{m}_i/\bar{\ri} = \bar{u}$ is indeed independent from $i$.
The previous relations become as $\delta \rightarrow 0$:
\begin{align}\label{lp1}
	& \nonumber \int_{\T} \sum_{i=1}^{n} \left( h'(\rri) - \frac{1}{2} \frac{\bar{m_i}^2}{\rri} - \frac{1}{2} \frac{\bar{J_i}^2}{\rri} \right)(\ri- \rri)\Big|_{s=0}^t \\ \nonumber
	&- \int_{0}^{t} \int_{\T} \sum_{i=1}^{n} \partial_s\left( h'(\rri) - \frac{1}{2} \frac{\bar{m_i}^2}{\rri} - \frac{1}{2} \frac{\bar{J_i}^2}{\rri} \right) (\ri- \rri)\; dxds \\ 
	&- \int_{0}^{t} \int_{\T} \sum_{i=1}^{n} \nabla_x \left( h'(\rri) - \frac{1}{2} \frac{\bar{m_i}^2}{\rri} - \frac{1}{2} \frac{\bar{J_i}^2}{\rri} \right) (\ri\ui - \rri\bar{u}) \;dxds =0,
	\end{align}
\begin{align}\label{lp2}
	& \nonumber \int_{\T} \sum_{i=1}^{n} \frac{\bar{m}_i}{\rri}(\ri \ui- \rri \bar{u})\Big|_{s=0}^t - \int_{0}^{t} \int_{\T} \sum_{i=1}^{n} \partial_s\left( \frac{\bar{m}_i}{\rri} \right) (\ri \ui- \rri \bar{u}) dxds \\ \nonumber
	&\! - \int_{0}^{t} \int_{\T} \sum_{i=1}^{n}\left(( \ri \ui \otimes \ui - \rri \bar{u} \otimes \bar{u}): \nabla \left(\frac{\bar{m}_i}{\rri} \right) +  (p(\ri)-p(\rri))\dive\left(\frac{\bar{m}_i}{\rri} \right)     \right)  dxds \\ \nonumber
	& \!- \int_{0}^{t} \int_{\T} \sum_{i=1}^{n} \left((\mu(\ri)v_i - \mu(\rri)\bar{v}_i)\nabla \dive \left(\frac{\bar{m}_i}{\rri} \right)   + (\nabla \mu(\ri)v_i - \nabla \mu(\rri)\bar{v}_i): \nabla\left(\frac{\bar{m}_i}{\rri} \right) \right) dxds \\\nonumber
	& \! - \frac{1}{2} \int_{0}^{t} \int_{\T} \sum_{i=1}^{n} \left ((\nabla \lambda(\ri)v_i - \nabla \lambda(\rri)\bar
	v_i)\dive \left(\frac{\bar{m}_i}{\rri} \right)  + (\lambda(\ri)v_i - \lambda(\rri)\bar{v}_i) \nabla\dive\left(\frac{\bar{m}_i}{\rri} \right) \right )dxds \\ \nonumber
	&\!	+  \nu \int_{0}^{t} \int_{\T} \sum_{i=1}^{n} \left(2\mu(\ri)D(\ui) - 2\mu(\rri)D(\bar{u})\right) :\nabla \bar{u} + \left(  \lambda (\ri) \dive{u_i} - \lambda(\rri) \dive \bar{u} \right)\dive \bar{u}  dxds\\
	&\!+\frac{1}{\e} \int_{0}^{t} \int_{\T} \sum_{i,j=1}^{n} b_{i,j} \ri \rj(\ui -\uj) \left(\frac{\bar{m}_i}{\rri} \right) dxds +\int_{0}^{t} \int_{\T} \sum_{i=1}^{n} \bar{R}_i \left(\frac{\bar{m}_i}{\rri} \right)  dxds=0,
	\end{align}
and
\begin{align}\label{lp3}
		& \nonumber\int_{\T} \sum_{i=1}^{n} \frac{\bar{J}_i}{\rri}(\ri v_i- \rri \bar{v}_i)\Big|_{s=0}^t - \int_{0}^{t} \int_{\T} \sum_{i=1}^{n} \partial_s\left( \frac{\bar{J}_i}{\rri} \right) (\ri v_i- \rri \bar{v}_i) dxds \\
		& - \int_{0}^{t} \int_{\T} \sum_{i=1}^{n}\left(( \ri \ui \otimes v_i - \rri \bar{u} \otimes \bar{v}_i): \nabla \left(\frac{\bar{J}_i}{\rri} \right) \right)dxds \\\nonumber
		& + \int_{0}^{t} \int_{\T} \sum_{i=1}^{n} \left((\mu(\ri) \ui - \mu(\rri)\bar{u})\nabla \dive \left(\frac{\bar{J}_i}{\rri} \right)  + (\nabla \mu(\ri)u_i - \nabla \mu(\rri)\bar{u}): \nabla \left(\frac{\bar{J}_i}{\rri} \right) \right) dxds \\
		&  + \frac{1}{2} \int_{0}^{t} \int_{\T} \sum_{i=1}^{n} (\nabla \lambda(\ri)\ui - \nabla \lambda(\rri)\bar u)\dive \left(\frac{\bar{J}_i}{\rri} \right) 
		\nonumber\\
		& + \frac{1}{2} \int_{0}^{t} \int_{\T} \sum_{i=1}^{n} (\lambda(\ri)\ui - \lambda(\rri)\bar{u}) \nabla\dive\left(\frac{\bar{J}_i}{\rri} \right) dxds = 0.
\end{align}
We notice that 
\begin{equation*}
    \frac{1}{\e} \int_0^t \int_{\T} \sum_{i,j=1}^n b_{i,j} \ri \rj (\ui-\uj) \cdot \left(\frac{\bar{m_i}}{\rri}\right)  dxds = \frac{1}{\e} \int_0^t \int_{\T} \sum_{i,j=1}^n b_{i,j} \ri \rj (\ui-\uj) \cdot \bar{u}  dxds = 0,
\end{equation*}
thanks to the symmetry of the matrix $\{ b_{i,j} \}_{i,j=1, \cdots,n}$.
Recalling the definition of the relative entropy $\eqref{ret}$, and using $\eqref{eq: re1}, \eqref{eq:re2}$, $\eqref{lp1}$, $\eqref{lp2}$, $\eqref{lp3}$ we obtain:
\begin{align*}
&\Sigma_{tot}(\hat\rho^{\e},\hat m^{\e},\hat J^{\e} | \bar{\hat \rho}, \bar{\hat m},\bar{\hat J}) \Big|_{s=0}^t\;dx + \frac{1}{2\e} \int_{0}^{t} \int_{\T} \sum_{i,j=1}^{n} b_{i,j} \ri \rj|\ui -\uj|^2  dxds 
\\
& \leq 
- 2\nu\int_{0}^{t} \int_{\T} \sum_{i=1}^n \mu(\ri) |D(\ui)|^2 dxds - \nu \int_{0}^{t} \int_{\T} \sum_{i=1}^n \lambda(\ri) |\dive{\ui}|^2dxds  \\
&\phantom{\leq}+ 2\nu\int_{0}^{t} \int_{\T} \sum_{i=1}^n \mu(\rri) |D(\bar{u})|^2 dxds + \nu \int_{0}^{t} \int_{\T} \sum_{i=1}^n \lambda(\rri) |\dive{\bar{u}}|^2 dxds 
\\
&\phantom{\leq}+  \nu \int_{0}^{\infty} \int_{\T} \sum_{i=1}^{n} \left(2\mu(\ri)D(\ui) - 2\mu(\rri)D(\bar{u})\right) :\nabla \bar{u} + \left(  \lambda (\ri) \dive{u_i} - \lambda(\rri) \dive \bar{u} \right)\dive \bar{u}  dxds
\\
&\phantom{\leq} - \int_{0}^{t} \int_{\T} \sum_{i=1}^{n} \partial_s\left( h'(\rri) - \frac{1}{2} \frac{\bar{m_i}^2}{\rri} - \frac{1}{2} \frac{\bar{J_i}^2}{\rri} \right) (\ri- \rri) + \partial_s\left( \frac{\bar{m}_i}{\rri} \right) (\ri \ui- \rri \bar{u}) \\
&\hspace{5cm} +  \partial_s\left( \frac{\bar{J}_i}{\rri} \right) (\ri v_i- \rri \bar{v}_i)dxds 
\\
&\phantom{\leq}- \int_{0}^{t} \int_{\T} \sum_{i=1}^{n} \nabla_x \left( h'(\rri) - \frac{1}{2} \frac{\bar{m_i}^2}{\rri} - \frac{1}{2} \frac{\bar{J_i}^2}{\rri} \right) (\ri\ui - \rri\bar{u}) dxds 
\\
&\hspace{5cm} - \int_{0}^{t} \int_{\T} \sum_{i=1}^{n} (p(\ri) - p(\rri)) \dive \left(\frac{\bar{m}_i}{\rri} \right) dxds
\\
&\phantom{\leq}- \int_{0}^{t} \int_{\T} \sum_{i=1}^{n} (\ri \ui \otimes \uj - \rri \bar{u} \otimes \bar{u}) : \nabla \bar{u} + (\ri \ui \otimes v_j - \rri \bar{u} \otimes \bar{v}_j): \nabla \bar{v}_i dxds
\\
&\phantom{\leq}- \int_{0}^{t} \int_{\T} \sum_{i=1}^{n} \mu(\ri)[(v_i - \bar{v}_i)\nabla \dive\bar{u}- (\ui - \bar{u})\nabla \dive \bar{v}_i] 
\\
&\hspace{5cm} + \nabla\mu(\ri)[\nabla \bar{u}(v_i -\bar{v}_i)-\nabla \bar{v}_i (\ui -\bar{u})] dxds 
\\
&\phantom{\leq}-\frac{1}{2}\int_{0}^{t} \int_{\T} \sum_{i=1}^{n}  \nabla \lambda(\ri)[(v_i -\bar{v}_i)\dive\bar{u} - (\ui -\bar{u})\dive \bar{v}_i]
\\
&\hspace{5cm} + \lambda(\ri)[(v_i -\bar{v}_i)\nabla \dive \bar{u} - (\ui-\bar{u})\nabla \dive \bar{v}_i]       dxds
\\
&\phantom{\leq}+\int_{0}^{t} \int_{\T} \sum_{i=1}^{n} (\mu(\ri)-\mu(\rri))(\bar{u}\nabla \dive \bar{v}_i    -\bar{v}_i\nabla \dive\bar{u} ) 
\\
&\hspace{5cm} + (\nabla\mu(\ri)- \nabla \mu (\rri))[\nabla \bar{v}_i \bar{u}- \nabla \bar{u}\bar{v}_i] dxds
\\
&\phantom{\leq}+\frac{1}{2}\int_{0}^{t} \int_{\T} \sum_{i=1}^{n}( \lambda(\ri)-\lambda(\rri))[\bar{u}\nabla \dive \bar{v}_i - \bar{v}_i\nabla \dive \bar{u}] 
\\
&\hspace{5cm} + (\nabla \lambda(\ri)- \nabla \lambda(\rri))(\bar{u}\dive\bar{v}_i - \bar{v}_i \dive \bar{u}) dxds.
\end{align*}
As in Proposition \ref{prop:2} we observe that the last two lines are equal to zero thanks to the property $\eqref{propeSK}$ of the stress tensors $S_i$, $K_i$. Using the relation $h''(\rri)= p'(\rri)/\rri$ we find that:
\begin{align*}
    &- \int_{0}^{t} \int_{\T} \sum_{i=1}^{n} \partial_s h'(\rri)  (\ri- \rri)+ \nabla_x  h'(\rri)(\ri \ui - \rri \bar{u})dxds=\\
	& \int_{0}^{t} \int_{\T} \sum_{i=1}^{n} p'(\rri)(\ri-\rri)\dive \bar{u} + \frac{\ri}{\rri}\nabla p(\rri)(\bar{u}-\ui)dsdx.
\end{align*}
Then, the following quantity
\begin{align*}
&\int_{0}^{t} \int_{\T} \sum_{i=1}^{n} \partial_s\left( \frac{1}{2}|\bar{u}|^2 \right)(\ri- \rri) + \nabla_x\left( \frac{1}{2}|\bar{u}|^2 \right)(\ri \ui- \rri \bar{u}) dxds
\\
& \ -\int_{0}^{t} \int_{\T} \sum_{i=1}^{n} \partial_s \bar{u}(\ri \ui- \rri \bar{u})+ \nabla \bar{u}:( \ri \ui \otimes  \ui - \rri \uui \otimes \uui ) dxds
\end{align*}
can be rearranged multiply the momentum equations of the strong solution $\bar{u},\bar{v}_i$, by $\ri(\bar{u} - \ui)$ and $\ri(\bar{v}_i - v_i)$ respectively. We get:
\begin{align*}
&\int_{0}^{t} \int_{\T} \sum_{i=1}^{n} \left( \partial_s\left( \frac{1}{2}|\bar{u}|^2 \right)(\ri- \rri) + \nabla_x\left( \frac{1}{2}|\bar{u}|^2 \right)(\ri \ui- \rri \bar{u}) - \partial_s \bar{u}(\ri \ui- \rri \bar{u}) \right)dxds
\\
&-\int_{0}^{t} \int_{\T} \sum_{i=1}^{n}
\nabla \bar{u}:( \ri \ui \otimes  \ui - \rri \uui \otimes \uui ) dxds
\\
&= - \int_{0}^{t} \int_{\T} \sum_{i=1}^{n}  \ri \nabla \bar{u} : (\ui -\bar{u}) \otimes (\ui -\bar{u}) + \frac{\ri}{\rri} \nabla p(\rri) (\bar{u}-\ui) 
\\
&\hspace{5cm} - \frac{\ri}{\rri} \dive \bar{S}_i (\bar{u}-\ui) - \frac{\ri}{\rri} \bar{R}_i(\bar{u}-\ui) dxds 
\\
&\phantom{=} +2\nu\int_{0}^{t} \int_{\T} \sum_{i=1}^{n} \dive(\mu( \rri) D \bar{u})) \frac{\ri}{\rri}(\bar{u}-\ui) dxds 
\\
&\phantom{=}+ \nu \int_{0}^{t} \int_{\T} \sum_{i=1}^{n} \nabla(\lambda(\rri) \dive{\bar{u}}) \frac{\ri}{\rri}(\bar{u}- \ui)dxds
\end{align*}
and
\begin{align*}
		&\int_{0}^{t} \int_{\T} \sum_{i=1}^{n} \left( \partial_s\left( \frac{1}{2}|\bar{v}_i|^2 \right)(\ri- \rri) + \nabla_x\left( \frac{1}{2}|\bar{v}_i|^2 \right)(\ri \ui- \rri \bar{u}) - \partial_s \bar{v}_i(\ri \ui- \rri \bar{u}) \right)dxds
		\\
		& -\int_{0}^{t} \int_{\T} \sum_{i=1}^{n}		 \nabla \bar{v}_i :( \ri \ui \otimes  v_i - \rri \uui \otimes \bar{v}_i ) dxds
		\\
		&= - \int_{0}^{t} \int_{\T} \sum_{i=1}^{n} \left(  \ri \nabla \bar{v}_i : (\ui -\bar{u}) \otimes (v_i -\bar{v}_i) + \frac{\ri}{\rri} \dive \bar{K}_i (\bar{v_i}- v_i) \right)dxds.
	\end{align*}
From the previous calculations the relative entropy verifies:
\begin{align*}
&\Sigma_{tot}(\hat \rho^{\e},\hat m^{\e},\hat J^{\e} | \bar{\hat \rho}, \bar{\hat m},\bar{\hat J}) \Big|_{s=0}^t + \frac{1}{\e} \int_{0}^{t} \int_{\T} \sum_{i,j=1}^{n} b_{i,j} \ri \rj|\ui -\uj|^2  dxds \\
&\leq
- 2\nu\int_{0}^{t} \int_{\T} \sum_{i=1}^n \mu(\ri) |D(\ui)|^2 dxdt - \nu \int_{0}^{t} \int_{\T} \sum_{i=1}^n \lambda(\ri) |\dive{\ui}|^2 dxds  
\\
&\phantom{\leq}+  \nu \int_{0}^{t} \int_{\T} \sum_{i=1}^{n} \left(2\mu(\ri)D(\ui):D(\bar{u})  +  \lambda (\ri) \dive{u_i}\dive \bar{u}\right)  dxds
\\
&\phantom{\leq}+2\nu\int_{0}^{t} \int_{\T} \sum_{i=1}^{n} \dive(\mu( \rri) D \bar{u})) \frac{\ri}{\rri}(\bar{u}-\ui) dxds 
\\
&\hspace{5cm} + \nu \int_{0}^{t} \int_{\T} \sum_{i=1}^{n} \nabla(\lambda(\rri) \dive{\bar{u}}) \frac{\ri}{\rri}(\bar{u}- \ui)dxds
\\
&\phantom{\leq}-\int_{0}^{t} \int_{\T} \sum_{i=1}^{n}p(\ri | \rri) \dive \bar{u} dxds + \int_{0}^{t} \int_{\T} \sum_{i=1}^{n}  \frac{\ri}{\rri} \bar{R}_i(\bar{u}-\ui) dxds
\\
& \phantom{\leq}- \int_{0}^{t} \int_{\T} \sum_{i=1}^{n} \left(  \ri \nabla \bar{u} : (\ui -\bar{u}) \otimes (\ui -\bar{u}) + \ri \nabla \bar{v}_i : (\ui -\bar{u}) \otimes (v_i -\bar{v}_i) \right) dxds
\\
&\phantom{\leq}- \int_{0}^{t} \int_{\T} \sum_{i=1}^{n} \mu(\ri)[(v_i - \bar{v}_i)\nabla \dive\bar{u}- (\ui - \bar{u})\nabla \dive \bar{v}_i] 
\\
&\hspace{5cm} + \nabla\mu(\ri)[\nabla \bar{u}(v_i -\bar{v}_i)-\nabla \bar{v}_i (\ui -\bar{u})] dxds 
\\
&\phantom{\leq}-\frac{1}{2}\int_{0}^{t} \int_{\T} \sum_{i=1}^{n} \nabla \lambda(\ri)[(v_i -\bar{v}_i)\dive\bar{u} - (\ui -\bar{u})\dive \bar{v}_i]
\\
&\hspace{5cm} + \lambda(\ri)[(v_i -\bar{v}_i)\nabla \dive \bar{u} - (\ui-\bar{u})\nabla \dive \bar{v}_i]       dxds 
\\
&\phantom{\leq}+ \int_{0}^{t} \int_{\T} \sum_{i=1}^{n} \frac{\ri}{\rri} \bigg( \mu(\rri)\dive\nabla \bar{v}_i + {}^t \nabla \mu(\rri) {}^t \nabla \bar{v}_i 
\\
&\hspace{5cm} + \frac{1}{2} \nabla \lambda(\rri) \dive \bar{v}_i + \frac{1}{2}\lambda(\rri) \nabla \dive \bar{v}_i  \bigg)(\bar{u}- \ui)dxds
\\
&\phantom{\leq}- \int_{0}^{t} \int_{\T} \sum_{i=1}^{n}\frac{\ri}{\rri} \bigg( \mu(\rri)\dive {}^t \nabla \bar{u} + {}^t \nabla \mu(\rri)  \nabla \bar{u} 
\\
&\hspace{5cm} + \frac{1}{2} \nabla \lambda(\rri) \dive \bar{u} + \frac{1}{2} \lambda(\rri) \nabla \dive \bar{u}  \bigg)(\bar{v}_i- v_i)dxds\\
&=: \sum_{i=1}^9 I_i.
\end{align*}

Concerning the viscous terms $I_1+I_2+I_3$, as already said above, here we should take advantage of the particular choice for Lam\'e constants to rearrange them properly  and then control  them using the relative entropy. To this end
we compute
\begin{align*}
    I_1+I_2+I_3 & = -2\nu\int_{0}^{t} \int_{\T} \sum_{i=1}^n \mu(\ri) |D(\ui-\bar{u})|^2 dxds
     - \nu \int_{0}^{t} \int_{\T} \sum_{i=1}^n \lambda(\ri) |\dive{(\ui-\bar{u})}|^2 dxds 
    \\
&\ - 2\nu \int_{0}^{t} \int_{\T} \sum_{i=1}^{n}\mu(\ri) D(\bar{u}): D(\ui -\bar{u}) dxds 
\\
&\ - \nu \int_{0}^{t} \int_{\T} \sum_{i=1}^{n}\lambda(\ri) \dive \bar{u}(\dive \ui - \dive \bar{u})dxds
\\
&\ +2\nu\int_{0}^{t} \int_{\T} \sum_{i=1}^{n} \dive(\mu( \rri) D \bar{u})) \frac{\ri}{\rri}(\bar{u}-\ui)dxds 
\\
&\ + \nu \int_{0}^{t} \int_{\T} \sum_{i=1}^{n} \nabla(\lambda(\rri) \dive{\bar{u}}) \frac{\ri}{\rri}(\bar{u}- \ui)dxds 
\\
&=: \hat{I}_1+\hat{I}_2+\hat{I}_3.
\end{align*}
Applying the divergence theorem in $\hat{I}_2$ we get
\begin{align*}
   \hat{I}_2 &=    2\nu \int_0^t \int_{\T} \sum_{i=1}^n \ri D(\bar{u}) : \frac{\nabla \mu(\ri)}{\ri} \otimes (\ui-\bar{u}) + \frac{\mu(\ri)}{\ri}\dive (D(\bar{u})) \ri(\ui-\bar{u}) dxds 
   \\
     &\  + \nu \int_0^t \int_{\T}\sum_{i=1}^n \frac{\nabla \lambda(\ri)}{\ri} \dive{\bar{u}} \; \ri(\ui-\bar{u}) + \frac{\lambda(\ri)}{\ri} \nabla( \dive{\bar{u}})\; \ri (\ui-\bar{u})  dxds,
\end{align*}
while $\hat{I}_3$ is equal to:
\begin{align*}
     \hat{I}_3 &=  -2\nu \int_0^t \int_{\T}\sum_{i=1}^n  \ri D(\bar{u}) : \frac{\nabla \mu(\rri)}{\rri} \otimes (\ui- \bar{u}) + \frac{\mu(\rri)}{\rri}\dive(D(\bar{u}))\ri(\ui- \bar{u}) dxds 
     \\
     & \ - \nu \int_0^t \int_{\T}\sum_{i=1}^n \frac{\nabla\lambda(\rri)}{\rri} \dive \bar{u}  \ri (\ui-\bar{u}) + \frac{\lambda(\rri)}{\rri} \nabla \dive \bar{u} \; \ri (\ui-\bar{u}) dxds.
\end{align*}
Using the definition of $v_i$, $\lambda(\ri)= 2(\mu'(\ri)\ri-\mu(\ri))$ and the property $\dive{D(\bar{u})}= \nabla \dive{\bar{u}}$ being a symmetric matrix, we have:  
\begin{align*}
  \hat{I}_2+\hat{I}_3 &=   2\nu \int_0^t \int_{\T}\sum_{i=1}^n \ri D(\bar{u}): (v_i- \bar{v_i}) \otimes (\ui-\bar{u}) dxds 
  \\
  &\ + 2\nu \int_0^t \int_{\T}\sum_{i=1}^n \ri \nabla (\dive \bar{u})( \mu'(\ri) - \mu'(\rri)) (\ui-\bar{u}) dxds 
  \\
  &\  + 2\nu \int_0^t \int_{\T}\sum_{i=1}^n \ri( \mu''(\ri) \nabla \ri -  \mu''(\rri)\nabla \rri) \dive \bar{u} (\ui-\bar{u}) dxds.
\end{align*}
Hence, we conclude that the viscous part $I_1+I_2+I_3$ rewrites as follows:
\begin{align*}
  I_1+I_2+I_3  &=  -2\nu\int_{0}^{t} \int_{\T} \sum_{i=1}^n \mu_L(\ri) |D(\ui-\bar{u})|^2 dxds
  \\
  &\ - \nu \int_{0}^{t} \int_{\T} \sum_{i=1}^n \lambda_L(\ri) |\dive{(\ui-\bar{u})}|^2 dxds 
  \\
   &\ + 2\nu \int_0^t \int_{\T}\sum_{i=1}^n \ri D(\bar{u}): (v_i- \bar{v_i}) \otimes (\ui-\bar{u}) \;dxds \\
  &\ + 2\nu \int_0^t \int_{\T}\sum_{i=1}^n \ri \nabla (\dive \bar{u})( \mu'(\ri) - \mu'(\rri)) (\ui-\bar{u}) dxds 
  \\
  & \ + 2\nu \int_0^t \int_{\T}\sum_{i=1}^n \ri( \mu''(\ri) \nabla \ri -  \mu''(\rri)\nabla \rri) \dive \bar{u} (\ui-\bar{u}) dxds.
\end{align*}

Let us recall that $\dive {}^t \nabla \bar{u} = \nabla \dive \bar{u}$ and $\nabla \dive \bar{v}_i = \dive \nabla \bar{v}_i$ since $\nabla \bar{v}_i$ is symmetric, being $\bar v_i$ a gradient. We define $\tilde{I}_3$ as follows:
\begin{align*}
\tilde{I}_3 &:=  - \int_{0}^{t} \int_{\T} \sum_{i=1}^{n} \left( \mu(\ri) - \frac{\ri}{\rri} \mu(\rri) \right)(\nabla \dive \bar{u}(v_i -\bar{v}_i)- \nabla \dive \bar{v}_i(\ui-\bar{u}))dxds
\\
&\  - \int_{0}^{t} \int_{\T} \sum_{i=1}^{n} \left( \nabla \mu(\ri) - \frac{\ri}{\rri} \nabla \mu(\rri) \right) \cdot (\nabla \bar{u}(v_i- \bar{v}_i) - \nabla \bar{v}_i(\ui -\bar{u}))dxds.
\end{align*}
Since $v_i= \nabla \mu(\ri) / \ri$ then $\tilde{I_3}$ reads as:
\begin{align*}
	\tilde{I}_3 & =  - \int_{0}^{t} \int_{\T} \sum_{i=1}^{n} \ri \left( \frac{\mu(\ri)}{\ri} - \frac{\mu(\rri)}{\rri} \right)(\nabla \dive \bar{u}(v_i -\bar{v}_i)- \nabla \dive \bar{v}_i(\ui-\bar{u}))dxds 
	\\
		& \  - \int_{0}^{t} \int_{\T} \sum_{i=1}^{n} \ri \left( v_i -\bar{v}_i \right) \cdot (\nabla \bar{u}(v_i- \bar{v}_i) - \nabla \bar{v}_i(\ui -\bar{u}))dxds.
	\end{align*}
Moreover, we define
\begin{align*}
\tilde{I}_4 &:=  - \frac{1}{2} \int_{0}^{t} \int_{\T} \sum_{i=1}^{n} \left( \lambda(\ri) - \frac{\ri}{\rri} \lambda(\rri)\right)((v_i -	\bar{v}_i)\nabla \dive \bar{u}- (u-\bar{u})\nabla \dive \bar{v})dxds 
\\
& \  - \frac{1}{2} \int_{0}^{t} \int_{\T} \sum_{i=1}^{n} \left( \nabla \lambda(\ri) - \frac{\ri }{\rri} \nabla \lambda(\rri) \right)((v-\bar{v}_i)\dive \bar{u} - (\ui-\bar{u})\dive \bar{v}_i)dxds.
\end{align*}
Since $\lambda(\ri) = 2(\ri \mu'(\ri)- \mu(\ri))$ one has:
\begin{align*}
		\tilde{I}_4  & =  - \frac{1}{2} \int_{0}^{t} \int_{\T} \sum_{i=1}^{n}  \ri \left( \frac{\lambda(\ri)}{\ri} - \frac{\lambda(\rri)}{\rri} \right)((v_i -	\bar{v}_i)\nabla \dive \bar{u}- (u-\bar{u})\nabla \dive \bar{v})dxds 
		\\
		& \ - \int_{0}^{t} \int_{\T} \sum_{i=1}^{n} \ri(\mu''(\ri)\nabla \ri - \mu''(\rri)\nabla \rri) ((v-\bar{v}_i)\dive \bar{u} - (\ui-\bar{u})\dive \bar{v}_i)dxds.
	\end{align*}
Therefore:
\begin{align*}
& I_6 + I_7+I_8+I_9  =\tilde{I}_3 + \tilde{I}_4 
 \\
 &\ = 
- \int_{0}^{t} \int_{\T} \sum_{i=1}^{n} \ri(\mu''(\ri)\nabla \ri - \mu''(\rri)\nabla \rri) ((v-\bar{v}_i)\dive \bar{u} - (\ui-\bar{u})\dive \bar{v}_i)dxds 
\\
& \ \ - \int_{0}^{t} \int_{\T} \sum_{i=1}^{n} \ri(\mu'(\ri) - \mu'(\rri)) ((v-\bar{v}_i) \nabla \dive \bar{u} - (\ui-\bar{u}) \nabla \dive \bar{v}_i)dxds
\\
& \ \ - \int_{0}^{t} \int_{\T} \sum_{i=1}^{n} \ri \left[(v_i-\bar{v}_i) \nabla \bar{u} (v_i-\bar{v}_i)- (v_i-\bar{v}_i) \nabla \bar{v}_i (\ui-\bar{u}) \right]dxds.
\end{align*}
Finally, the relative entropy verifies \eqref{eq:relent}
and the proof is complete.
\end{proof}
\subsection{Stability result and convergence of the  limit}
As in the previous limit, with the relative entropy estimate $\eqref{eq:relent}$ of Proposition \ref{prop:1} at hand, we are able to control the  relaxation limit using  the  quadratic quantity \eqref{defpsi}:
\begin{align*}
\Psi(t)&= \int_{\T} \sum_{i=1}^{n}\left( \frac{1}{2} \rho_i \left|\frac{m_i}{\ri} - \frac{\bar{m}_i}{\rri} \right|^2 +  \frac{1}{2} \rho_i \left|\frac{J_i}{\ri} - \frac{\bar{J}_i}{\rri} \right|^2 + h(\ri |  \rri) \right)dx 
\\
& = \int_{\T} \sum_{i=1}^{n}\left( \frac{1}{2} \rho_i \left|u_i - \bar u \right|^2 +  \frac{1}{2} \rho_i \left|v_i - \bar v_i \right|^2 + h(\ri |  \rri) \right)dx .
\end{align*}
The proof of our convergence follows the blueprint of \cite{HJT}, again  generalizing the results of the latter by including viscosity terms and considering a more general class of capillarity coefficients. As already mentioned above, this last generalization can be done thanks to the enlarged formulation $\eqref{ek:1}$ in terms of the drift velocity; see \cite{BL} and Lemma \ref{lemma8} in Section \ref{sec:4}.  The crucial estimate in the proof done in  \cite{HJT} under the framework A1) 
is the following control  of the kinetic energy in terms of the interaction energy; see  \cite[ Theorem 9 and Remark 12]{HJT},
\begin{equation}\label{control}
    \frac{1}{2}\sum_{i,j=1}^n b_{i,j} \ri \rj|\ui-\uj|^2 \geq \delta \sum_{i=1}^n \ri^2 |\ui-u|^2
\end{equation}
for  a constant $\delta>0$. 
\begin{theorem}\label{teo:stab1}
Let $T>0$ be fixed and let $(\hat \rho^{\e}, \hat m^{\e}, \hat J^{\e})$ be as in Definition \ref{def:ws} and $(\bar{\hat\rho}, \bar{\hat m}, \bar{ \hat J})$ be a smooth solution of $\eqref{eklim}$. Assume the pressure $p(\ri)$ is given by the $\gamma$--law $\ri^{\gamma}$, $\gamma > 1$.  Let  $\mu(\ri)= \ri^{\frac{s+3}{2}}$ with $\gamma \geq s+2$ and $s \geq -1$, and 
assume that 
 $\ri^\e \in L^{\infty}([0,T]; L^{\infty}(\T))$, namely there exist $0 < k,N$ such that
\begin{equation}\label{cr}
    0< k \leq \ri^\e \leq N \text{ in } \R^3, \; 0<t<T.
\end{equation}
Finally, assume condition A1) holds. Then, for any $t \in [0,T]$,  the stability estimate
\begin{equation}\label{stab1}
	\Psi(t) \leq  ( \Psi(0) + \e C(\delta)(\nu^2+1)) \exp^{C(\nu^2 + \nu +1)t}
\end{equation}
holds true, where $C$ is a positive constant depending on $T$, $K$ the $L^1$ bound \eqref{A1} for $\ri^{\e}$, assumed to be uniform in $\epsilon$, $\rri$,  and its derivatives. Moreover, if $\Psi(0) \rightarrow 0$ as $\epsilon \rightarrow 0$, then as $\epsilon \rightarrow 0$
\begin{equation*}
	\sup_{t \in [0,T]} \Psi(t) \rightarrow 0.
\end{equation*}
\end{theorem}
\begin{proof}
Starting from the relative entropy calculation $\eqref{eq:relent}$ of Proposition \ref{prop:1}, 
and the definition of $\Psi(t)$  we have
\begin{align*}
    & \Psi(t) + \frac{1}{2\e} \int_{0}^{t} \int_{\T} \sum_{i,j=1}^{n} b_{i,j}\ri \rj |\ui-\uj|^2 dxds  + 2\nu\int_{0}^{t} \int_{\T} \sum_{i=1}^n \mu(\ri) |D(\ui-\bar{u})|^2 dxds
    \\
    &\ \ + \nu \int_{0}^{t} \int_{\T} \sum_{i=1}^n \lambda(\ri) |\dive(\ui-\bar{u})|^2 dxds
    \\
    &\ \leq \Psi(0) + \sum_{i=1}^9 \mathcal{J}_i,
\end{align*}
where
\begin{align*}
\sum_{l=1}^9 \mathcal{J}_l &:= 	 2\nu \int_0^t \int_{\T}\sum_{i=1}^n \ri D(\bar{u}): (v_i- \bar{v_i}) \otimes (\ui-\bar{u}) dxds \\
  &\ + 2\nu \int_0^t \int_{\T}\sum_{i=1}^n \ri \nabla (\dive \bar{u})( \mu'(\ri) - \mu'(\rri)) (\ui-\bar{u}) dxds \\
  & \ +  2\nu \int_0^t \int_{\T}\sum_{i=1}^n \ri( \mu''(\ri) \nabla \ri -  \mu''(\rri)\nabla \rri) \dive \bar{u} (\ui-\bar{u}) dxds\\
		&\ +\int_{0}^{t} \int_{\T} \sum_{i=1}^{n}  \frac{\ri}{\rri} \bar{R}_i\cdot (\bar{u}-\ui)  dxds \\
		&\  -\int_{0}^{t} \int_{\T} \sum_{i=1}^{n}p(\ri | \rri) \dive \bar{u} dxds \\
		&\  - \int_{0}^{t} \int_{\T} \sum_{i=1}^{n}   \ri \nabla \bar{u} : (\ui -\bar{u}) \otimes (\ui -\bar{u})  dxds\\
		& \ - \int_{0}^{t} \int_{\T} \sum_{i=1}^{n} \ri \nabla \bar{u}: (v_i-\bar{v}_i) \otimes  (v_i-\bar{v}_i) dxds\\
		&\  - \int_{0}^{t} \int_{\T} \sum_{i=1}^{n} \ri(\mu''(\ri)\nabla \ri - \mu''(\rri)\nabla \rri) \cdot ((v-\bar{v}_i)\dive \bar{u} - (\ui-\bar{u})\dive \bar{v}_i)dxds \\
		&\  - \int_{0}^{t} \int_{\T} \sum_{i=1}^{n} \ri(\mu'(\ri) - \mu'(\rri)) ((v_i -\bar{v}_i) \cdot\nabla \dive \bar{u} - (\ui-\bar{u}) \cdot \nabla \dive \bar{v}_i) dxds.
\end{align*}

Similarly to what we did for Theorem \ref{theo: stab2}, we shall estimate these remainder term by term; the main differences  is related with the different scaling we are using here, which in particular forces us to rewrite the first three terms as shown above, and treat them as quadratic terms controlled by $\Psi$, thanks to the particular choice $\mu_L(\ri)= \mu(\ri)$ and $\lambda_L(\ri)=\lambda(\ri)$.
Specifically,  applying Young's inequality and using Lemma \ref{lemma8}, and  $\mu''(\ri) \nabla \ri = (s+1)/v_i$, we get:
\begin{align*}
    & |\mathcal{J}_1| \leq C\nu \int_0^t \int_{\T}\sum_{i=1}^n \ri |\ui-\bar{u}|^2 dxds + C\nu \int_0^t \int_{\T}\sum_{i=1}^n \ri |v_i-\bar{v_i}|^2dxds, \\
    & |\mathcal{J}_2| \leq C\nu \int_0^t \int_{\T}\sum_{i=1}^n h(\ri |\rri) dxds + C\nu \int_0^t \int_{\T}\sum_{i=1}^n \ri |\ui-\bar{u}|^2dxds,   \\
    & |\mathcal{J}_3| \leq C(s)\nu \int_0^t \int_{\T}\sum_{i=1}^n \ri |v_i - \bar{v_i}|^2 dxds + C\nu \int_0^t \int_{\T}\sum_{i=1}^n \ri |\ui-\bar{u}|^2 dxds, 
\end{align*}
    that is
\begin{align*}
|\mathcal{J}_1|+|\mathcal{J}_2|+|\mathcal{J}_3| \leq C\nu \int_0^t \Psi(s)ds.
\end{align*}
We then split $\mathcal{J}_4$ in four terms as follows:
\begin{align*}
\mathcal{J}_4 & =  \int_{0}^{t} \int_{\T} \sum_{i=1}^{n}\frac{\ri}{\rri} \bar{R}_i(\bar{u}-\ui) \;dxds  \\
& = \int_{0}^{t} \int_{\T} \sum_{i=1}^{n}\frac{\ri}{\rri} \bar{R}_i (\bar{u} - \ui + u - u) \;dxds \\
		& = -  \int_{0}^{t} \int_{\T} \sum_{i=1}^{n}\ri(\ui -u) \cdot \left(\frac{\dive \bar{T}}{\bar{\rho}} - \frac{\dive \bar{T}_i}{\rri} \right)dxds\\
		& \ - \int_{0}^{t} \int_{\T} \sum_{i=1}^{n}\ri(u -\bar{u}) \cdot \left(\frac{\dive \bar{T}}{\bar{\rho}} - \frac{\dive \bar{T}_i}{\rri} \right)dxds\\
		& \ -\nu\int_{0}^{t} \int_{\T} \sum_{i=1}^{n}\ri(\ui -u) \cdot \left(\frac{\dive \bar{D}}{\bar{\rho}} - \frac{\dive \bar{D}_i}{\rri} \right)dxds\\
		& \ - \nu\int_{0}^{t} \int_{\T} \sum_{i=1}^{n}\ri(u -\bar{u}) \cdot \left(\frac{\dive \bar{D}}{\bar{\rho}} - \frac{\dive \bar{D}_i}{\rri} \right)dxds\\
		& =: \mathcal{J}_{4,1}+\mathcal{J}_{4,2}+\mathcal{J}_{4,3}+\mathcal{J}_{4,4}.
\end{align*}
Thanks to Young's inequality and using \eqref{rel1}  we obtain
\begin{align*}
\mathcal{J}_{4,1} &=   \int_{0}^{t} \int_{\T} \sum_{i=1}^n \ri(\ui -u) \frac{\dive \bar{T}_i}{\rri} dxds
\\
&\leq \frac{\delta}{4 \e} \int_{0}^{t} \int_{\T}  \sum_{i=1}^n \ri^2|\ui- u|^2 dxds + \frac{C \e}{\delta}\int_{0}^{t} \int_{\T}  \sum_{i=1}^n \left|\frac{\dive \bar{T}_i}{\rri} \right|^2dxds,
\end{align*}
and therefore \eqref{control} implies
\begin{equation*}
|\mathcal{J}_{4,1}| \leq \frac{1}{8 \e}\int_{0}^{t} \int_{\T} \sum_{i=1}^{n} b_{i,j} \ri \rj |\ui-\uj|^2 \;dxds +  \e C T.
\end{equation*}
For what concerns $\mathcal{J}_{4,2}$ we have:
\begin{align*}
	\mathcal{J}_{4,2} & = - \int_{0}^{t} \int_{\T}  \sum_{i=1}^n \ri(u -\bar{u}) \cdot \left(\frac{\dive \bar{T}}{\bar{\rho}} - \frac{\dive \bar{T}_i}{\rri} \right) \; dxds \\
	& = - \int_{0}^{t} \int_{\T}(u -\bar{u}) \cdot  \left(\frac{\rho}{\bar{\rho}} \sum_{i=1}^{n} \dive \bar{T}_i -\sum_{i=1}^{n}\frac{\ri}{\rri} \dive \bar{T}_i \right)\;dxds \\
	& = - \int_{0}^{t} \int_{\T} \sum_{i=1}^{n} \left(\frac{1}{\bar{\rho}} - \frac{\ri}{\rri \rho} \right) \rho (u - \bar{u}) \cdot \dive \bar{T}_i \;dxds \\
	& \leq \int_{0}^{t} \int_{\T} \rho|u -\bar{u}|^2 \;dxds + C \int_{0}^{t} \int_{\T} \rho \sum_{i=1}^{n} \left|\frac{1}{\bar{\rho}} - \frac{\ri}{\rri \rho} \right|^2 \;dxds.
	\end{align*}
Then we use the uniform bounds \eqref{cr} for $\ri$ to conclude
\begin{equation*}
\rho|u-\bar{u}|^2 = \frac{1}{\rho}\left|\sum_{i=1}^{n} \ri(\ui-\bar{u}) \right|^2 \leq \frac{n}{\rho}\sum_{i=1}^{n} \ri^2|\ui-\bar{u}|^2 \leq \frac{n\;N}{k} \sum_{i=1}^{n} \ri |\ui-\bar{u}|^2,
\end{equation*}
and
\begin{equation*}
\sum_{i=1}^{n} \left|\frac{1}{\bar{\rho}} - \frac{\ri}{\rri \rho} \right|^2 = \sum_{i=1}^{n} \left|\frac{\rho - \bar{\rho}}{\rho \bar{\rho}} + \frac{\rri - \ri}{\rri \rho} \right|^2 \leq C \sum_{i=1}^{n}  |\ri- \rri|^2 \leq C \sum_{i=1}^{n}h(\ri | \rri),
\end{equation*}
where the last inequality follows again form the uniform bound \eqref{cr}; see also  \cite{LT,GLT}. 
Therefore,  there exists $C:=C(n,k,N)$ such that:
\begin{equation*}
\mathcal{J}_{4,2} \leq C\int_{0}^{t} \int_{\T} \sum_{i=1}^{n}\left( \ri|\ui - \bar{u}|^2 + |\ri- \rri|^2 \right) dxds \leq C\int_{0}^{t} \Psi(s) ds.
\end{equation*}
The terms $\mathcal{J}_{4,3}$ and $\mathcal{J}_{4,4}$ can be controlled similarly. Indeed, for $\mathcal{J}_{4,3}$ we obtain
\begin{align*}
      \mathcal{J}_{4,3}  & =  \int_{0}^{t} \int_{\T} \sum_{i=1}^{n}
      \frac{\sqrt{\delta}}{\sqrt{8\e}}\ri( \ui -u ) \frac{\nu \sqrt{8\e}}{\sqrt{\delta}} \frac{\dive \bar{D_i}}{\rri} dxds \\
      &\leq 
      \frac{\delta}{8\e}\int_{0}^{t} \int_{\T} \sum_{i=1}^{n} \ri^2 |\ui-u|^2 dxds + \frac{C \e}{\delta} \nu^2  \int_{0}^{t} \int_{\T} \sum_{i=1}^{n} \left|\frac{\dive \bar{D_i}}{\rri} \right|^2 dxds \\ 
      & \leq 
      \frac{1}{8 \e}\int_{0}^{t} \int_{\T} \sum_{i,j=1}^{n} b_{i,j} \ri \rj|\ui -\uj|^2  \;dxds + C \e \nu^2 T,
     \end{align*}
thanks to \eqref{rel1}, \eqref{control}, and Young's inequality. Concerning $\mathcal{J}_{4,4}$ we get
\begin{align*}
    \mathcal{J}_{4,4} &  =  \nu\int_{0}^{t} \int_{\T} \sum_{i=1}^{n}\ri(u -\bar{u}) \cdot \left(\frac{\dive \bar{D}}{\bar{\rho}} - \frac{\dive \bar{D}_i}{\rri} \right) dxds 
    \\
    & = \nu\int_{0}^{t} \int_{\T} \sum_{i=1}^{n} \left(\frac{1}{\bar{\rho}} - \frac{\rho}{\rri \bar{\rho}}
\right) \rho (u- \bar{u}) \dive{\bar{D_i}}dxds 
\\
&\leq C\int_{0}^{t} \int_{\T} \rho^2|u -\bar{u}|^2 dxds + C \nu^2 \int_{0}^{t} \int_{\T} \sum_{i=1}^{n} \left(\frac{1}{\bar{\rho}} - \frac{\ri}{\rri \rho} \right)^2 dxds 
\\
& \leq C \int_{0}^{t} \int_{\T} \sum_{i=1}^{n} \ri |\ui-\bar{u}|^2 dxds + C \nu^2 \int_{0}^{t} \int_{\T} \sum_{i=1}^{n} h(\ri | \rri) dxds
\\
& \leq  C (\nu^2 + 1)\int_{0}^{t} \Psi(s) ds.
\end{align*}
Summarizing, the term $\mathcal{J}_{4}$ satisfies
\begin{equation*}
    |\mathcal{J}_{4}| \leq  \frac{1}{4\e}\int_{0}^{t} \int_{\T} \sum_{i,j=1}^{n} b_{i,j}\ri \rj |\ui-\uj|^2dxds  + C(\nu^2+1) \int_0^t \psi(s)ds + C\e(\nu^2 + 1) T.
\end{equation*}
The quadratic terms  $\mathcal{J}_5$ $\mathcal{J}_6$, and $\mathcal{J}_7$ are treated in a standard way, as in Theorem \ref{theo: stab2}:
	\begin{align*}
 |\mathcal{J}_{5}| & \leq \int_{0}^{t} \int_{\T}	\sum_{i=1}^{n}|p(\ri | \rri) \dive \bar{u}| \;dxds \leq C \int_{0}^{t} \int_{\T} \sum_{i=1}^{n} h(\ri | \rri) dxds;\\
|\mathcal{J}_{6}| &\leq \int_{0}^{t} \int_{\T}\sum_{i=1}^{n}  \left|  \ri \nabla \bar{u} : (\ui -\bar{u}) \otimes (\ui -\bar{u}) \right| dxds \leq 
C  \int_{0}^{t} \int_{\T} \sum_{i=1}^{n} \ri \left|u_i -\bar u \right|^2 dxds;
\\
	 |\mathcal{J}_{7}| &\leq \int_{0}^{t} \int_{\T} 
	\sum_{i=1}^{n} \left | \ri\nabla \bar{u} : (v_i-\bar{v}_i) \otimes (v_i-\bar{v}_i) \right | dxds
	 \leq 
C\int_{0}^{t} \int_{\T} \sum_{i=1}^{n} \ri \left| v_i- \bar{v}_i\right|^2 dxds.
	\end{align*}
Moreover, concerning $\mathcal{J}_8$, recalling the relation $\mu''(\ri)\nabla \ri - \mu''(\rri)\nabla \rri =\frac{s+1}{2} (v_i-\bar{v}_i)$, we have
\begin{equation*}
	\begin{aligned}
 |\mathcal{J}_8| & \leq   \int_{0}^{t} \int_{\T} 
 \sum_{i=1}^n
 |\ri(\mu''(\ri)\nabla \ri - \mu''(\rri)\nabla \rri)| | ((v-\bar{v}_i)\dive \bar{u} - (\ui-\bar{u})\dive \bar{v}_i)| dxds \\
& \leq C 
\int_{0}^{t} \int_{\T}  \sum_{i=1}^n \ri \left|v_i - \bar{v}_i \right|^2  dxds + C \int_{0}^{t} \int_{\T}
 \sum_{i=1}^n \ri \left| u_i - \bar{u} \right|^2 dxds,
\end{aligned}
\end{equation*}
using again Young's inequality.
Finally, for $\mathcal{J}_9$, we use Lemma \ref{lemma8} as in Section \ref{sec:4} to conclude
\begin{align*}
|\mathcal{J}_9| & \leq \int_{0}^{t} \int_{\T}  \sum_{i=1}^n |\ri(\mu'(\ri) - \mu'(\rri)) ((v-\bar{v}_i)\cdot \nabla \dive \bar{u}| dxds \\
&\ +  \int_{0}^{t} \int_{\T}  \sum_{i=1}^n |\ri(\mu'(\ri) + \mu'(\rri))(\ui-\bar{u})\cdot \nabla \dive \bar{v}_i | dxds  \\
& \leq C \int_{0}^{t} \int_{\T} \sum_{i=1}^n \left ( h(\ri | \rri) + \ri \left| \frac{m_i}{\ri} - \frac{\bar{m}_i}{\rri} \right|^2 + \ri \left| \frac{J_i}{\ri} - \frac{\bar{J}_i}{\rri} \right|^2 \right) dxds.
\end{align*}
Collecting all estimates above, the relative entropy inequality becomes:
\begin{align*}
& \Psi(t)   + \frac{1}{4 \e} \int_{0}^{t} \int_{\T} \sum_{i,j=1}^{n} b_{i,j} \ri \rj |\ui-\uj|^2 dxds + 2\nu \int_{0}^{t} \int_{\T} \sum_{i=1}^n \mu(\ri) |D(\ui-\bar{u})|^2 dxds 
\\
&\ \ + \nu\int_{0}^{t} \int_{\T} \sum_{i=1}^n \lambda(\ri) |\dive(\ui-\bar{u})|^2 dxds
\\
& \ \leq \Psi(0)+ \e C(\delta)(\nu^2+1) + C(\nu^2+ \nu+ 1) \int_{0}^{t} \Psi(s) ds,
\end{align*}
where $C:= C(s,k,\delta,n,N, T)$, and the Gronwall Lemma concludes the proof.
\end{proof}
\begin{remark}\label{rem4-last}
As already pointed out in Section \ref{sec:4}, see in particular Remark \ref{rem4}, it is worth observing that that the stability estimate $\eqref{stab1}$ is consistent with the one obtained    in \cite{HJT} for the inviscid case. Indeed as $\nu \to 0^{+}$ we obtain:
\begin{equation*}
    \Psi(t) \leq  ( \Psi(0) + \e C(\delta)) \exp^{Ct}
\end{equation*} 
which is is exactly the one obtained by the Authors in \cite{HJT} for the Euler-Korteweg case.
\end{remark}
	

\begin{thebibliography}{}

\bibitem{AT22}
Alves, N.J., Tzavaras, A.E. 
The relaxation limit of bipolar fluid models. {\it
Discrete Contin. Dyn. Syst.} {\bf 42} (2022), no. 1, 211-237.

\bibitem{ACCLS21}
 Antonelli, P., Cianfarani Carnevale, G., Lattanzio, C., Spirito, S. 
 Relaxation limit from the quantum Navier--Stokes equations to the quantum drift--diffusion equation. 
 \emph{J. Nonlinear Sci.} {\bf 31} (2021), no.\ 5, Paper No.\ 71, 32 pp.
 
	\bibitem{AM1}   Antonelli, P. and  Marcati, P.
	On the finite energy weak solutions to a system in Quantum Fluid Dynamics, \textit{Comm. Math. Phys.} {\bf 287} (2009), 657-686.
	
	\bibitem{AM2} Antonelli, P. and  Marcati, P.  The Quantum Hydrodynamics system in two space dimensions, \textit{Arch. Ration. Mech. Anal.} {\bf 203} (2012), 499-527.  
	
	\bibitem{AS} Antonelli, P. and Spirito, S. Global existence of finite energy weak solutions of Quantum Navier--Stokes equations, \textit{Arch. Ration. Mech. Anal.} {\bf 225} (2017), 1161-1199. 
	
 	\bibitem{AS2} Antonelli, P. and Spirito, S. 
 	Global existence of weak solutions to the Navier--Stokes--Korteweg equations, \textit{Ann. Inst. H. Poincar\'e Anal. Non Lin\'eaire} 
 	{\bf 39} (2022), no.\ 1, pp.\ 171-200.
	
	\bibitem{Bia19}
	Bianchini, R. Strong convergence of a vector--BGK model to the incompressible Navier--Stokes equations via the relative entropy method. 
	{\it J. Math. Pures Appl. (9)}  {\bf132} (2019), 280-307. 
	
	\bibitem{bou}
	Bouchut, F., A reduced stability condition for non linear relaxation to conservation laws. \textit{J. Hyperbolic Differ. Equ.} {\bf1} (2004), 149-170.
	
	\bibitem{BL}  Bresch, D., Gisclon, M., and Lacroix--Violet, I. On Navier-Stokes-Korteweg and Euler-Korteweg systems: application to quantum fluids models. \textit{ Arch. Ration. Mech. Anal.} {\bf 233} (2019), no. 3, 975-1025.
	
	\bibitem{CP} Caflisch, R. and Papanicolau G.. The fluid-dynamical limit of a nonlinear model Boltzmann equation. \textit{ Commun. Pure Appl. Math.} {\bf32} (1979), 589-616.
	
	\bibitem{Carrillo} 
 	Carrillo, J.A., Peng, Y., and Wr\'{o}blewska-Kami\'{n}ska, A. 
 	Relative entropy method for the relaxation limit of hydrodynamic models, 
 	{\it Netw. Heterog. Media} {\bf 15} (2020), no.\  3, 369-387.
	
	\bibitem{CLL}
	Chen, G.Q., Levermore, C.D., and Liu T.-P. Hyperbolic conservation laws with stiff relaxation terms and entropy. {\it Commun. Pure Appl. Math.} {\bf 47} (1994), 787-830.
	
	\bibitem{CL} Cianfarani Carnevale, G., Lattanzio, C.; High friction limit for Euler-.Korteweg and Navier-Stokes-Korteweg models via relative entropy approach, \emph{J. Differential Equations} {\bf 269} (2020), 10495-10526.
	
	\bibitem{CG}
	Coulombel, J-.F. and Goudon, T. The strong relaxation limit of the multidimensional isothermal Euler equations. \textit{Trans. Amer. Math. Soc.} {\bf359} (2007), 637-648.
	
	\bibitem{DM04} 
	Donatelli, D. and Marcati, P. Convergence of singular limits for multi-D semilinear hyperbolic systems to parabolic systems.
	{\it Trans. Amer. Math. Soc.} {\bf 356} (2004), 2093-2121.
	
	\bibitem{FT19}
	Feireisl, E. and Tang, T. 
	On a singular limit for the stratified compressible Euler system. {\it Asymptot. Anal.} {\bf 114} (2019), no.\ 1-2, 59-72. 
	
	\bibitem{GLT} Giesselmann, J.,  Lattanzio, C., and; Tzavaras, A.E. 
	Relative energy for the Korteweg theory and related Hamiltonian flows in gas dynamics.
	\textit{Arch. Ration. Mech. Anal.} {\bf 223} (2017), no. 3, 1427-1484.
	
	\bibitem{GT}
	 Giesselmann, J. and  Tzavaras, A.E. Stability properties of the Euler-Korteweg system with nonmonotone pressures. \textit{Appl. Anal.} {\bf96} (2017), 1528-1546.
	 
		\bibitem{HJT} 
	Huo, X., J\"ungel, A., and Tzavaras, A.E. 
	High-friction limits of Euler flows for multicomponent systems. {\it Nonlinearity} {\bf32} (2019), no.\ 8, 2875-2913.
	
	\bibitem{JU2}
	J\"ungel, A. and Peng. Y.-J. A hierarchy of hydrodynamics models for plasmas: zero-relaxation-time limits. \textit{Commun. Partial Diff. Eqs.} {\bf24} (1999), 1007-1033.
	
	\bibitem{Lat00}
	 Lattanzio, C. On the 3-D bipolar isentropic Euler-Poisson model for semiconductors and the drift-diffusion limit. 
	 {\it Math. Models Methods Appl. Sci.} {\bf 10} (2000), no. 3, 351-360.
	 
	\bibitem{LT} Lattanzio, C. and  Tzavaras,  A.E. 
	Relative entropy in diffusive relaxation. \textit{ SIAM J. Math. Anal.} {\bf 45} (2013), no. 3, 1563-1584. 
	
	\bibitem{LT2} Lattanzio, C. and  Tzavaras,  A.E.
	From gas dynamics with large friction to gradient flows describing diffusion theories.
	\textit{Comm. Partial Differential Equations} {\bf 42} (2017), no. 2, 261-290.
	
	\bibitem{OR}
	Ostrowski, L. and Rohde, C.
	Compressible multicomponent flow in porous media with Maxwell--Stefan diffusion.
	{\it Math. Methods Appl. Sci.} {\bf 43} (2020), no.\ 7, 4200-4221.
	
		\bibitem{thanos}
	Tzavaras, A. Relative entropy in hyperbolic relaxation. {\it Commun. Math. Sci.} {\bf 3} (2005), 119-132.
	
	\bibitem{WESSKRI}
	Wesslingh. J. and Krishna R. \textit{Mass Transfer in Multicomponent Mixtures.} Delft University Press, Delft, Netherlands, 2000.	
	
	\bibitem{yo}
	Yong, W.-A, Entropy and global existence for hyperbolic balance laws. \textit{Arch. Rational Mech. Anal.} {\bf 172} (2004), 247-266.
	
\end{thebibliography}
\end{document}